\documentclass[12pt]{amsart}
\usepackage{setspace}
\usepackage{amsmath,amsthm,amssymb,amsfonts}
\usepackage{hyperref}
\hypersetup{colorlinks=true}
\usepackage{tikz, float}
\usetikzlibrary{decorations.pathreplacing,angles,quotes}
\numberwithin{equation}{section}
\makeindex
\newcommand{\length}{\mathrm{length}}

\newcommand{\supp}{\mathrm{supp}}

\newcommand\xappa\kappa
\newcommand\yota\iota
\newcounter{consta}

\newcounter{constb}

\newcounter{constc}[section]

\newcommand{\proj}{\mathrm{Pr}}

\newcommand{\R}{\mathbb{R}}
\newcommand{\Rc}{\mathcal{R}}
\newcommand{\Qc}{\mathcal{Q}}

\newcommand{\dist}{\mathrm{dist}}

\newcommand{\norm}[1]{\left\lVert#1\right\rVert}
\DeclareFontFamily{OT1}{rsfs}{}
\DeclareFontShape{OT1}{rsfs}{n}{it}{<-> rsfs10}{}
\DeclareMathAlphabet{\mathscr}{OT1}{rsfs}{n}{it}
\swapnumbers
\newtheorem{thm}{Theorem}[section]
\newtheorem{lem}[thm]{Lemma}

\newtheorem{prop}[thm]{Proposition}
\newtheorem*{lem*}{Lemma}
\newtheorem*{thm*}{Theorem}
\newtheorem*{conj*}{Conjecture}
\newtheorem*{prop*}{Proposition}
\newtheorem{defn*}{Definition}

\newtheorem{cor}[thm]{Corollary}

\newtheorem{standassm}[thm]{Standing assumptions}
\theoremstyle{definition}
\newtheorem{defn}[thm]{Definition}
\theoremstyle{remark}
\newtheorem{ex}[thm]{Example}
\newtheorem{rem}[thm]{Remark}
\newtheorem{obs}[thm]{Observation}

\newtheorem*{obs*}{Observation}
\newtheorem*{rem*}{Remark}
\theoremstyle{definition}

\begin{document}
\title[Measure rigidity of Anosov flows]{Measure rigidity of Anosov flows via the factorization method}
\author{Asaf Katz}
\address{Department of Mathematics, University of Michigan, Ann Arbor, Michigan 48109, USA}
\email{asaf.katz@gmail.com}

\date{}

\begin{abstract} 
Using the factorization method of Eskin and Mirzakhani, we show that generalized $u$-Gibbs states over quantitatively non-integrable partially hyperbolic systems have absolutely continuous disintegrations on unstable manifolds.

As an application, we show a pointwise equidistribution theorem analogous to the equidistribution results of Kleinbock-Shi-Weiss and Chaika-Eskin. 
\end{abstract}
\maketitle
\section{Introduction}

In this paper we apply the factorization method originating in the work of A. Eskin and M. Mirzakhani~\cite{eskin-mirzakhani} to the study of measure rigidity for Anosov flows and diffeomorphisms.
The main result of~\cite{eskin-mirzakhani}, a measure classification theorem for $P$-invariant probability measures over the moduli space of translation surface, is too technical to be presented here; the interested reader may consult~\cite{wright} for more details. 
The main ingredient in this measure classification theorem  is~\cite[Theorem~$2.1$]{eskin-mirzakhani}. This theorem, which can be applied in more abstract settings than the geometrical settings of~\cite{eskin-mirzakhani}, actually provides a general technique for proving measure rigidity results in the presence of a hyperbolic flow, coupled with parabolic behavior induced from some of its unstable distributions.
This strategy was carried out by Eskin-Lindenstrauss~\cite{eskin_lindenstrauss,eskin_lindenstrauss-long} where they gave a different (and more general) proof to the theorem of Benoist-Quint~\cite{B-Q} on measure rigidity of random walks on homogeneous spaces.

In this paper we carry the technique further, showing a measure rigidity result in the general setting of non-uniformly hyperbolic systems, which include in particular certain  Anosov flows and diffeomorphisms.

Let $(M,g_{t},\mu)$ be a measure preserving system, where $M$ is a Riemannian manifold and $\mu$ is $g_{t}$-invariant and ergodic measure.

Oseledets theorem, as applied to the tangent bundle $TM$ states
\begin{thm}
There exists a set of numbers $\left\{\lambda_{i}\right\}$ and a \emph{measurable splitting} of $TM$, 
\begin{equation*}
    TM = \oplus_{i} E^{\lambda_{i}},
\end{equation*}
such that for every $0\neq v\in E^{\lambda{i}}$,
\begin{equation*}
    \lambda_{i} = \lim_{t\to\infty} \frac{\ln \left\lVert g_{t}.v \right\rVert}{\left\lVert v\right\rVert}.
\end{equation*}
\end{thm}
The numbers $\left\{ \lambda_{i}\right\}$ are called the \emph{Lyapunov exponents associated to $(M,g_{t},\mu)$} and the respective subspaces $\left\{E^{\lambda_{i}}\right\}$ are called \emph{the associated Lyapunov subspaces} and the splitting is called the Lyapunov splitting of $TM$ with respect to $\mu$,
By ergodicity, $\lambda_{i}$ and $\dim E^{\lambda_{i}}$ are constant $\mu$ almost everywhere.
A point $x\in M$ for which the Lyapunov splitting is defined is called a \emph{bi-regular} point.

We will order the Lyapunov exponents as $$\lambda_{k}>\lambda_{k-1}>\ldots >\lambda_{1}> 0 > \lambda_{-1} >\ldots \lambda_{-r}.$$

For each bi-regular point $x\in M$, we define the following sets of $M$
\begin{equation*}
    W^{\geq \lambda_{i}}(x) = \left\{ y\in M \mid \limsup_{t\to\infty} \frac{1}{t}\log d(g_{-t}.x, g_{-t}.y) \leq -\lambda_{i} \right\}.
\end{equation*}
A fundamental result of Y. Pesin~\cite{Pesin_1976} states that for $\lambda_{i}>0$ these are $C^{\infty}$-immersed smooth manifolds. They are called the \emph{$i$'th-unstable manifold of $g_{t}$ at $x$}.
Similarly, we define the \emph{stable submanifolds} at $x$ as
\begin{equation*}
    W^{\leq \lambda_{i}}(x) = \left\{ y\in M \mid \limsup_{t\to\infty} \frac{1}{t}\log d(g_{t}.x, g_{t}.y) \leq -\lambda_{i} \right\},
\end{equation*}
for $\lambda_{i}<0$.
We have that $W^{\geq \lambda_{i+1}}(x)$ is a subset of $W^{\geq \lambda_{i}}(x)$.

We define \emph{the unstable manifold} of $g_{t}$ at $x$ to be
\begin{equation*}
    W^{u}(x) = W^{\geq \lambda_{1}}(x).
\end{equation*}
Similarly, we define the \emph{stable} manifold as 
\begin{equation*}
    W^{s}(x) = W^{\leq \lambda_{-1}}(x).
\end{equation*}
Their tangent spaces are denoted $E^{u}, E^{s}$ respectively.

\begin{standassm}\label{standassm}
Throughout this paper, we will make the following assumptions:
\begin{itemize}
    \item There exists at-least \emph{two positive Lyapunov exponents}, namely $0<\lambda_{1}<\lambda_{2}<\ldots$. 
    \item $\dim E^{\lambda_{1}}=1$, namely the subspace corresponding the \emph{least positive} Lyapunov exponent is \emph{simple}.
    \item The only Lyapunov exponent equal to zero is in the $g_{t}$-flow direction. 
    \item $\dim E^{s}\geq 1$, namely the set of negative Lyapunov exponents is non-empty.
\end{itemize}
\end{standassm}

We make the following definition.
\begin{defn}
The \emph{fast-unstable} manifold through $x$, where $x$ is a bi-regular point, is defined to be
\begin{equation*}
    W^{uu}(x) = W^{\geq \lambda_{2}}(x).
\end{equation*}
\end{defn}
By our second assumption, for any bi-regular point $x$, we have that $W^{uu}(x)$ is a co-dimension 1 measurable lamination of $W^{u}(x)$, which is tangent to $E^{uu}$. This follows by using normal forms, as in \S~\ref{subsec:normal-forms}.

In view of our third assumption, we may define the \emph{center-stable and center-unstable manifolds} as
\begin{equation*}
    W^{cs}(x) = \bigcup_{t\in \mathbb{R}} g_{t}.W^{s}(x), \ W^{cu}(x) = \bigcup_{t\in \mathbb{R}} g_{t}. W^{u }(x).
\end{equation*}

Recall also the following definition.
\begin{defn}
A measure $\mu$ on $M$ is called a \emph{Sinai-Ruelle-Bowen measure} (or SRB measure) if its conditionals along the unstable manifolds $W^{u}$ are absolutely continuous with respect to the Lebesgue measure over $W^{u}$.
\end{defn}

\begin{defn}
We say that a $g_{t}$-invariant probability measure $\mu$ is a \emph{generalized $u$-Gibbs state} if its conditionals along the \emph{fast unstable} manifolds $W^{uu}$ are absolutely continuous with respect to the Lebesgue measure over $W^{uu}$. 
\end{defn}

It follows readily from the definition, and absolute continuity of $W^{uu}$ that every $SRB$ measure is a generalized $u$-Gibbs measure.
We refer the reader to~\cite{dolgopyat-u-gibbs-notes} for more material regarding $u$-Gibbs states.
In~\cite{PesinSinai1982}, considering partially-hyperbolic dynamics, Pesin and Sinai construct $u$-Gibbs measures by averaging smooth densities along expanding leaves.

The constructions of~\cite{Pesin_1976} discussed above are only \emph{measurable} in nature.
In particular, the radius of injectivity of the local manifolds $W^{\star}_{\text{loc}}$ is only a measurable function.
In view of Lusin's theorem, for any $\epsilon>0$, there exists a subset $M_{\epsilon}\subset M$ of measure $\mu(M_{\epsilon})>1-\epsilon$ for which all the measurable functions involved in defining these structures are continuous. In particular, these functions and the injectivity radius are of bounded size on $M_{\epsilon}$. Such sets are termed \emph{Lusin sets}. From now on, we will work in such a Lusin set.

Let $\mathcal{B}$ be a measurable partition, subordinate to $W^{u}$ with the one-sided Markov property. Such partitions are constructed in~\cite{Ledrappier_Strelcyn_1982, led-young}.
We denote by $\mathcal{B}(x)$ the atom containing $x$.
Using this partition, one may define \emph{conditional measures} with respect to this partition, which we denote by $\mu^{uu}_{x}$.
Using an analogous partition $\mathcal{B}^{-}$ subordinate to $W^{s}$, we may define the conditional measure $\mu^{s}_{x}$.
For a detailed discussion, see Section~\ref{sec:conditional-measures}.

Given this partition, for a generic point $x$, we define $W^{uu}_{\text{loc}}(x)$ as $W^{uu}(x) \cap \mathcal{B}(x)$. Similarly, we define $W^{s}_{\text{loc}}(x)$ as $W^{s}_{\text{loc}}(x) = W^{s}(x) \cap \mathcal{B}^{-}(x)$.
We define $W^{cs}_{\text{loc}}(x) = \cup_{t\in [-1,1]}g_{t}W^{s}_{\text{loc}}(x)$. 

The following definition plays a key role in our theorem.

\begin{defn}\label{def:QNI2}
A system $(M,g_{t},\mu)$ satisfying Assumptions~\ref{standassm} has the \emph{quantitative non-integrability} property (QNI) if:
\begin{itemize}
\item There is $\alpha>0$ and,
\item For every $\epsilon>0$ a subset $\mathcal{X} \subset M$ of measure $\mu(\mathcal{X})>1-\epsilon$ and, 
\item For every $\nu > 0$ a constant $C:=C(\nu,\epsilon)$ and a constant $t_0 =~ t_0(\nu,\epsilon)$ so that: 
\end{itemize}

\noindent Suppose $x\in \mathcal{X}$ and $t > t_0$ is such that $g_{t}x \in \mathcal{X}$ and $g_{-t}x \in \mathcal{X}$ then:

There is a subset $S_x \subset g_{t}(W^{s}_\text{loc})(g_{-t}x)$ with 
$$\mu^{s}_x(S_x) > (1-\nu) \mu^{s}_x(g_{t}W^{s}_\text{loc}(g_{-t}x))$$ 
with the following property:

For all $y \in S_x$ there exists $U_y \subset  g_{-t}(W^{uu}_\text{loc}(g_{t}x))$ with $$\mu^{uu}_x(U_y) > (1-\nu) \mu^{uu}_x(g_{-t}(W^{uu}_\text{loc}(g_{t}x))$$ so that if $z \in U_y$ then
\begin{equation}\label{eq:QNI-full-def}
d(W^{uu}_\text{loc}(y), W^{cs}_\text{loc}(z)) > C  e^{-\alpha t}.
\end{equation}

\end{defn}

The above Definition can be adapted to diffeomorphisms by taking a suspension construction, or alternatively use discrete times and modifying changing \eqref{eq:QNI-full-def} to a condition over the stable manifold. 

An alternative slightly more restrictive definition of the QNI condition is presented in Appendix~\ref{app:QNI-cond}.

Now we may state our main measure rigidity theorem.
\begin{thm}\label{thm:measure-classification}
Assume $(M,g_{t},\mu)$ that $\mu$ satisfies Assumptions~\ref{standassm} and $\mu$ is a generalized $u$-Gibbs measure with respect to $W^{uu}$. If $\mu$ satisfy the QNI condition, then $\mu$ is an SRB measure.
\end{thm}

We refer the reader to \S\ref{sec:backgroud} for a detailed account of the definitions involved in the discussion that follows.
An Anosov flow is a hyperbolic flow over a compact $N$-dimensional Riemannian manifold $M$; at every point $x\in M$, the tangent bundle admits a splitting 
\begin{equation*}
    T_{x}M=E^{s}(x)\oplus E^{0}(x)\oplus E^{u}(x)
\end{equation*}
to \emph{contracting, neutral} and \emph{expanding} parts, respectively, such that there exists $\lambda_{s}<0<\lambda_{u}$ for which
\begin{itemize}
    \item $\left\lVert g_{t}v\right\rVert \leq e^{\lambda_{s}t}\left\lVert v \right\rVert$ for all $v\in E^{s}$, \\
    \item $\left\lVert g_{t}v\right\rVert \geq e^{\lambda_{s}t}\left\lVert v \right\rVert$ for all $v\in E^{s}$, \\
    \item The $g_{t}$ action over $E^{0}$ is isometric and $\dim E^{0}=1$.
\end{itemize}

SRB measures exist in the case of Anosov system~\cite{Sinai3,young-1-central}.
Furthermore, in the case of transitive Anosov systems, such measures are unique, ergodic and mixing.
In the case where volume is preserved, the Anosov system is transitive and the unique SRB measure is volume.

\begin{cor}
Suppose $(M,f,\mu)$ is an \emph{Anosov diffeomorphism or a flow} satisfying Assumptions~\ref{standassm} and $\mu$ is a generalized $u$-Gibbs measure satisfying the QNI condition, then $\mu$ is SRB, and it is unique if the system is transitive.
\end{cor}

It is not known which classes of partially hyperbolic systems support $SRB$ measures in general.
In case of partially hyperbolic systems, our Theorem~\ref{thm:measure-classification} shows that generalized $u$-Gibbs measures that satisfy the QNI condition are actually $SRB$, hence such systems admit $SRB$ measures.

\begin{defn}\label{defn:dominated-splitting-eu}
Consider a manifold $M$ equipped with an Anosov flow $g_{t}$.
Assume moreover that the expanding subspace admits a \emph{dominated splitting} for every $x\in M$ as $$ E^{u}(x) = E^{uu}(x)\oplus E^{1}(x)$$ where $E^{1}(x)$ is one-dimensional, and being expanded in a slower rate than $E^{uu}(x)$, namely there exists a $t_{0}>0$ such that for any two unit vectors $u\in E^{uu}, u'\in E^{1}$, for any $t>t_0$ we have that
\begin{equation*}
    \frac{\left\lVert g_{t}u\right\rVert}{\left\lVert g_{t}u'\right\rVert} \geq 2.
\end{equation*}
The system $(M,g_{t})$ is called \emph{highly quantitatively non-integrable} (HQNI) if every $g_{t}$-invariant and ergodic generalized $u$-Gibbs state with respect to $E^{uu}$ satisfies the QNI property. 
\end{defn}
We give an example of such a flow (including an idea how to produce those) in \S\ref{sec:applications}.
For such class of manifolds we can deduce the following \emph{pointwise} equidistribution theorem, in the spirit of~\cite[Theorems~$2.6, 2.10$]{emm}

\begin{thm}\label{thm:equi}
Assume $(M,g_{t})$ is a HQNI flow.
Moreover, assume that the normal forms coordinates on $W^{>1}_{\text{loc}}(x)$ are \emph{linear}, for all $x$ in $M$.
Then for every $x\in M$, for $\mu^{uu}_{x}$ almost every $y\in W^{uu}(x)$, we have
\begin{equation}\label{eq:empherical-def}
    \frac{1}{T}\int_{t=0}^{T}f(g_{t}.y)dt  \to \int_{M}fdm,
\end{equation}
for any $f\in C(M)$, where $m$ is the unique SRB measure of the system $(M,g_{t})$.
\end{thm}

We define normal forms in \S~\ref{subsec:normal-forms}.
We remark here that the assumption about linearity of the normal forms over $W^{uu}_{\text{loc}}$ is done to simplify the argument. It seems possible to remove this assumption using more calculations via general  normal forms coordinates.
Moreover, the assumption is automatically filed in the case of $\dim E^{u}=2$ with two distinct positive Lyapunov exponents.

A main example we consider in this paper is the Borel-Smale Anosov diffeomorphism discussed in Example~\ref{ex:borel-smale} and certain fiber-perturbations of it, as discussed in Example~\ref{ex:borel-smale-pert}.
This is a volume-preserving Anosov system for which the space is the nilmanifold $N\times N/\Gamma$ where $N$ is the $3$-dimensional Heisenberg group and $\Gamma$ is some irreducible lattice.
Considering a suspension of this class of diffeomorphisms, for appropriately chosen parameters, we show in Example~\ref{ex:borel-smale} that any generalized $u$-Gibbs state for $E^{>1}$ satisfy the QNI condition, resulting in a HQNI system.
As an application, our measure classification theorem shows that any $u$-Gibbs measure must be the Haar measure, and the equidistribution result mentioned above is applicable.
Moreover, this system is also an example of a \emph{fibered system}, where the basis is the four dimensional torus $\mathbb{T}^{4}$ and the fiber is $\mathbb{T}^{2}$, where the action on both basis and fibers is Anosov, but the action over the fibers is \emph{slower} than the basis.
We show in Example~\ref{ex:borel-smale-pert} that for certain small $C^{\infty}$ perturbations of the fiber dynamics, our QNI and HQNI assumptions hold.
A similar example using the fibered system $ASL_{2}(\mathbb{R})/ASL_{2}(\mathbb{Z})$ is described in Example~\ref{ex:asl2}.

Our adaption of the factorization technique requires non-trivial input in the form of certain quantitative non-integrability estimates over measurable laminations of the system. We verified this condition for certain perturbations of homogeneous systems. A promising direction of research is to try to verify this condition to other types of dynamical systems.

We consider the set $\mathcal{A}^{k}\subset \text{Diff}^{k}(\mathbb{T}^{3})$. This is the set of $C^{k}$ Anosov maps with dominated splitting as in Definition~\ref{defn:dominated-splitting-eu}.

In an upcoming paper of Avila-Croviser-Eskin-Potrie-Wilkinson-Zhang~\cite{ACEPWZ}, they will show that for any map in $\mathcal{A}^{\infty}(\mathbb{T}^{3})$ either $E^{s}, E^{uu}$ are jointly integrable, or any $u$-Gibbs state is SRB.
Their proof involves first showing a dichotomy that either the system satisfies the QNI condition, or the Franks-Manning conjugacy preserves the fast unstable foliation and then applying the results of this paper.

As minimality of $W^{uu}$ is $C^1$-open and $C^\infty$-dense inside the set $\mathcal{A}^{\infty}$ as shown by Avila-Crovisier-Wilkinson, it seems likely that systems that satisfy a QNI condition are open and dense in that set.

Since this paper was sent to publication, two interesting results related to the current paper have appeared.

In~\cite{Alvarez_Leguil_Obata_Santiago_2022}, Alvarez-Leguil-Obata-Santiago study the problem of rigidity of $3D$ Anosov diffeomorphisms over $\mathbb{T}^{3}$. They show that for any perturbation of a \emph{conservative map} in $\mathcal{A}^{2}$, either a QNI-like condition holds or $E^{s}, E^{uu}$ are jointly integrable. The volume preserving assumption give a bunching condition over the Lyapunov spectrum, allowing one to deduce extra smoothness of a relevant holonomy map.
In their restrictive setting, they achieve a factorization theorem for substantially less regular systems.

In~\cite{Crovisier_Obata_Poletti_2022}, Crovisier-Obata-Poletti construct a smooth system on $\mathbb{T}^4$ in the sense of the Borel-Smale example discussed in Example~\ref{ex:borel-smale}. These systems cannot be constructed in a purely algebraic form.
They apply a similar computation to the one in Example~\ref{ex:borel-smale-pert} (c.f.~\cite[\S~$5$]{Crovisier_Obata_Poletti_2022}), together with earlier results of Obata, in order to conclude that the conclusion that $u$-Gibbs measures are $SRB$ is prevalent for these system as well.

Of a particular interest is studying the QNI concept in the settings of analytic maps. Preliminary results for holomorphic maps over compact K\"{a}hler surfaces (related to to~\cite{brown-hertz}), appeared in~\cite{Cantat_Dujardin_2021}. In the settings of Anosov diffeomorphisms, Gogolev et al. made some numerical studies in~\cite{gogolev}

\subsection*{Organization of the paper}
In \S\ref{sec:backgroud} we recall certain preliminaries from smooth dynamics and ergodic theory.
In \S\ref{sec:conditional-measures} we construct certain conditional measures on the unstable manifolds, using normal forms coordinates, which will allow us to show extra invariance.
Section \S\ref{sec:factorization} is the most technical part of the paper, which revoles around construction of special linear operators $A(q,u,\ell,t)$ that will allow us to measure the relative divergence between the two strong-unstable manifolds.
In \S\ref{sec:bilip} we prove the crucial bilipschitz estimates regarding some stopping time which are defined by the factorization theorem.
In \S\ref{sec:8-pts} we employ the Eskin-Mirzakhani technique in order to conclude extra invariance.
In \S\ref{sec:applications} we provide several examples and applications of our measure rigidity result towards equidistribution problems.
Appendix~\ref{app:factorization-details} gives the technical details regarding the ``factorization theorem'' used in~\S\ref{sec:factorization}.
Appendix~\ref{app:avoiding-tech} gives the proofs of several technical lemmata from \S\ref{sec:conditional-measures}.
Appendix~\ref{app:QNI-cond} discuss a slightly more restrictive but natural QNI condition.

\subsection*{Acknowledgments}
The author would like to express his deep appreciation to Alex Eskin for explaining his important work together with Maryam Mirzakhani about measure classification and providing many important insights about the factorization technique.
The author also would like to thank Aaron Brown for useful suggestions and comments, primarily about the usage of normal forms coordinates.
It is a pleasure to thank Amie Wilkinson, Elon Lindenstrauss, Federico Rodriguez Hertz, Clark Butler and Rafael Potrie for helpful and motivating discussions. The author also would like to thank the referees for providing many useful suggestions and comments that helped to improve the paper and its presentation.

\section{Background from smooth dynamics}\label{sec:backgroud}
\subsection{Basics}

Let $M$ be a closed compact manifold, equipped with a smooth $1$-parameter flow $\left\{g_{t} \right\}_{t\in\mathbb{R}}$.

Assume from now on that $(M,g_{t})$ is equipped with a Borel probability measure $\mu$ which is $g_{t}$-invariant and ergodic.

\subsection{Oseledets' multiplicative ergodic theorem}
We recall the following version of Oseledets' multiplicative ergodic theorem.
\begin{thm}[Oseledets' ergodic theorem]\label{thm:oseledets}
 Assume that $V\to (\Omega,\beta,\mu,T)$ is a cocycle over an ergodic probability measure-preserving system.
 Assume that $V$ is equipped with a metric $\norm{-}$ on each fiber such that
 \begin{align}
   \label{eqn:cond_L1_bdd_cocycle}
   \int_{\Omega} \log^+\norm{T_\omega}_{op} d\mu(\omega) < \infty.
 \end{align}
 Here $\log^+(x):=\max(0,\log x)$ and $\norm{-}_{op}$ denotes the operator norm of a linear map between normed vector spaces.
 
 Then there exist real numbers $\lambda_1<\lambda_2<\cdots <\lambda_k$ (with perhaps $ \lambda_1 = -\infty $) and $T$-invariant subbundles of $V$ defined for a.e. $\omega\in \Omega$:
 \[
 0\subsetneq V^{\leq \lambda_1} \subsetneq \cdots \subsetneq V^{\leq \lambda_k} = V
 \]
 such that for vectors $v\in V^{\leq \lambda_{i+1}}_\omega\setminus V^{\leq \lambda_{i}}_\omega$ we have
 \begin{align}
 \label{eqn:growthofv}
   \lim_{N\to \infty} \frac 1N \log \norm{T^N v} \to \lambda_i.
 \end{align}
\end{thm}
The flag of the subspaces is known as the \emph{forward flag}.

As we assume our transformation is invertible, applying the same theorem to its inverse, one recovers the \emph{backward flag}.

Refining both, one may recover the \emph{Oseledets splitting} of $V$.

Applying Oseledets' theorem to the derivative cocycle defined by the $g_{t}$ flow over the tangent bundle $TM$ yields
\begin{cor}
For $\mu$-almost every $x\in M$, there exists a splitting of $T_{x}M$ named the \emph{Lyapunov splitting}, such that
\begin{equation*}
    T_{x}M = \oplus_{i} E^{\lambda_{i}}(x)
\end{equation*}
and for every $v\in E^{\lambda_{i}}(x)$ we have that
\begin{equation*}
    \lim_{t\to\infty}\frac{1}{t}\log\lVert g_{t}.v\rVert = \lambda_{i}.
\end{equation*}
\end{cor}
As we may think about the assignment taking a point $x\in M$ to the Lyapunov splitting of $T_{x}M$ or an appropriate flag (inside an appropriate Grassmanian) as a \emph{measurable section} of a smooth bundle, in view of Lusin's theorem, for every $\varepsilon>0$ there exists a compact subset of $M$ such that this section is continuous.
Hence for a compact set of arbitrarily large measure, we may bound those growth rates uniformly.

Moreover, as the individual Lyapunov spaces become orthogonal in the sense of \cite[Remark~$3.1.8$]{pesin}, by introducing the so-called ``Pesin norms'' again at the cost of discarding a subset of measure at-most $\varepsilon$ from this subspace there exists some $\theta$ such that the individual Lyapunov spaces are $\theta$-transverse at any point of our subset.
Furthermore, by restricting to a smaller set if needed, it follows from \emph{Pesin theory} that one may find a compact set of large measure for which all the geometry of the embedded disks $W^{uu}_{\text{loc}}(x)$ varies continuously and in a bounded way (c.f. \cite[Supplement $S.3$]{Katok1995}).
We will denote this set as $M_{Os,\theta}$.

Given $\varepsilon>0$ and a time $T'$ we may define the following subset of \emph{Oseledets regular points} $M_{\text{Os-reg},\varepsilon,T'}\subset M$ in the following manner:
\begin{equation*}
    M_{\text{Os-reg},\varepsilon,T'} = \left\{x\in M \mid \forall t>T' \ \lvert\{0\leq s \leq t \mid g_{t}.x\in M_{Os,\theta} \}\rvert \geq (1-2\varepsilon)\cdot t \right\}.
\end{equation*}

In words, a point $x$ is Oseledets' regular if for any large time $t>T'$, the proportion of the time the orbit $\left\{g_{s}.x\right\}_{0\leq s \leq t}$ spends inside the set of points with good Lyapunov splitting is greater than $1-2\varepsilon$.

In view of Birkhoff's pointwise ergodic theorem, applied to the system $(M,g_{t},\mu)$ towards the return times of $\mu$-almost every point to $M_{Os,\theta}$, we see that given $\varepsilon>0$, for $T'$ large enough,  $\mu(M_{\text{Os-reg},\varepsilon,T'}) \geq~1-~3\varepsilon$.

\subsection{Lyapunov norms}\label{sub:lyap-norm}
We will need the following result due to Y. Pesin:
\begin{thm}[\cite{pesin} \S$3.5$]
Let $T$ be a measurable cocycle over a flow $\{g_{t}\}$. 
Assume $x\in X$ is an Oseledets bi-regular point, and assume $\left\{ H_{i}(x)\right\}$ are the various Oseledets subspaces associated to $T$ in $x$.
There exists an inner product $\left<, \right>_{x}$ on $H_{i}(x)$ such that there exist constants $C_{1},C_{2}$ for which
\begin{equation*}
    C_{1}\cdot t \leq \log\left(\frac{\lVert (g_{t}).v \rVert_{x}}{\lVert v \rVert_{x}} \right) \leq C_{2} \cdot t
\end{equation*}
for any $0\neq v\in H_{i}(x)$ with the naturally associated norm $\lVert v \rVert_{x}^{2}~=~\left<v,v \right>_{x}$.
Moreover, distinct Oseledets subspaces are \emph{orthogonal} with respect to this inner-product.
\end{thm}
The following Lemma is standard in the literature (c.f. \cite[Lemma~$2.7$]{eskin_lindenstrauss} \cite[Lemma~$4.17$]{eskin-mirzakhani}).
\begin{lem}[Norm comparison]\label{lem:norm-comp}
For every $\delta>0$ there exists some compact set $K\subset \Omega$ with $\mu(K)>1-\delta$ and $C(\delta)>0$ such that for any $x\in \Omega$ and $v\in V(x)$ we have
\begin{equation}
    C(\delta)^{-1}\cdot \lVert v \rVert \leq \lVert v \rVert_{V(x)} \leq C(\delta)\cdot \lVert v \rVert,
\end{equation}
where the middle norm is the Lyapunov norm and the other norms are arbitrary fixed norm defined over $V$.
\end{lem}

The proof of this lemma follows by an application of Lusin's theorem.
Lemma~\ref{lem:norm-comp} allows one to directly compare estimates in the Lyapunov norm and the arbitrary norm (up to multiplicative constant) and will play a role in the proof of the bilipschitz estimates in \S~\ref{sec:bilip}.

\subsection{Normal forms coordinates}\label{subsec:normal-forms}
From now, we restrict ourselves to the study of the unstable lamination $W^{u}$.
Let $0<\lambda_1<\dots<\lambda_\ell$ be the distinct Lyapunov exponents of $g_{t}$, restricted to the expanding subspace $E^{u}$
and let $E^{u}(x)=E^{1}(x) \oplus \dots \oplus E^{\ell}(x)$ be the splitting of $E^{u}(x)$ into the (expanding) Lyapunov subspaces given by Oseledets' Theorem~\ref{thm:oseledets}.

\begin{defn}\label{defn:sub-resonant}
We say that a map between vector spaces is {\em polynomial}\, if each component  is given by a polynomial
in some, and hence every, basis.
We identify $\mathbb{R}^{n}=E^{1}\oplus \cdots \oplus E^{n}$, for some splitting of $\mathbb{R}^{n}$.
We consider an affine polynomial map $P: \mathbb{R}^{n} \to \mathbb{R}^{n}$,
split it into components $(P_1(t),\dots,P_{\ell}(t))$, where $P_i: \oplus_{j=1}^{n}E^{j} \to E^{i}$.
Each $P_i$ can be written uniquely as a linear combination of  polynomials
of specific homogeneous types: we say that $Q: \mathbb{R}^{n} \to E^i $ has homogeneous type $s= (s_1, \dots , s_\ell)$ if for any real numbers
$a_1, \dots , a_\ell$ and vectors
$t_j\in E^{j}$, $j=1,\dots, \ell,$ we have
\begin{equation}\label{eq:stype}
Q(a_1 t_1+ \dots + a_\ell t_\ell)= a_1^{s_1} \cdots  a_\ell^{s_\ell} \, Q( t_1+ \dots + t_\ell).
\end{equation}

We say that a polynomial map $P: \mathbb{R}^{n} \to \mathbb{R}^{n}$ is {\em sub-resonant}
if each component $P_i$ has only terms of homogeneous types $s= (s_1, \dots , s_\ell)$  satisfying {\em sub-resonance relations}
\begin{equation}\label{eq:sub-resonance}
\lambda_i\ge\sum s_j \lambda_j, \quad\text{where
$s_1,\dots,s_{\ell}\,$ are non-negative integers.}
\end{equation}
Clearly, for any sub-resonance relation we have $s_j=0$ for $j<i$ and $\sum s_j \leq \lambda_\ell/ \lambda_1$.
It follows that sub-resonant polynomial maps have degree at most
\begin{equation}\label{degree}
d=d(\chi)= \lfloor \chi_\ell/\chi_1 \rfloor.
\end{equation}
Sub-resonant polynomial maps $P: \mathbb{R}^{n} \to \mathbb{R}^{n}$ with $P(0)=0$ with invertible derivative at the origin form a group with respect to composition.
We will denote this  finite-dimensional Lie group by $G^{\chi}$.
When we identify a particular copy of $\mathbb{R}^{n}$ with $T_{x}W^{u}(x)$ for a Pesin regular point $x$, we will denote the associated copy of this Lie group by $G^{\chi}_{x}$.

Let $n=\dim W^{u}(x)$ so $T_{x}W^{u}(x)\simeq \mathbb{R}^{n}$.
Consider some $y\in W^{u}(x)$, we also have $W^{u}(x)=W^{u}(y)$, hence $T_{y}W^{u}(y)\simeq\mathbb{R}^{n}$.
Furthermore, if $x,y$ are both Oseledets regular, both of the tangent spaces $T_{x}W^{u}(x), T_{y}W^{u}(y)$ refine to an Oseledets splitting $\oplus_{i>0} E^{\lambda_{i}}(x), \oplus_{i>0} E^{\lambda_{i}}(y)$ with $E^{i}(x)\simeq~E^{i}(y)$ for all $i=1,\ldots, n$.
We denote the class of sub-resonant polynomial between $\oplus_{i>0} E^{\lambda_{i}}(x)$ and $\oplus_{i>0} E^{\lambda_{i}}(y)$ by $\mathsf{SR}_{x,y}$.

Consider 
\begin{equation}\label{eq:definition-of-unstable-param}
    \mathcal{H}_{x}=\left\{\varphi:\mathbb{R}^{n}\to W^{u}(x) \mid D_{0}\varphi \text{ is an isometry} \right\}.
\end{equation}
The group $G^{\chi}_{x}$ acts on $\mathcal{H}_{x}$ by pre-composition.
A fundamental theorem by B. Kalinin and V. Sadovskaya~\cite[Theorem~$2.5$]{kalinin} enables us to construct \emph{normal forms coordinates} over a subset of $M$ of full measure, which contains full unstable leaves.
We interpret the construction of~\cite{kalinin} in the following manner: 
consider the bundle of isometric parameterization of the unstables manifolds, modulo the sub-resonant group $G^{\chi}$, namely $\bigsqcup_{x\in M}\mathcal{H}_{x}/G^{\chi}_{x}$. \emph{Normal forms coordinates} amount choosing a \emph{measurable} section 
\begin{equation*}
    x\mapsto \phi_{x} \in \bigsqcup_{x\in M}\mathcal{H}_{x}/G^{\chi}_{x}
\end{equation*}
which is $g_{t}$-equivariant in the sense:

\begin{equation*}
    \phi_{g_{t}.x} \circ g_{t} \circ \phi_{x}^{-1} \in \mathsf{SR}_{x,g_{t}.x},
\end{equation*}
namely the $g_{t}$ dynamics over $W^{u}(x)$ is given by a sub-resonant polynomial. 
Furthermore, if a pair of points $x,x'$ are unstably related then $\phi_{x'}=~\mathsf{P}_{x,x'}\circ\phi_{x}$ for a sub-resonant polynomial $\mathsf{P}_{x,x'}\in~\mathsf{SR}_{x,x'}$. 
\end{defn}

\begin{ex}
In terms of coordinates, the normal forms coordinates will have the following structure
\begin{equation}\label{eq:G-subresonant-group}
    G^{\chi}_{x} = \left\{(y_{1},\ldots, y_{n}) \mapsto (y_{1}, y_{2}+\mathsf{P}_{2}(y_{1}), \ldots, y_{n}+\mathsf{P}_{n}(y_{1},\ldots, y_{n-1}) \right\},
\end{equation}
where $\mathsf{P}_{i}:E^{1}\oplus \cdots \oplus E^{i-1} \to E^{i}$ are polynomials, each of which is a sum of homogeneous polynomials $\mathsf{P}_{i,j}$ satisfying the sub-resonance condition~\eqref{eq:sub-resonance}.
For example if we consider only $E^{1}\oplus E^{2}$, the maps will be of the form 
\begin{equation}\label{eq:2-dim-sub-res}
(y_{1},y_{2}) \mapsto (y_{1}, y_{2}+\mathsf{P}_{2,x}(y_{1}))     
\end{equation}
where $\mathsf{P}_{2,x}:E^{1}(x)\to E^{2}(x)$ is a polynomial, vanishing at $0$, of degree less or equal to $\lambda_{2}/\lambda_{1}$.
In case of two points $x,x'$ such that $x'\in W^{u}(x)$ the resulting maps from $\mathsf{SR}_{x,x'}$ can be described in the following manner: 
Determine coordinates $(y_{1},y_{2})$ on $T_{x}W^{u}(x)$.
Then the matching sub-resonant coordinates on $T_{x'}W^{u}(x')$ are $(y_{1}',y_{2}')$ where
\begin{equation}\label{eq:explicit-sub-res}
    y_{1}' = a_{1}\cdot y_{1}+b_{1}, \hspace{1 cm}  y_{2}'=a_{2}\cdot y_{2}+\mathsf{P}_{2,x}(y_{1})+b_{2}
\end{equation}
for certain numbers $a_{1},a_{2},b_{1},b_{2}\in \mathbb{R}$ and $\mathsf{P}_{2}:E^{1}(x)\to E^{2}(x)$ a polynomial, vanishing at zero, of degree less or equal to $\lambda_{2}/\lambda_{1}$.
In the case demonstrated in~\eqref{eq:2-dim-sub-res}, we can describe the dynamics of the $g_{t}$-flow for the time-$1$ map and its iteration as
\begin{equation}\label{eq:g-action-normal-forms}
    \begin{split}
    g_{1}.(y_{1},y_{2}) &= \left(e^{\lambda_{1}}\cdot y_{1}, e^{\lambda_{2}}y_{2}+\mathsf{P}_{2,x}\left(y_{1}\right)\right), \\
    g_{2}.(y_{1},y_{2}) &= \left( e^{2\cdot\lambda_{1}}\cdot y_{1}, e^{\lambda_{2}}\cdot \left( e^{\lambda_{2}}y_{2}+\mathsf{P}_{2,x}\left(y_{1}\right)\right) + \mathsf{P}_{2,g_{1}.x}\left(e^{\lambda_{1}}\cdot y_{1}) \right) \right)\\
    \vdots
    \end{split}
\end{equation}
\end{ex}
The previous example illustrates the fact that the strictly sub-resonant polynomial group is \emph{nilpotent}.

We note the following easy observation.
\begin{obs}\label{obs:sub-resonance}
The polynomial $P_{1,x}:\mathbb{R}^{n}\to E^{1}$ is an  \emph{affine} map.
\end{obs}
\begin{proof}
We have that $\lambda_{i}>\lambda_{1}$ for each index $i>1$, hence we must have $s_{i}=0$ in view of the sub-resonance condition~\eqref{eq:sub-resonance}.
Furthermore, finer analysis of~\eqref{eq:sub-resonance}
 gives $s_{1}=0,1$ is the only possibility. \end{proof}

\section{Conditional measures}\label{sec:conditional-measures}
\texttt{Standing assumption} In order to simplify the notation, we will assume from now on that $M$ carries a measurable Lyapunov splitting $$TM=E^{s}\oplus E^{0}\oplus E^{\lambda_{1}}\oplus E^{\geq \lambda_{2}},$$ where we denote $W^{uu}_{\text{loc}}$ - the \emph{fast unstable manifold} with a tangent space equal to $E^{\geq \lambda_{2}}$ and $W^{u}$ - the \emph{unstable manifold} with a tangent space equal to $E^{1}\oplus E^{\geq \lambda_{2}}$. The proof of our theorem carries verbatim to the case of higher dimensional manifolds from it.

\subsection{Construction of a partition}
A non-empty closed subset $R\subset M$ is called a  \emph{rectangle} if
\begin{itemize}
    \item $diam(R)<\delta$ for $\delta$ small enough,
    \item $R=\overline{\text{int}R}$, where $\text{int}R$ is is the interior of $R$,
    \item the segment $[x,y]\subset R$ for every $x,y\in R$.
\end{itemize} 
A rectangle $R$ has a \emph{direct product structure} that is given $x\in R$ there exists a homeomorphism 
\begin{equation*}
    \theta:R\to R\cap W^{s}(x) \times R\cap W^{0}(x) \times R\cap W^{u}(x).
\end{equation*}
One may show that $\theta, \theta^{-1}$ are Holder continuous.
A finite cover \\ $\tilde{R}~=~\left\{R_1, R_2,\ldots, R_p \right\}$ of $M$ by rectangles $R_{i}$ is said to be a Markov partition if
\begin{itemize}
    \item $\text{int}R_{i}\cap \text{int}R_{j}=\emptyset$ for all $i\neq j$,
    \item for each $x\in \text{int}R_{i}\cap g_{-1}.(\text{int} R_{j})$ we have
    \begin{equation*}
        g_{1}.(W^{s}(x)\cap R_{i})\subset W^{s}(g_{1}.x)\cap R_{j}, \ g_{1}.(W^{u}(x)\cap R_{i}) \supset W^{u}(g_{1}.x)\cap R_{j}. 
    \end{equation*}
\end{itemize}

Such partitions was constructed first by Y. Sinai for the case of Anosov diffeomorphismes and later generalized~\cite{Sinai1,ratner-partitions,bowen-markov}.

As we are interested in a more general setting than just Anosov and we do not require the full force of Markov partitions, we will only be interested in partitions which are \emph{subordinate to $W^{u}$}.
\begin{defn}
    A measurable partition $\xi$ is \emph{subordinate to $W^{u}$} if for $\mu$ almost every $x\in M$:
    \begin{itemize}
        \item $\xi(x) \subset W^{u}(x)$,
        \item $\xi(x)$ contains a neighborhood of $x$ open in the submanifold topology of $W^{u}(x)$.
    \end{itemize}
\end{defn}
In this more general context, such partitions were constructed by Ledrappier-Strelcyn~\cite[Proposition~$3.1$]{Ledrappier_Strelcyn_1982} and Ledrappier-Young~\cite[\S $3$]{led-young}.
In particular, one can find a subordinate measurable partition which is \emph{generating and separating}. Moreover, one can construct such a partition which satisfy the forward Markov condition
\begin{equation*}
    W^{u}((g_{1}.x)\cap R_{j} \subset g_{1}\left(W^{u}(x)\cap R_{i}\right),
\end{equation*}
for every $x\in \text{int} R_{i}\cap g_{-1}(\text{int} R_{j})$.
We call such a partition a partition having the  \emph{one-sided Markov property}.
For more background on such partitions, see~\cite[Appendix~$B$]{Brown_Triestino_2019} and~\cite[Proposition $3.7$, Lemma $3.8$ and Appendix~$B$]{eskin-mirzakhani}.

We fix such a partition of $M$ and denote it by $\mathcal{B}_{0}$.
We denote by $\mathcal{B}_{0}(x)$ the atom of $\mathcal{B}_{0}$ containing $x$.

\begin{obs}
Suppose $x\in M$ and $y\in \mathcal{B}_{0}(x)\cap W^{u}_{\text{loc}}(x)$ then
$$ g_{-t}.y \in \mathcal{B}_{0}(g_{-t}.x). $$ 
\end{obs}

This property follows from the construction at~\cite[Proposition~$3.1.1$]{Ledrappier_Strelcyn_1982}
For $t>0$ we denote 
\begin{equation*}
\mathcal{B}_{t}(x) = g_{-t}.(\mathcal{B}_{0}(g_{t}.x)\cap W^{u}(g_{t}.x)).
\end{equation*}

\begin{lem}
The following properties of $\mathcal{B}_{t}$ hold:
\begin{enumerate}
    \item For $t'>t\geq 0$ we have $\mathcal{B}_{t'} \subset \mathcal{B}_{t}$.
    \item Suppose that $t\geq 0$,$t'\geq 0$ and $x,x'\in M$ such that $\mathcal{B}_{t}(x)\cap~\mathcal{B}_{t'}(x')\neq~\emptyset $. Then either $\mathcal{B}_{t'}(x')\subset \mathcal{B}_{t}(x)$ or $\mathcal{B}_{t}(x)\subset \mathcal{B}_{t'}(x')$.
\end{enumerate}
\end{lem}
\begin{proof}
The first part follows immediately from the observation above.
For the second part, we assume that $t'\geq t$, as the partitions are refined as $t$ increases, if $\mathcal{B}_{t}(x)\cap~\mathcal{B}_{t'}(x')\neq~\emptyset$ we must also have $\mathcal{B}_{t}(x)\cap~\mathcal{B}_{t}[x']\neq~\emptyset$.
Now assume $y\in \mathcal{B}_{t}(x)\cap\mathcal{B}_{t}(x') $. Then we have that $g_{t}.y\in \mathcal{B}_{0}(g_{t}.x)$ and also $g_{t}.y\in \mathcal{B}_{0}(g_{t}.x')$. As $\mathcal{B}_{0}$ is a partition, we must have that $\mathcal{B}_{0}(g_{t}.x)=\mathcal{B}_{0}(g_{t}.x')$, which means in turn that $\mathcal{B}_{t}(x)=\mathcal{B}_{t}(x')$ per the definition of $\mathcal{B}_{t}$. Again by the refinement property we get that
\begin{equation*}
    \mathcal{B}_{t'}(x')\subset\mathcal{B}_{t}(x')=\mathcal{B}_{t}(x).
\end{equation*}
\end{proof}

\subsection{Construction of conditional measures over the stable and unstable leaves}\label{subsec:conditionals-construction}

We refer the reader to the survey in~\cite[\S5]{einsiedler-lindenstrauss} and the book~\cite[\S5]{einsiedlerbook} for more background about conditional measures.

By its construction, the sets $\mathcal{B}_{0}(x)$ are the atoms of a \emph{measurable partition} of $M$ subordinate to $W^{u}$.
We let $\left\{\mu_{x}^{u}\right\}_{x\in M}=~\left\{\mu\mid_{W^{u}(x)}\right\}_{x\in M}$ denote the set of conditional measures of $\mu$ along this partition (c.f \cite[$\S1.4, \S3.2$]{clim-katok}).

This gives a rise to a \emph{measurable function} $\mathbf{f}_{1}$ defined as in~\cite[\S1.5]{clim-katok}\cite[Theorem~$5.9$]{einsiedler-lindenstrauss}.

Namely, for each atom $\mathcal{C} \in \mathcal{B}$ we have
\begin{equation*}
    \mathbb{E}(\phi \mid \mathcal{C}) = \int \phi d\mathbf{f}_{1}[x],
\end{equation*}
for all measurable functions $\phi:X\to\mathbb{R}$.

\begin{defn}\label{def:wasserstein}
We define the (normalized) \emph{Wasserstein metric} $d_{W}$ between two conditional measures (of bounded support) as
\begin{equation*}
    d_{W}(\mu_{1},\mu_{2})=\sup_{h:M\to\R \text{ is Lipschitz with }Lip(h)\leq 1 }\left\{\left\lvert\int_{M}h(x)\left(\frac{d\mu_{1}(x)}{\mu_{1}(M)}-\frac{d\mu_{2}(x)}{\mu_{2}(M)}\right)\right\rvert\right\}.
\end{equation*}
\end{defn}

While this metric is weaker than the Radon metric, it does induce the topology of weak-$\star$ convergence over the space of measures, up to normalization.

The function $\mathbf{f}_{1}$ should be thought as a measurable function from $M$ to the set of
probability measures defined over $M$, endowed with a suitable topology coming from the Wasserstein metric as defined in Definition~\ref{def:wasserstein}  (c.f. \cite[Proposition~$5.17$]{einsiedler-lindenstrauss}).

\begin{rem}
While the conditional measures are normalized over each atom by their construction, this normalization is not the correct normalization for us as we will study push-forwards of these measures which may a-priori change its normalization. This modified distance function makes this renormalization transparent while keeping the main feature - namely absolute continuity, intact.
A standard argument based on Poincare's recurrence will be used in order to make sure the renormalization is indeed by the correct factors.
\end{rem}

We will interpret the function $\mathbf{f}_{1}$ in the following way -  consider a normal forms structure over the lamination induced by the unstable manifolds $\left\{W^{u}(x)\right\}_{x\in M}$.
Hence for almost every $x$, we have a map $\phi_{x}:\mathbb{R}^{n}\to W^{u}(x)$, arising from the choice of normal forms coordinates.
Consider its inverse map $\psi_{x}:W^{u}(x)\to\mathbb{R}^{n}$.
Restrict $\psi_{x}$ to $\mathcal{B}_{0}(x)$.
We may push-forward the measure $\mu_{x}^{u}$ by the map $\psi_{x}$ to a measure over $\mathbb{R}^{n}$, $\psi_{x}.\mu_{x}^{u}$.
Hence we may identify $\mathbf{f}_{1}(x)$ with $\psi_{x}.\mu_{x}^{u}$.

Given a subresonant polynomial $S:\R^n\to\R^n$, we may get subresonant map $\tilde{S}$ over the leaf $W^{u}(x)$ by conjugation with $\psi_{x}$.
\begin{defn}
We will say that $\mathbf{f}_{1}$ is \emph{S-invariant} if for $\mu$ almost every $x$ we have
\begin{equation*}
    d_{W}\left(\tilde{S}.\mathbf{f}_{1},\mathbf{f}_1 \right)=0.
\end{equation*}
We note here that the measure $\mathbf{f}_1$ is only defined in each atom, in view of the definition of the distance $d_{W}$, the measures only match over the intersection of a set with its $\tilde{S}$-translation.
\end{defn}
Using this definition, we may define the notions of a generalized $u$-Gibbs state and an SRB measure.
\begin{defn}
We say that $\mu$ is an SRB measure if for $\mu$ almost every $x$, for every subresonant polynomial $S\in G^{\chi}$ and its associated subresonant map $\tilde{S}$ we have
\begin{equation*}
    d_{W}\left(\tilde{S}.\mathbf{f}_1(x),\mathbf{f}_1(x)\right)=0.
\end{equation*}
\end{defn}

\begin{lem}
The above definition of SRB measure is equivalent to the definition given in the introduction, namely that $f\mid W^{u}$ is absolutely-continuous.
\end{lem}

The proof follows from Ledrappier's result in~\cite{Ledrappier-Sinai}.
A more detailed account can be found in~\cite[\S9, Appendix D]{Brown_Malicet_Obata_Santiago_Triestino_Alvarez_Roldan_2019}.
By the construction of normal forms by Kalinin and Sadovskaya~\cite{kalinin-sadovskaya-singe-lyapunov} one may endow $M$ with the (measurable) action of the subresonant group, which is $g_{t}$-equivarient by construction.
\begin{proof}[Proof Sketch]
Fix some $x\in M$ a $\mu$-generic point.
By Ledrappier~\cite[Proposition~$3.7$]{Ledrappier-Sinai}, as $\mu$ is assumed to be $SRB$ measure, we can write $\frac{d\mu}{dLeb} = \Delta$ for a Radon-Nykodim derivative $\Delta$, where $Leb$ stands for the Lebesgue measure over $W^{u}(x)$.
Moreover, the following formula for $\Delta$ is given in~\cite[Proposition~$3.1$]{Ledrappier-Sinai}
\begin{equation*}
    \Delta(y,z)=\prod_{i=1}^{\infty}\frac{\text{Jac}^{u}(g_{-i}.y)}{\text{Jac}^{u}(g_{-i}.z)},
\end{equation*}
for $y,z\in W^{u}(x)$.
We endow $W^{u}(x)$ with the normal forms coordinates.
Hence we may find two vectors $v_{y}, v_{z}\in\mathbb{R}^{\dim E^{u}}$ such that in that coordinate system, $y,z$ are identified with $v_{y}, v_{z}$ respectively.
Using the equivarient property of the normal forms coordinates we get
\begin{equation*}
    g_{-t}.y = (dg_{-t}.x)\cdot \left( v_{y}+\mathsf{P}_{x,y}(v_{y}) \right), 
    g_{-t}.z = (dg_{-t}.x)\cdot \left( v_{z}+\mathsf{P}_{x,z}(v_{z}) \right),
\end{equation*}
for some subresonant polynomials $\mathsf{P}_{x,y},\mathsf{P}_{x,z}$.
Writing the above in a matrix form according to $\{v_{1},\ldots, v_{\dim E^{u}}\}$ we get
\begin{equation*}
    g_{-t}.\star = \left(dg_{-t}.x\right) \cdot \left(v_{\star} + \mathsf{N}^{-t}_{x,y}v_{\star}\right),
\end{equation*}
where $\star$ is either $y,z$ and $\mathsf{N}^{-t}$ is some \emph{nilpotent} upper triangular matrix in view of the subresonant condition (c.f.~\eqref{eq:g-action-normal-forms}).
Taking derivatives (in the $v$ coordinates) of the above equation we get that the derivative is $(dg_{-t}.x)\left(I+\tilde{\mathsf{N}}^{-t}_{x,\star}(v_{\star})\right)$, for some other nilpotent map $\tilde{\mathsf{N}}^{-t}_{x,\star}$.
As $\tilde{\mathsf{N}}^{-t}_{x,\star}(v_{\star})$ is nilpotent, its only eigenvalue is $0$, hence the only eigenvalue of $I+N^{-t}_{x,\star}(v_{\star})$ is equal to $1$.
It readily follows that the Jacobian is constant along $W^{u}(x)$.
\end{proof}

Fix some $x\in M$. Consider the group of subresonant polynomials over $\R^n$, $G^{\chi}_{x}$.
This group contains a proper subgroup consists of subresonant polynomials where the action is the identity over the least component. We denote this subgroup by $G^{\chi,>1}$.

\begin{defn}
We say that $\mu$ is a generalized $u$-Gibbs state if for almost every $x\in M$, for every $S\in G^{\chi,>1}$, its associated map $\tilde{S}$ satisfy
\begin{equation*}
    d_{W}\left(\tilde{S}.\mathbf{f}_1(x),\mathbf{f}_1(x)\right)=0.
\end{equation*}
\end{defn}

Under this interpretation, the claim that $\mathbf{f}_{1}$ is absolutely continuous along $W^{>1}_{\text{loc}}$ is equivalent to saying that $f_{1}$ is invariant under translations along the subspaces $\oplus_{i>1}E^{i}(x)$ as identified in $\mathbb{R}^{n}$ by the map $\psi_{x}$.

This can be seen easily as the proportion constant between $\tilde{S}.\mathbf{f}_1$ and $\mathbf{f}_1$, for a fixed $\tilde{S}$, gives rise to a multiplicative cocycle into $\R$.
Using Poincare recurrence, and the fact that the unstable lamination is uniformly contracted by $g_{-t}$, we get that this cocycle is trivial. 

\begin{lem}
This interpretation of $\mathbf{f}_{1}$ does not depend on the choice of normal forms coordinates used, in the case $\mu$ is a generalized $u$-Gibbs state with respect of $E^{>1}$.
\end{lem}
\begin{proof}
By definition of $\mathcal{H}_{x}$ as in~\eqref{eq:definition-of-unstable-param} and the interpretation of normal forms coordinates, a different coordinate structure amounts to considering a post-composition of $\psi_{x}$ with a map $\eta_{x}\in G_{x}^{\chi}$, which induces an isometry over $\mathbb{R}^{n}$.
Hence the Jacobian of $\eta_{x}$ at the identity equals to $1$.
Furthermore, considering the finer structure of $\eta_{x}$ arising from the sub-resonance relations~\eqref{eq:sub-resonance} (see also~\eqref{eq:G-subresonant-group}), we see that it fixes the subspace corresponding to $E^{1}$.
Over $E^{>1}$, we have that $\psi_{x}.\mu_{x}^{u}$ is Lebesgue, hence the related Radon-Nykodim derivative equal to $1$ there as well. 
\end{proof}
\begin{obs}
As $E^{u}(x)/E^{uu}(x)$ is a one-dimensional $g_{t}$-equivarient bundle, the $g_{t}$ action induces a linear cocycle over it upon a choice of a trivialization. Concretely the cocycle which we will denote by $\mathcal{R}$ is defined by considering a unit vector $v\in E^{u}(x)/E^{uu}(x)$ and defining
\begin{equation}\label{eq:R-defn}
\mathcal{R}(g_{t}.x) = \lVert g_{t}.v  \rVert,
\end{equation}
where the norm is the Pesin norm related to the Osceldets splitting.
In view of the normal forms coordinates as in \S\ref{subsec:normal-forms}, this cocycle $\mathcal{R}$ measures exactly the expansion along the second coordinate in \eqref{eq:explicit-sub-res},\eqref{eq:g-action-normal-forms}.
With this interpretation, we may relate $\mathbf{f}_{1}$ to $\mathcal{R}$ by:

\begin{equation*}
    \mathbf{f}_{1}(g_{t}.x) = \mathcal{R}(g_{t}.x) \cdot (g_{t}.\mathbf{f}_{1}(x)).
\end{equation*}
\end{obs}

Fix a generic Pesin point $x\in M$. Consider some other generic Pesin point $x'\in \mathcal{B}_{0}(x)$.
In what proceeds we will be working in $\psi_{x}(\mathcal{B}_{0}(x))\subset\mathbb{R}^{n}$.
We are interested in measuring the Hausdorff distance between two copies of the subspace $E^{uu}$, one cased at the origin (amounting to $W^{uu}_{\text{loc}}(x)$) and one copy translated by a vector $v_{x'}$, amounting to $W^{uu}_{\text{loc}}(x')$.
The vector $v_{x'}$ can be calculated by $\mathsf{P}_{x,x'}(0,0)$ for some sub-resonant polynomial $\mathsf{P}_{x,x'}$.
As we are interested in the Hausdorff distance between the two subspaces, it is enough to consider the projection of $v_{x'}$ to the $E^{\lambda_1}$-subspace.
Applying the $g_{t}$ map as in~\eqref{eq:g-action-normal-forms}, we see that $\mathcal{R}$ exactly measures this distance.

\subsection{A result about avoiding subspaces}\label{subsec:avoiding}
The following Lemma follow essentially verbatim from~\cite[\S5]{eskin-mirzakhani} (see also \cite[\S8]{eskin_lindenstrauss-long}).
As the proof is technical and relays on several auxiliary lemmata, we give them in appendix~\ref{app:avoiding-tech}.

\texttt{Standing assumptions:}In what follows, let $V$ be a measurable vector bundle over $M$, $F:~M\to~V$ a measurable section.
In particular, for almost every $q\in M$ it defines a map $F_{q}:W^{s}\to V(q)$.

The following avoidance lemma, which is valid in large generality, will be used during the proof of the Eskin-Mirzakhani scheme in order to show that one may find many points $q'$ on $W^{s}(q)$ for which an associated vector $F_{q}(q')$ in a specific line $F_{q}(W^{s}(q))$ in some vector bundle $V(q)$ avoids some ``problematic subspace'' denotes $\mathfrak{M}_{u}$.  
We assume that $V(q)$ contains $W^{s}(q)$ via an injective and isometric embedding, given by $F_{q}(W^{s}(q))$, which we will identify with the identity map.
\begin{lem}\label{lem:subspace-avoidance}
For any $\delta>0$ there exists constants $c_{1}(\delta),\epsilon_{1}(\delta)>0$ which tend to $0$ as $\delta$ tends to $0$ and $\rho(\delta),\rho'(\delta)>0$ such that for every subset $M'\subset M$ of measure larger than $1-\delta$ there exists a subset $M''\subset M$ of measure larger than $1-c_{1}(\delta)$ such that the following holds, for any $q\in M$ there exists a measurable map to proper subspaces of $V(q)$, denoted $u\mapsto \mathfrak{M}_{u}$.

Then for any $q\in M''$ there exists $s.q=q'\in W^{s}_{\text{loc}}(q)\cap M'$ such that 
\begin{equation*}
    \rho'(\delta)\leq dist(q,q') \ll 1,
\end{equation*}
and
\begin{equation}
    d(F_{q}(q'), \mathfrak{M}_{u}) > \rho(\delta).
\end{equation}
\end{lem}

\section{Factorization}\label{sec:factorization}
During this Section we will refer quite extensively the diagram in Figure~\ref{fig:8-pts}, pioneered by Eskin-Mirzakhani.
We will actually only handle the ``tilted'' half in this part, which we draw in the figure below:
\begin{figure}[H]
    \centering
    \scalebox{1}{
\begin{tikzpicture}

\coordinate (q) at (-2,0); 
\coordinate (q_1) at (-0.9,3);
\coordinate (q') at (2,0);
\coordinate (q'_1) at (0.9,3);
\coordinate (uq_1) at (-3.5,4);
\coordinate (q_2) at (-4,9);
\coordinate (q'_2) at (-2.5,9.5);
\coordinate (uq'_1) at (-4.5,5);
\coordinate (z) at (-2.7,4.5);


\draw plot [smooth] coordinates {(q) (q_1)};
\draw plot [smooth] coordinates {(q') (q'_1)};
\draw plot [smooth] coordinates{(uq_1) (q_2)};
\draw (uq_1) -- (q_2);

\draw[very thick] (q_1) .. controls +(up:0.2cm) and +(left:0.2cm) .. (uq_1) node[pos=0.5, below, sloped]{\smaller $W^{uu}(q_{1})$};
\draw[very thick] (q'_1) .. controls +(up:0.5cm) and +(right:0.3cm)  .. (uq'_1) node[pos=0.5, above, sloped]{\smaller $W^{uu}(q'_{1})$};


\draw[dashed] (q) -- (q') node[pos=0.5, below]{\tiny $O(1)$};
\draw[dashed] (q_1) -- (q'_1) node[pos=0.5, below=0.1 cm]{\tiny $O\left(e^{-\lambda_{C}\cdot\ell}\right)$};
\draw[dashed] (uq_1) -- (z);
\draw[dashed] (q_2) -- (q'_2) node[pos=0.5,above, sloped]{\tiny $O(\varepsilon)$};
\draw[dashed] (q'_2) -- (z);

\filldraw (q) circle (2pt) node[below]{$q$};
\filldraw (q') circle (2pt) node[below]{$q'$};
\filldraw (q_1) circle (2pt) node[below left]{$q_1$};
\filldraw (q'_1) circle (2pt) node[below right]{$q'_1$};
\filldraw (uq_1) circle (2pt) node[below left]{$u.q_1$};
\filldraw (q_2) circle (2pt) node[above]{$q_2$};
\filldraw[opacity=0.2] (q'_2) circle (2pt) node[above, opacity=.2,text opacity=0.5]{$q'_2$};
\end{tikzpicture}
}
    \caption{Illustration of the points in \S\ref{sec:factorization}}
    \label{fig:8-left-pts}
\end{figure}
We will try to familiarize the reader with the basics of the factorization technique in order to explain the motivation to the technical construction below.
The factorization technique of Eskin-Mirzakhani tries to study the drift developing by moving two stabely-related points along their fast unstables and applying the central-stable holonomy.
This drift, in general, will grow at exponential rate.
Then one flows the points forward until this drift grows to size $O(\delta)$, where $\delta$ is some predetermined macroscopic size. 

Referring to Figure~\ref{fig:8-left-pts} we have:
\begin{itemize}
    \item $q,q'$ are assumed to be \emph{stably} related, with distance $O(1)$ from each other.
    \item $q_{1},q'_{1}$ are defined by $g_{\ell}.q,g_{\ell}.q'$ respectively.
    \item $u.q_{1}\in W^{uu}(q_{1})$ is a point of distance $O(1)$ on the fast unstable leaf originating at $q_{1}$.
    \item $q_{2}=g_{t}.u.q_{1}$.
\end{itemize}

A crucial step is choosing the points in a way that all the points chosen are ``good points'' in terms of their dynamical features.
This is done by the introduction of an operator $A(q,u,\ell,t)$ measuring the divergence happening in Figure~\ref{fig:8-pts}.

We will use the following definition of local Hausdorff distance between two immersed submanifolds in our manifold.
\begin{defn}\label{defn:local-haus-distance}
Let $X,Y\subset M$ be two compact subsets.
We define the local Hausdorff distance at a point $p\in M$ between $X$ and $Y$, $hd_{p}(X,Y)$ as the Hausdorff distance between $B^{M}_{p}(\omega)\cap X$ and $B_{p}^{M}(\omega)\cap Y$, where $\omega$ is some small fixed constant which is smaller than the radius of injectivity of $M$.
\end{defn}

The goal of this section is to prove the following two technical factorization theorems:
\begin{thm}[Factorization at $t=0$]
\label{thm:factorization-t-0}
For any $\Lambda>0$ given, there exists a finite dimensional $g_{t}$-equivariant vector bundle $V$ over $M$ and a measurable equivariant section $F_{q}:W^{s}(q)\to V(q)$ and a family of linear maps $$\mathbf{A}(q_{1},u,\ell):V(q) \to \Rc(u.q_{1/2})\simeq~\mathbb{R}$$ and for every compact subset $K\subset M$, for every $\delta>0$ there exists a subset $K_{\delta}\subset K$ such that $\mu(K\setminus K_{\delta})<\delta$, constant 
$C=C(K_{\delta})$,
such that if all the points
$q$,$q'$,$q_{1}$,$q_{1}'$,$u.q_{1}$ belong to $K_{\delta}$, we get
\begin{equation}
\left\lvert hd_{(u.q_{1})}(W^{uu}_{\text{loc}}(u.q_{1}), W^{uu}_{\text{loc}}(q'_{1})) - \left\lVert \mathbf{A}(q_1,u,\ell).F_{q}(q')\right\rVert \right\rvert \leq C \cdot e^{-\Lambda}.
\end{equation}
Moreover, for every $q\in K_{\delta}$, the subspace
$\text{ess-span}\left\{\cup_{q'\in W^{s}_{\text{loc}}(q)\cap K_{\delta}}F_{q}(q')\right\}\leq V(q)$ is a measurable $g_t$-equivariant subspace.
\end{thm}

The above theorem ``morally'' says that one may find a linear operator $\mathbf{A}$ depending only on the left hand side of the diagram, namely $q_{1},u,\ell$ such that this operator measures approximately the deviation in the slow unstable direction after applying the central-stable holonomy.

\begin{rem}
The above theorem is valid without any non-integrability assumption.
\end{rem}

One may extend the construction above into ``future times'', namely the top left hand side of the diagram, as long as $t$ can is bounded linearly in terms of $\ell$.

\begin{thm}[Factorization]\label{thm:factorization}
For any $\beta>0$ given, there exists a finite dimensional $g_{t}$-equivariant vector bundle $V$ over $M$ and a measurable equivariant section $F_{q}:W^{s}(q)\to V(q)$ and a family of linear maps $$A(q_1,u,\ell,t):~V(q)~\to~\Rc(g_{t}.u.q_1)\simeq~\mathbb{R}$$ and for every compact subset $K\subset M$ , for every $\delta>0$ there exists a subset $K_{\delta}\subset K$ such that $\mu(K\setminus K_{\delta})<\delta$ and constants $C=C(K_{\delta}), \alpha=\alpha(K_{\delta})$ such that if all the points
$q$,$q'$,$q_{1}$,$q_{1}'$,$u.q_{1},g_{t}.u.q_{1}\in~K_{\delta}$ we have that for $t \leq \beta \cdot \ell$ we get
\begin{equation}
\left\lvert hd_{g_{t}.(u.q_{1})}(W^{uu}_{\text{loc}}(g_{t}.u.q_{1}), W^{uu}_{\text{loc}}(g_{t}.q'_{1})) - \left\lVert A(q_1,u,\ell,t).F_{q}(q')\right\rVert \right\rvert \leq C \cdot e^{-\alpha \cdot \ell}.
\end{equation}
Moreover, for every $q\in K_{\delta}$, the subspace
$\text{ess-span}\left\{\cup_{q'\in W^{s}_{\text{loc}}(q)\cap K_{\delta}}F_{q}(q')\right\}\leq V(q)$ is a measurable $g_t$-equivariant subspace.
\end{thm}
The bundle $\mathcal{R}$ is defined in~\eqref{eq:Q-R-construction} and essentially captures the cocycle $\mathcal{R}$ defined in~\eqref{eq:R-defn}.

\begin{rem}
It is not correct that the derived action is contracting, as one may easily see that the slow unstable direction should \emph{grow} under $g_{t}$ in future times.
Nevertheless, one clever idea in the construction is that as long as one imposes a bound of the form $t\leq \beta\cdot \ell$, one may ``extend'' the vector bundle $V(q)$ to approximate the distance more and more, so even after flowing for time $t\leq \beta\cdot\ell$, the approximation error is still small.
In particular, the bundles $V(q)$ and the maps $F_{q}(q')$ in the the above theorems may be different, where the later bundles and maps have larger dimensions.
\end{rem}

We will apply this theorem combined with an a-priori growth bound given by the \emph{quantitative non-integrability} assumption over our system.
The following definition plays a crucial role in the Eskin-Mirzakhani factorization scheme.
\begin{defn}
Fix some $q'\in W^{s}(q)$.
Pick some $\ell$ large.
Assume that $0<\varepsilon<1$ is chosen such that $\varepsilon>dist(g_{\ell}.q,g_{\ell}.q')$.
We define the \emph{stopping time} $\mathsf{T}_{1}=\mathsf{T}_{1}(q,q',u,\ell,t)$ as
\begin{equation*}
    \mathsf{T}_{1} = \sup\left\{ t\geq 0 \mid hd_{g_{t}.(u.q_{1})}(W^{uu}_{\text{loc}}(g_{t}.u.q_{1}), W^{uu}_{\text{loc}}(g_{t}.q'_{1})) \leq \varepsilon\right\}.
\end{equation*}
\end{defn}
\begin{lem}\label{lem:a-priori}
Assume $(M,g_{t})$ is \emph{quantitatively non-integrable} with $q_{1}, u.q_{1}, q'_1$ chosen as in Definition~\ref{def:QNI2} then there exists some $\beta=~\beta(M,g_{t})>~0$ such that the stopping time $\mathsf{T}_{1,(\epsilon)}$ satisfies 
$$\mathsf{T}_{1,(\epsilon)} \leq \beta\cdot\ell.$$
\end{lem}
This Lemma will be proved in \S\ref{sub:a-priori-stopping}.
Morally speaking, the above Lemma determines an upper bound for the time $t=\mathsf{T}_{1}$ such that the strong unstable $W^{uu}(g_{t}.u.q_1)$ and the central-stable projection of $W^{uu}(g_{t}.u.q'_{1})$ are $\epsilon$ separated.

Combining the above Lemma with the Theorem yields
\begin{cor}\label{cor:approximation-of-stopping-time}
Taking $\beta$ as in Lemma~\ref{lem:a-priori} and applying Theorem~\ref{thm:factorization} adapted to this $\beta$, we may calculate the stopping time $\mathsf{T}_{1,(\epsilon)}$, up to an error of $O(e^{-\alpha\cdot\ell})$ as
\begin{equation}
    \mathsf{T}_{1,(\epsilon)} = \sup\left\{ t\leq \beta\cdot\ell \  \bigg\vert \ \lVert A(q_{1},u,\ell,t).F_{q}(q') \rVert \leq \epsilon \right\},
\end{equation}
for $q_{1},u.q_{1}$ satisfying QNI.
\end{cor}

We note that the quantitative non-integrability play only a role in determining the value of $\beta$ in the above theorem, see subsection~\ref{sub:a-priori-stopping} for details and the proof of Lemma~\ref{lem:a-priori}.

It will be crucial for the construction of the operator $A(q_{1},u,\ell,t)$ that we will only be interested in a range of times $t$ which are linearly bounded in $\ell$.

The map $A$ plays a prominent role in the work of Eskin-Mirzakhani~\cite[Section~$6$]{eskin-mirzakhani} and Eskin-Lindenstrauss~\cite[Section~$\S 3$]{eskin_lindenstrauss},\cite[Section~$\S 4$]{eskin_lindenstrauss-long} as it allows one to conclude the stopping time for the growth of the divergence between the two points $q_{2}=g_{t}.(u.q_{1})$ and $q'_{2}=g_{t}.(u.q'_{1})$ in the eight points scheme.

In the setting of Eskin-Lindenstrauss~\cite{eskin_lindenstrauss}, this is easily understood as the map $A$ which essentially ``translates'' between coordinates in $T_{q_{2}}M$ and $T_{q_{2}'}M$ is easily defined by an appropriate Adjoint action (as we may identify the corresponding tangent spaces with the Lie algebra).
In the setting of Eskin-Mirzakhani, one can define such a map (albeit in a much more complicated manner) using the affine structure defined over the moduli space (c.f. \cite{eskin-mirzakhani}, \S6, Proposition~$6.11$, equation $(6.25)$).

The main difference in these situations compared to the one we have at hand is that \emph{the stable projection is smooth} in both situations, thanks to the additional algebraic structure in both settings. 


Unfortunately, one cannot recover a complete identification of those tangent spaces in the general class of systems we consider here (not even for Anosov systems).. Nevertheless, our construction bellow allows one to ``approximately match'' $W^{u}(q_{1})$ and $W^{u}(q'_{1})$ such that we may approximately project $W^{uu}_{\text{loc}}(q_{1})$ into $W^{u}(q_{1}')$ through a smooth map (indeed, polynomial) , and this unstable part of the space is the one getting expanded by the dynamics, which allows one to recover the stopping time estimates needed in order to apply the eight points scheme.

Moreover, this construction cannot lead to a choice of a point $u.q_{1}'$, as we may not choose $u$ on $W^{uu}_{\text{loc}}(q_{1}')$ in any reasonable (say smooth) fashion with respect to the choice of $u.q$.
The formulation of Theorem~\ref{thm:factorization} gives the advantage of ``ignoring'' the first direction of divergence (which will be in the direction $E^{uu}(g_{t}.q_{1}')$) and achieving control over the divergence in the second subspace.

The construction is fairly technical but can be summarized as follows:
Given a point $u.q_{1}\in W^{uu}_{\text{loc}}(q_{1})$ we approximate, by means of a power series expansion, the manifolds $W^{s}(u.q_{1})$ and $W^{u}(q_{1}')$, and upon appropriate truncation, we may calculate a point $z\in W^{u}(q_{1}')$ which is the approximate stable projection of $u.q_{1}$. This is the contents of subsection~\ref{subsec:approximate-stable}.

From the point $z$ one may calculate the the ``horizontal'' divergence distance of $g_{t}.z$ from $g_{t}.u.q_{1}'$ (this in general can be done by normal forms coordinates, see the beginning of subsection~\ref{subsection:A-construction}) and then one may show that this distance is the distance that grows, as the ``vertical distance'' (in the form of the approximation of the center-stable projection) does not grow and stay bounded by $O(dist(q_1,q_1'))$.

The above construction will generate data related to the various points $q,q', q_{1},q'_{1},u.q,z$.
The theorem is formulated in a way that all the dependence is over the non-tagged points (namely the left hand side of the diagram), except for the initial data at $q'$. Hence for example we may not use any data which is dynamically calculated at $q'_1$ or $z$, except from data which has been preallocated into $V(q)$ by $F_{q}(q')$.
In order to make usage of this data, we use certain polynomial approximations to holonomies which are discussed in subsection. So for the rest of the chapter we will assume the data needed at $q_{1}'$ and $z$ is available, and in subsection~\ref{sub:approx-holonomies} we will explain how to remedy the situation. We will denote this set of data by diamond ($\Diamond$).

\subsection{The half-way points}
Consider the diagram as in Figure~\ref{fig:8-pts}.
Our first deviation from the Eskin-Mirzakhani scheme is to base our construction of $A(q,u,\ell,t)$ over the half way points $q_{1/2},q'_{1/2}$ instead of $q_{1},q'_{1}$.

In the previous works implementing the factorization technique, one could use the holonomy in order to relate $u.q_{1}$ and $u.q'_{1}$. This is no longer true in the case we consider here.
Nevertheless, the following construction can be summarized in the following observation:
\begin{obs}
Let $g_{t}.q,g_{t}.q'$ be two \emph{exponentially close} stably-related points, say $dist(g_{t}.q,g_{t}.q')<e^{-\lambda_{c}\cdot t}$.
Then by flowing backwards by means of the $g_{t}$ flow for some fraction of $t$, say $t'$, the flow box \\ $W^{u}_{\text{loc}}(g_{t}.q)\cup~ \left(W^{cs}_{\text{loc}}(g_{t}.q)\cap B_{e^{-\lambda_{C}\cdot t}}(q)\right)$ is contained in $B_{e^{-\lambda'_{C}\cdot t}}(g_{t-t'}.q)$ for some $\lambda'_{C}=~\lambda'_{C}(\lambda_{C},t'/t)$.
\end{obs}
This observation allows us to effectively shrink all the distance to be exponentially small. As we will use power series approximations, working in exponential scale ensures us that errors stay small even after taking high order approximations.

Recall that $q_{1}=g_{\ell}.q$, $q_{1}'=g_{\ell}.q'$, with $dist(q_{1},q_{1}')$ being exponentially small, as $q,q'$ are stably-related.
Assume that $u.q_{1}\in W^{u}(q_{1})$ for some $u.q_{1}$ such that $dist(u.q_{1},q_{1})=O(1)$. Notice that $dist(u.q_{1}, z) \approx~ dist(q_{1},q_{1}')$ for any point $z$ which is approximately the center-stable projection of $u.q_{1}$ to $W^{u}(q_{1})$, but nevertheless $dist(u.q_{1},q_{1}), dist(z,q_{1}') \approx O(1)$.
We would like to make all those distances small (in exponential scale) in order to have exponential bounds in the distances.
In order to remedy it, we will work in the half-way points.
We define $q_{1/2}=~g_{\ell/2}.q, q_{1/2}'=~g_{\ell/2}.q'$.
This construction is illustrated in Figure~\ref{fig:half-way}.
\begin{figure}[H]
    \centering
    \scalebox{0.85}{\begin{tikzpicture}
    \coordinate (q_1/2) at (2,0); 
    \coordinate (q'_1/2) at (2,2);
    \coordinate (x) at (0,0); 
    \coordinate (uq'_1/2) at (-0.5,2.3);
    \coordinate (z) at (0.2,1.7);
    \coordinate (q_1) at (-3,4);
    \coordinate (q'_1) at (-3,5);
    \coordinate (uq_1) at (-8,3.7);
    \coordinate (uq'_1) at (-8.6,5.3);

    \draw [fill=gray, opacity=0.2] (3,-0.3) to [bend left] (-1,-0.3) to (-1,0.7) to [bend right] (3,0.7) to (3,-0.3) ;
    \node[right] at (3,-0.3) {\smaller $W^{u}(q_{1/2})$};

    \draw [fill=gray, opacity=0.2] (3,1.7) to [bend left] (-1,1.7) to (-1,2.7) to [bend right] (3,2.7) to (3,1.7) ;
    \node[right] at (3,1.7) {\smaller $W^{u}(q'_{1/2})$};

    \draw [fill=gray, opacity=0.2] (-2.7,3.6) to [bend left] (-9,3.6) to (-9,4.6) to [bend right] (-2.7,4.6) to (-2.7,3.6) ;
    \node[right] at (-2.7,3.6) {\smaller $W^{u}(q_{1})$};

    \draw [fill=gray, opacity=0.2] (-2.7,4.9) to [bend left] (-9,4.9) to (-9,5.8) to [bend right] (-2.7,5.8) to (-2.7,4.9) ;
    \node[right] at (-2.7,4.9) {\smaller $W^{u}(q'_{1})$};

    \filldraw (q_1/2) circle (2pt) node[right]{$q_{1/2}$};
    \filldraw (q'_1/2) circle (2pt) node[right]{$q'_{1/2}$};
    \filldraw (x) circle (2pt) node[below]{$x$};
    \filldraw (z) circle (2pt) node[below right]{$z$};
    \filldraw (q_1) circle (2pt) node[below]{$q_1$};
    \filldraw (q'_1) circle (2pt) node[above left]{$q'_1$};
    \filldraw (uq_1) circle (2pt) node[ left]{$u.q_1$};
    \filldraw (-8,4.7) circle (2pt);

    \draw[loosely dotted] (q_1) to (q_1/2);
    \draw[loosely dotted] (q'_1) to (q'_1/2);
    \draw[loosely dotted] (uq_1) to (x);
    \draw[loosely dotted] (uq'_1/2) to (uq'_1);
    \draw[very thick] (q_1/2) to [bend left] (x);
    \draw[very thick] (q'_1/2) to [out=190, in=340](uq'_1/2);
    \draw[very thick] (q'_1) to [out=210, in=340](uq'_1);
    \draw[very thick] (q_1) to [bend left] (uq_1);

    \end{tikzpicture}}
    \caption{Change from $q_{1}, q'_{1}$ to the half-way points $q_{1/2}, q'_{1/2}$.}
    \label{fig:half-way}
\end{figure}
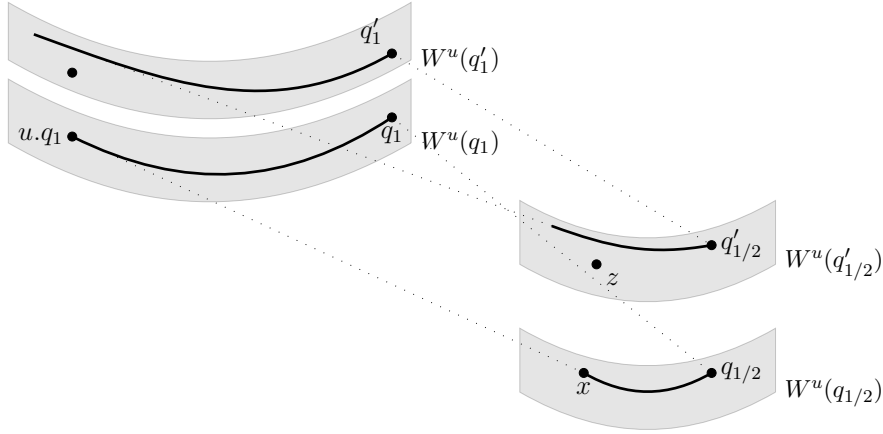

We note here that the midway points are still \emph{exponentially close}. More concretely, as $q,q'$ are stably-related, we have an exponent $\lambda_{C}<0$ as $dist(g_{t}.q,g_{t}.q') \ll_{q,q',\epsilon} e^{(\lambda_{C}+\epsilon)\cdot t}$ we get that $$dist(q_{1/2},q'_{1/2})\ll_{q,q',\epsilon} e^{0.5(\lambda_{C}+\epsilon)\cdot \ell}.$$

Furthermore, as $u.q_{1}\in W^{uu}_{\text{loc}}(q_{1})$ we get
$ dist(g_{-t}.u.q_{1},g_{-t}.q_{1}) \ll_{q_{1},u.q_{1},\epsilon}~e^{-(\lambda_{2}-\epsilon)\cdot t}$, in particular for $t=\ell/2$ we get
$$ dist(g_{-\ell/2}.u.q_{1}, q_{1/2}) \ll_{q_{1},u.q_{1},\epsilon} e^{-0.5(\lambda_{2}-\epsilon)\cdot \ell}. $$

Now assume $z$ is the stable projection of $x$ to $W^{u}(q_{1/2}')$ then we have $dist(x,z)\approx dist(q_{1/2},q_{1/2}') = O(0.5e^{(\lambda_{C}+\epsilon)\cdot\ell}).$
Moreover, $dist(z,q_{1/2}')$ can be approximated by the triangle inequality to be bounded by 
\begin{equation*}
    \begin{split}
        O(dist(x,z)+dist(x,q_{1/2})+dist(q_{1/2},q_{1/2}')) &= O\left(e^{-(\lambda_{2}-\epsilon)\cdot\ell/2}+e^{-(\lambda_{C}+\epsilon)\cdot\ell/2}\right) \\
        &= O\left(e^{-0.5(\lambda'-\epsilon)\cdot \ell} \right),
    \end{split}
\end{equation*}
for $\lambda'=\min\left\{\lambda_{2},-\lambda_{C}\right\}>0$.

\subsection{A-priori bound over the stopping time}\label{sub:a-priori-stopping}

Recall we under the assumption that $(M,g_{t})$ satisfies the quantitative non-integrability condition.
As such, the dynamical quadrilateral $\mathcal{Q}$ formed from $q_{1/2}$,$q'_{1/2}$ and $x$ satisfies that the distance between the stable projection of $x$, $W^{s}(x)\cap~W^{u}(q_{1/2}')$ and $W^{uu}_{\text{loc}}(q_{1/2}')$ is bounded from below by $dist(q_{1/2},q_{1/2}')^{\mathfrak{w}}$.
Using the notations of Definition~\ref{def:QNI2}, we require $x\in U_{q'_{1/2}}$ and $q'_{1/2}\in~S_{q_1/2}$.
When applying the $g_{t}$ dynamics, this distance grows at-most by a factor of $e^{\lambda_{1}\cdot t}$, for some expansion constant $\lambda_{1}$ depending on the Lyapunov spectrum. Hence as we want to bound from above the stopping time $\tau_{1}$, we get the a-prior bound - 
\begin{equation}
    e^{(\lambda_{1}-\epsilon)\cdot \tau_{1}}\cdot dist(q_{1/2},q_{1/2}')^{^{\mathfrak{w}}} \leq 1.
\end{equation}
Recall that $ dist(q_{1/2},q_{1/2}') \leq D\cdot  e^{0.5(\lambda_{C}+\epsilon)\cdot\ell}\cdot dist(q,q')$, for some constant $D$ (which can be made uniform over sets of arbitrarily large measure). Hence we get the a-priori bound 
\begin{equation}\label{eq:a-priori-bound}
    \tau_{1} \leq \frac{-0.5\cdot (\lambda_{C}+\epsilon)\cdot \ell}{\lambda_{1}-\epsilon} \leq 4\cdot\frac{\left\lvert \lambda_{C} \right\rvert}{\lambda_{1}} \cdot \ell,
\end{equation}
where the last inequality is true for a choice of small enough $\epsilon$, for example $\epsilon<\lambda'/2$.

Picking for example 
$\beta = 8\cdot \frac{\left\lvert \lambda_{\max\text{ con}} \right\rvert}{\lambda_{1}} $, where $\lambda_{\max\text{ con}}$ stands for the most negative Lyapunov exponent of $(M,g_{t},\mu)$,
to be the constant given in the factorization Theorem~\ref{thm:factorization}, concludes the demonstration of Lemma~\ref{lem:a-priori}. 

\begin{rem}
In the second definition of QNI, this estimate also follows easily. Consider \eqref{eq:QNI-full-def}, applying $g_{t}$, we see that the distance grows by an exponential factor of $e^{(\lambda_{1}\pm \epsilon)\cdot t}$. 
Moreover, for $q_{1/2},q'_{1/2} \in \mathcal{L}$ we see that for all $k>k_0$ we have
$dist(g_{t}.q_{1/2},g_{t}.q'_{1/2})\geq C\cdot e^{\alpha\cdot k}\cdot e^{(\lambda_{1}\pm\epsilon)\cdot k}.$
Using the same analysis as above, we see that we may bound $\tau_{1}$ by $\tau_{1}\leq \frac{\alpha\cdot \lambda_{C}}{2\cdot \lambda_{1}}\ell$.
\end{rem}

\subsection{Approximating the intersection point $z$}\label{subsec:approximate-stable}
For the reminder of the paper, we choose a (measurable) $g_{t}$-equivariant trivialization of the tangent bundle $TM$, which exists $\mu$-almost everywhere.

Consider $x=~u.q_{1/2}$ for some $u.q_{1/2}\in W^{uu}_{\text{loc}}(q_{1/2})$.
We have the smooth curve $W^{s}(x)$.
Expand this curve in a Taylor polynomial, 
\begin{equation}\label{eq:stable-expansion}
    W^{s}(x) = P_{x,N}(\overline t) + R_{x,N}(\overline t),
\end{equation} with $\deg P_{x,N} \leq N$ being the Taylor expansion of $W^{s}(x)$ of order $N$ and we have the reminder estimate  $R_{x,N}(\overline t) \leq O(\lVert \overline t\rVert ^{N+1})$.
Furthermore, we have the smooth immersed manifold $W^{u}(q'_{1/2})$, which we expand as well
\begin{equation}\label{eq:unstable-exp}
    W^{u}(q_{1/2}') = P_{q'_{1/2},M}(\overline s) + R_{q'_{1/2},M}(\overline s)    \tag{\(\Diamond\)} 
\end{equation}
with $\deg P_{q'_{1/2},M} \leq M$ being the Taylor expansion of $W^{u}(q'_{1/2})$ of order $M$ and we have the reminder estimate  $R_{q'_{1/2},M}(\overline s) \leq O(\lVert \overline s\rVert ^{M+1})$.

We choose $N$ such that $$\sup\{\lvert R_{x,N}(\overline t) \rvert \mid \lVert\overline t \rVert \leq O(dist(q_{1/2},q'_{1/2})\} \ll e^{-\alpha \cdot \ell}$$ for our given $\ell,\alpha$.
Now we choose $M$ such that $$\sup\{\lvert R_{q'_{1/2},M}(\overline s) \rvert \mid \lVert\overline s \rVert \leq O(dist(q_{1/2},q'_{1/2})\} \ll e^{-\alpha \cdot \ell} \cdot e^{-C\cdot \ell}$$ for our given $\ell,\alpha$ as well, where $C$ is chosen such that $\lambda_{1}\cdot \beta \leq C$.

We solve for $\overline t,\overline s$ such that 
\begin{equation*}
    \lVert P_{x,N}\left(\overline t\right) - P_{q'_{1/2},M}\left(\overline s\right) \rVert \ll e^{-\alpha \cdot \ell}.
\end{equation*}

This gives a point in space which up to a modification of size $O(e^{-\alpha \cdot \ell})$ belongs to $W^{cs}(x)\cap W^{u}(q'_{1/2})$. We modify if needed and we call this point $z$, see Figure~\ref{fig:z-construction}.

\begin{figure}[H]
    \centering
    \scalebox{1}{\begin{tikzpicture}
    \coordinate (q_1/2) at (5,0); 
    \coordinate (q'_1/2) at (5,4);
    \coordinate (x) at (0,0); 
    \coordinate (z) at (0.5,3.5);
    \coordinate (uq'_1/2) at (0.4,4.5);

    \draw [fill=gray, opacity=0.2] (6,3.5) to [bend left] (-1,3) to (-0.8,6) to [bend right] (6,6) to (6,3.5) ;
    \node[right] at (6,3.5) {\smaller $W^{u}(q'_{1/2})$};
    
    \draw [fill=gray,opacity=0.2] (6,-0.5) to [bend left] (-1,-1) to (-0.8,2) to [bend right] (6,2) to (6,-0.5) ;
    \node[right] at (6,-0.5) {\smaller $W^{u}(q_{1/2})$};
    
    \draw[very thick] plot [smooth] coordinates {(q_1/2) (x)};
    \node[below] at (3,0) {\smaller $W^{uu}_{loc}\left(q_{1/2}\right)$};

    \draw[very thick] (q'_1/2) to [bend left] (uq'_1/2) ;
    \node[below] at (3,3.5) {\smaller $W^{uu}_{loc}\left(q'_{1/2}\right)$};

    \draw plot [smooth] coordinates {(x) (z)};
    
    \filldraw (q_1/2) circle (2pt) node[right]{$q_{1/2}$};
    \filldraw (q'_1/2) circle (2pt) node[right]{$q'_{1/2}$};
    \filldraw (x) circle (2pt) node[below]{$x$};
    \filldraw (z) circle (2pt) node[above left]{$z$};
    \draw[dashed] (q_1/2) to (q'_1/2);
    
    \end{tikzpicture}}
    \caption{Illustration of $z$.}
    \label{fig:z-construction}
\end{figure}
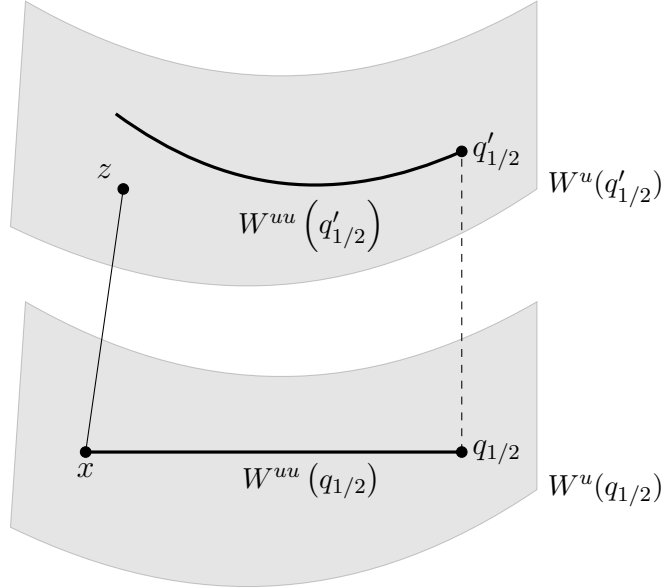

\subsection{Construction and matching of flags at $z$}\label{subsec:matching-flags}
The next subsections deal with the actual measurement of the approximation.
Notice that the point $z$ is calculated as an approximate intersection of $W^{cs}(u.q_{1/2})$ and $W^{u}(q'_{1/2})$. Notice that the coordinate systems where we calculate the point $z$ is different.
In order to conclude invariance in $W^{u}(z)$, we will need to measure its deviation from $W^{uu}(z)$ (in terms of normal form coordinates). Moreover, we will need to do so in terms of $q$-data. 
In order to do so, we will translate the $z$ points coordinates from $W^{u}(z)$ to $W^{cs}(u.q)$ and then later to $W^{u}(u.q)$.

In order to do so, we will need to consider the data encoding the forward flag at $z$ (which is smooth along $W^{u}(z)=W^{u}(q'_{1/2})$ and the backwards flag along $W^{cs}(u.q)$.

We will also need the following theorem due to D. Ruelle:
\begin{thm}[\cite{ruelle1979ergodic} Theorem~$6.3.(b)$]\label{thm:smoothness-leaves}
The forward flag $\tilde{\mathcal{Q}}(\star)$ is \emph{smooth} along \emph{stable} leaves. The backward flag $\tilde{\mathcal{R}}(\star)$ is \emph{smooth} along \emph{unstable} leaves.
\end{thm}

Consider the point $z\in W^{cs}(x)$. By Ruelle's theorem, the map $z\mapsto \tilde{Q}(z)$ is a smooth map along $W^{cs}(x)$. We denote this map by 
\begin{equation}\label{eq:Qc-def}
T^{cs}_{x}(z)=\tilde{\Qc}(z).
\end{equation}
Similarly the map $z\mapsto \tilde{R}(z)$ is a smooth map along $W^{u}(q'_{1/2})$.
We denote this map as 
\begin{equation}\label{eq:R-definition}
    T^{u}_{q_{1/2}}(z)=\tilde{\Rc}(z).\tag{\(\Diamond\Diamond\)} 
\end{equation}

where we identify the subspaces according to our chosen trivialization.

As $x,q_{1/2}'$ are chosen in a good set with respect to Oseledets theorem, and the measurable section from the manifold to the Grassmannian bundle matching each point with its Oseledets splitting is Holder continuous, we have that there exist some $\zeta>0$
\begin{equation}
    dist(\tilde{\Qc}(x),\tilde{\Qc}(q'_{1/2})) ,dist(\tilde{\Rc}(x),\tilde{\Rc}(q'_{1/2})) \ll dist(x,q'_{1/2})^{\zeta},
\end{equation}
by \cite[Theorem~$A$]{filip}.

Using the smoothness of the the forward flag along stable leaves we get:
\begin{cor}\label{cor:Q-close}
For any $t\geq 0$ we have that
\begin{equation}
    dist(\Qc(g_{t}.x),\Qc(g_{t}.z)) \ll dist(g_{t}.x,g_{t}.z)^{\zeta} \ll e^{-\lambda_{C}\cdot\zeta\cdot t}\cdot dist(x,z)^{\zeta}.
\end{equation}
\end{cor}
As $x,z$ are central-stably related, we can apply local structure coordinates around $g_{t}.x,g_{t}.z$ for $t\geq 0$. 
By the definition of $\mathcal{Q}$ one may ignore the central direction.
The rest follows at once by the smoothness of the entire flag.
We note that by the definition of $\tilde{\mathcal{Q}}$.

We also have
\begin{obs}
$dist(\tilde{\Rc}(x),\tilde{\Rc}(z)) \ll dist(x,z)^{\zeta}.$
\end{obs}
This estimate follows by the triangle inequality as follows
\begin{align*}
    dist(\tilde{\Rc}(x),\tilde{\Rc}(z)) &\leq dist(\tilde{\Rc}(x),\tilde{\Rc}(q'_{1/2})) + dist(\tilde{\Rc}(q'_{1/2}),\tilde{\Rc}(z)) \\
    &\ll dist(x,q'_{1/2})^{\zeta} + dist(q'_{1/2},z)^{\zeta} \\
    &\ll dist(x,z)^{\zeta}.
\end{align*}

Now we arrive to the main technical Lemma, which allows us to control the error between the backwards flags along $g_{t}.x$ and $g_{t}.z$ and in particular show its error grows at-most polynomially in their distance
\begin{lem}\label{lem:R-close}
For any $t\geq 0$ we have that 
\begin{equation*}
    dist(\tilde{\Rc}(g_{t}.x),\tilde{\Rc}(g_{t}.z)) \ll dist(g_{t}.x,g_{t}.z)^{\zeta}.
\end{equation*}
\end{lem}

For the proof of this Lemma, we will need the following two technical Lemmata, essentially due to Brin~\cite{brin1995ergodicity,katok2001smooth, filip}.
\begin{lem}\label{lem:bound-over-derivative}
Let $g_{1}:M\to M$ be a smooth map, where $M$ is a compact manifold.
Let $d_\star g^{1}$ denotes the associated derivative map of $g_{1}$.
For any $a>\max_{p\in M}\lVert d_{p}g_{1} \rVert ^{1+\zeta}$ there exists some $D>1$ such that for every $n\in \mathbb{N}$ and every $p_{1},p_{2}\in M$ we have
\begin{equation*}
    \lVert d_{p_1}g_{1}^{n} - d_{p_2}g_{1}^{n} \rVert \leq D\cdot a^{n}\cdot dist(p_{1},p_{2})^{\zeta}.
\end{equation*}
\end{lem}
For its proof see~\cite[Lemma~$A.2$]{katok2001smooth}. We remark that by appropriately renormalizing time, the proof of the Lemma works verbatim in the continuous time settings as well.
\begin{lem}\label{lem:lin-alg-growth}
Assume that $\mathbb{R}^{n}=E\oplus E'$ and $\left\{ A_{n}\right\}_{n\in \mathbb{N}},\left\{ B_{n}\right\}_{n\in \mathbb{N}}$ are two sequences of invertible linear operators such that for all $v\in E$
\begin{equation*}
    \lVert A_{n}.v \rVert \geq e^{(\lambda_{1}-\varepsilon)n}\cdot \lVert v \rVert,
\end{equation*}
and for all $v\in E'$
\begin{equation*}
    \lVert A_{n}.v \rVert \leq e^{(\lambda_{1}-2\varepsilon)n}\cdot \lVert v \rVert,
\end{equation*}
for some $\varepsilon>0$, and
\begin{equation*}
    \lVert A_{n}-B_{n} \rVert \ll e^{(\lambda_{1}-2\varepsilon)n}.
\end{equation*}
Furthermore, assume that $dist(E,F)\ll 1$ for some linear subspace $F\subset \mathbb{R}^n$ then
\begin{equation}
    dist(A_{n}.E,B_{n}.F) \ll 1.
\end{equation}
\end{lem}
\begin{proof}

Pick $v\in F$, We may write $v=v_{1}+v_{2}$ with $v_{1}\in E, v_{2}\in E'$.
By our assumption, $\lVert v_{1} \rVert \geq C\cdot \lVert v \rVert$.
Now we do the following calculation
\begin{equation*}
\begin{split}
    B_{n}.v &= A_{n}.v + (B_{n}-A_{n}).v \\
    &= A_{n}.v_{1} + (A_{n}.v_{2}+(B_{n}-A_{n}).v).
\end{split}
\end{equation*}
Note that $A_{n}.v_{1}\in A_{n}.E$ and
\begin{equation*}
\begin{split}
    \lVert A_{n}.v_{1} \rVert &\geq e^{(\lambda_{1}-\varepsilon)n}\lVert v_{1} \rVert \\
    &\geq C\cdot e^{(\lambda_{1}-\varepsilon)n}\lVert v \rVert.
\end{split}
\end{equation*}
Furthermore, we have that
\begin{equation*}
    \lVert A_{n}.v_{2} + (B_n-A_n).v \rVert \ll e^{(\lambda_{1}-2\varepsilon).n}\cdot \lVert v \rVert.
\end{equation*}
Hence by normalizing the projection of $B_{n}.v$ over $A_{n}.E$ we get that
\begin{equation*}
    dist(B_{n}.v , A_{n}.E) \ll e^{-\varepsilon\cdot n}.
\end{equation*}
\end{proof}

Now we may complete the proof of Lemma~\ref{lem:R-close}.
Using Lemma~\ref{lem:bound-over-derivative}, there exists some $a=a(M,g_{t})$ and $\beta=\beta(M,g_{t})>0$ such that for all $x,z\in M$
\begin{equation*}
\lVert d_{x}g_{t} - d_{z}g_{t} \rVert \leq a^{t}\cdot dist(x,z)^{\zeta}.
\end{equation*}
Specializing to our chosen $x,z$, assuming $x,z$ are close enough (which is $O(dist(q_{1/2},q'_{1/2}))$), we have that 
\begin{equation*}
\lVert d_{x}g_{t} - d_{z}g_{t} \rVert \leq e^{(\lambda_{1}-\varepsilon) t}    
\end{equation*}
for any $t\ll_{dist(x,z)} 1$, by Lemma~\ref{lem:bound-over-derivative}. We note we may pick the upper bound for $t$ such that $t$ is arbitrarily large (larger than any fixed predefined value), which will allow us to assume that $g_{t}.x $ is a good-Oseledets point for some $t$ in the range where this estimate is valid as $x$ was defined in the good set.

Applying Lemma~\ref{lem:lin-alg-growth} we get that $dist(\tilde{\mathcal{R}}(g_{t}.x),\tilde{\mathcal{R}}(g_{t}.z)) \ll 1$.
Repeating the argument with $g_{t}.x,g_{t}.z$ (where we have \\ $dist(g_{t}.x,g_{t}.z)\leq~dist(x,z)$), we may get this estimate for any $t$.

Furthermore, we define the following vector bundles (at $x$, and in general at any Oseledets good point)
\begin{equation}\label{eq:Q-R-construction}
    \Qc(x) = E^{\leq\lambda_{1}}(x)/E^{<\lambda_1}, \Rc(x)=E^{\geq \lambda_{1}}(x)/E^{>\lambda_{1}}.
\end{equation}
We notice that $\Qc$ and $\Rc$ are sub-quotients of the forwards and backwards flags, respectively.
The construction clearly extends to the point $z$ via the endowment of the flags given in the beginning of this subsection.

We notice that both $\Qc$ and $\Rc$ are $1$-dimension vector bundles. At any Osceldets bi-regular point $x\in M$, we get an isomorphism

\begin{equation*}
    \Qc(x) \simeq E^{\lambda_1}(x) \simeq \Rc(x)
\end{equation*}

As a result, by composing thhe first isomorphism with the other, we may define the set of functions 
\begin{equation}\label{eq:I-def}
    I_{*}^{t} : \Qc(\star) \to \Rc(\star)
\end{equation} as members of $Hom(\Qc(g_{t}.\star),\Rc(g_{t}.\star))$ for $\star$ being either $x$ or $z$ as the function identifying the appropriate sub-spaces.

In principle, $I$ is a scalar comparing the norms of those one dimensional bundles.

For $\star =x$ and $x$ was assumed to be a regular point, the Oseledts subspaces are transverse for $g_{t}.x$ in a good set, hence there exists a family of invertible maps $I^{t}_{x}:\Qc(x)\to \Rc(x)$ and $0<c<C$ such that for any $t$ in a good set, $c<\lvert I_{x}^{t} \rvert <C$ and $g_{t}.I_{x} = I_{x}^{t}$.

This provides an \emph{identification} between the two subspaces by means of the functions $I_{\star}^{0}$ for a subset of $M$ of arbitrarily large measure.

For $\star = z$, for the same set of times $t$, using Corollary~\ref{cor:Q-close} and the technical Lemma~\ref{lem:R-close} we have that for some (possibly smaller or larger) $c'<C'$ we have that
$c' < \lvert I_{z}^{t} \rvert < C'$ and similarly $g_{t}.I_{z}=I_{z}^{t}$.

\begin{rem}
During the course of the proof, while using Lemma~\ref{lem:bound-over-derivative}, as we may enlarge $\ell$ if we wish, we can effectively shrink $dist(q_{1/2},q_{1/2}')$ hence making $\lVert d_{x}g_{t} - d_{z}g_{t} \rVert$ arbitrarily small.
\end{rem}

\subsection{Construction of a transfer function}\label{subsection:A-construction}
While ideally one would like to work with normal forms coordinates~\cite{kalinin} in order to calculate the distance between $z$ and $W^{uu}_{\text{loc}}(q'_{1/2})$ inside the unstable manifold $W^{u}(q'_{1/2})$, due to the simplicity of the situation we have in hand, we may overcome this difficulty by working with the vector bundle $\Qc$ (note that due to the sub-resonance condition, the last expression in the normal forms coordinates is indeed a single scalar).

We will assume now that $x,z$ are stably related, later we will indicate how to modify the construction in the case where $x,z$ are center stably related.

For $x,z$ we define the following functions:
\begin{equation*}
    \varphi_{x}^{t} = \lVert g_{t}.v \rVert_{\Qc(x)}/\lVert v \rVert_{\Qc(x)}, \varphi_{z}^{t} = \lVert g_{t}.u \rVert_{\Qc(z)}/\lVert u \rVert_{\Qc(z)},
\end{equation*}
for any non-zero $v,u$ in $\Qc(x),\Qc(z)$ respectively. It is easy to verify that $\varphi$ is independent from the choice of $v,u$.

\begin{lem}\label{lem:holonomy-L}
The limit
$$L(x,z)=\lim_{t\to\infty} \frac{\varphi_{z}^{t}}{\varphi_{x}^{t}}$$ exists.
\end{lem}
We note that $L(x,z)$ is a number, but we think of this number as a member of $Hom(\Qc(x),\Qc(z))\simeq GL_{1}(\mathbb{R})$.

\begin{proof}
We have the following observation
$$ \lim_{T\to \infty} \frac{\phi_{g_{T}.z}^{1}}{\phi_{g_{T}.x}^{1}} = 1,$$
as $dist(g_{T}.x,g_{T}.z) \to 0$ and the function is continuous over the stable manifold.
Furthermore, as the function is Holder continuous we have that
$$ $$
$$ \left\lvert \frac{\phi_{g_{T}.z}^{1}}{\phi_{g_{T}.x}^{1}} -1 \right\rvert \leq \frac{C\cdot dist(g_{T}.x,g_{T}.z)^{\alpha}}{\phi_{g_{T}.x}^{1}} \ll dist(g_{T}.x,g_{T}.z)^{\alpha} ,$$
where $C,\alpha$ are uniform over $W^{s}(x)=W^{s}(z)$ and $\min \phi^{1}_{g_{T}.x}$ is bounded away from zero over the orbit $g_{\star}.x \subset g_{\star}.W^{uu}(q_{1/2})$.

We will consider $\log( \phi_{z}^{t}/\phi_{x}^{t})$. Fix some $t_{1}$ large and let $t_{2}=t_{1}+1$.
We have by the cocycle property that
\begin{equation}
    \begin{split}
       \left\lvert \log( \phi_{z}^{t_2}/\phi_{x}^{t_2}) - \log( \phi_{z}^{t_1}/\phi_{x}^{t_1})  \right\rvert &= \left\lvert \log( \phi_{g_{t_{1}.z}}^{t_2-t_1}) - \log( \phi_{g_{t_{1}}.x}^{t_2-t_1})  \right\rvert \\
&= \log\left(\frac{\phi_{g_{t_{1}.z}}^{t_2-t_1}}{\phi_{g_{t_{1}}.x}^{t_2-t_1}}\right) \\
&\ll dist(g_{t_{1}}.x,g_{t_{1}}.z)^{\alpha} \\
&\ll e^{-\lambda_{C}\cdot t_{1}\cdot\alpha}. 
    \end{split}
\end{equation}
So we conclude that $\sum_{t_{1}}\log( \phi_{z}^{t_{1}+1}/\phi_{x}^{t_{1}+1})-\log( \phi_{z}^{t}/\phi_{x}^{t})$ is summable, hence we have that $\lim_{t_{1}\to \infty} \log( \phi_{z}^{t}/\phi_{x}^{t})$ exists, and so is $L(z,x)$.
\end{proof}
The map $L(x,z)$ is defined as a map from $\Qc(x)$ to $\Qc(z)$, but we are interested in a map from $\Rc(z)$ to $\Rc(x)$.
Using the map $I$ we can define a map from $\Rc(z)$ to $\Rc(x)$ by $I_{x}^{-1}\circ L^{-1}(x,z) \circ I_{z}$. 
Moreover, we still need to handle the situation where $x,z$ are only flow-stably related. In this case, we have that $x_1=g_s . x$ for some $s\ll dist(q_{1/2},q'_{1/2})$ is stably related to $z$, hence $L(x_1,z)$ is well defined. Moreover, we have that $\mathcal{Q}(g_{s}.x)=\mathcal{Q}(x)$.

\subsection{Simplified construction of $A(q,u,\ell,t)$ and $V(q)$}
For that subsection we assume we may calculate all the required data at the points $q,q',q_{1/2},q'_{1/2},x,z$.
The calculation of the data at $q'_{1/2}$ and $z$ is dependent (at-least partially, on the $\Diamond$ denoted parts) on the existence of holonomies which will be discussed on the following subsection.

We may define $A$ as follows - 
given $q_{1},\ell,u$ one calculate $q_{1/2}$ by 
$$ q_{1/2}=g_{-\ell/2}.q_{1} $$
then one sets $x=u.q_{1/2}$.
Now one defines $V^{s-exp}$ as the vector bundle which captures all the Taylor coefficients in the expansion ~\eqref{eq:stable-expansion} of the stable manifold up to the required error at $x$, including the location of $x$ (in terms of some global embedding of $M$ into $\mathbb{R}^{n}$ say). We note that the dimension of this vector bundle (which is related to the number of derivative) can be computed a-priori and beforehand.
Furthermore, for $q'$ one may construct a similar vector bundle $V^{\text{u-exp}}.$
Given $u, q_{1/2}, q_{1/2}'$ and the data in $V^{\text{s-exp}}(u.q_{1/2}), V^{\text{u-exp}}(q_{1/2}')$, one may recover $z$ (up to the prescribed error). We define a function $$z=z(V^{\text{s-exp}}(u.q_{1/2}), V^{\text{u-exp}}(q_{1/2}'))$$ taking values in $W^{u}(q_{1/2}')$.
As $q_{1/2}, q_{1/2}'$ are related to $q,q'$ by the flow depending on the value of $\ell$, we may also write
$$z=z(u,\ell, V^{\text{s-exp}}(q), V^{\text{u-exp}}(q')).$$
We define $F_{q}(q')$ to be the assignment from $q'\in W^{s}(q)$ to \\ $V^{\text{s-exp}}(q)\oplus~V^{\text{u-exp}}(q')$.
So we actually have $$z=~z(u,\ell,F_{q}(q')).$$
This provides us with a map
\begin{equation*}
    \mathbf{A}(q_1,u):V(q_{1/2}) \to \Rc(u.q_{1/2}).
\end{equation*}
Now by precomposing with $g_{\ell/2}$, we may extend this map to be
\begin{equation*}
    \mathbf{A}(q_{1},u,\ell):V(q) \to \Rc(u.q_{1/2}),
\end{equation*}
which is evaluated over a section $F_{q}:W^{s}(q')\to V(q)$ containing the needed data from $q',q'_{1/2}$.
We note that at this stage, one needs to know only the base points $q,q'$, the time $\ell$ and the point $u.q_{1/2}$ in order to calculate all the approximations needed, where the computations themselves are done in the half-way points.

We eventually will extend this map for any $t\geq 0$ by:
\begin{equation}
    A(q_1,u,\ell,t).F_{q}(q') = g_{t+\ell/2}.\mathbf{A}(q_{1},u,\ell).F_{q}(q') \in \Rc(q_{3}).
\end{equation}
This equation measures the relevant Hausdorff distance, as by considering normal forms coordinates, $\Rc$ measures the distance away from $W^{uu}_{\text{loc}}$.

\subsection{Approximation of holonomies}\label{sub:approx-holonomies}
The following theorem of A. Brown, A. Eskin ,S. Filip and F. Rodriguez-Hertz, based on construction of cocycle normal forms, allows one embed bundle which are smooth along stables inside bundles which admits smooth stable holonomies.
\begin{thm}[Constructive approximation of holonomies~\cite{brown_eskin_filip}]\label{thm:constructive-holonomies}
Let $V$ be a $g_{t}$-equivariant vector bundle over $M$ which is smooth along stables.
There exists another $g_{t}$-equivariant vector bundle $V_{\text{hol}}$ over $M$ such that
\begin{enumerate}
    \item $V_{\text{hol}}$ \emph{admits holonomies} in the following sense - for any generic $x,y \in M$ which are of bounded distance, there exists a linear map $H(x,y):~V_{\text{hol}}(x)\to~V_{\text{hol}}(y)$ such that
    \begin{equation*}
        V_{\text{hol}}(g_{t}.y) = H(g_{t}.x,g_{t}.y) \circ (g_{t}^{x})_{\star} \circ H(y,x). V_{\text{hol}}(y).
    \end{equation*}
    \item For any generic point $x\in M$ there exists an \emph{injective affine map} $j_{x}:V(x) \to V_{\text{hol}}(x)$ which changes smoothly along $y\in W^{s}(x)$.
\end{enumerate}
\end{thm}

This theorem allows one to translate the set of derivatives in $q'$ to the set of derivatives in $q'_{1/2}$ needed in order to calculate the Taylor expansion in~\eqref{eq:unstable-exp}.

Moreover, as $q'$ is Oseledets good point, $\tilde{\Rc}(q')$ is well-defined and so using the map, one may recover $\tilde{\Rc}(q'_{1/2})$ as needed in~\eqref{eq:R-definition}.

Furthermore, in order to calculate $\tilde{\Rc}(z)$, one needs to calculate 
$T^{u}_{q'_{1/2}}(z)=\tilde{\Rc}(q'_{1/2})$.
As $\tilde{\Rc}(q'_{1/2})$ changes smoothly over $W^{u}(q'_{1/2})$ (using Ruelle's theorem~\ref{thm:smoothness-leaves}), we may expand $\tilde{\Rc}(q'_{1/2})$ by a Taylor expansion to a sufficient high degree in order to approximate $\tilde{\Rc}(z)$ up to a sufficiently small error (the error will be a power of $dist(z,q'_{1/2})\approx dist(q_{1/2},q'_{1/2})$ to some power, chosen such that the error will be small enough so when get expanded by $e^{\lambda_{1}\cdot\alpha\cdot\ell}$ the total error would still be bounded by $e^{-\beta \cdot \ell}$).

We note that the polynomial function one deduces from Theorem~\ref{thm:constructive-holonomies} does not interfere with the described proof, as by using composition, one still recover the appropriate Taylor approximations at the various stages, albeit maybe of higher degrees, hence one just need to enlarge the vector bundle $V(q)$ in order to account for higher degree derivatives, but the maximal degree can be in principle bounded a-priori before the starting of the computation.

The final construction of the function $F_{q}:W^{s}(q) \to V(q)$ goes as follows:
Given a point $q'\in W^{s}(q)$, we will approximate a ``generalized holonomy'' defined as $P^{-}(q_{1/2},q'_{1/2})$ (c.f. \eqref{eq:P-def}) and apply it to a suitable bases of some subspaces of vector bundles at $q_{1/2}$ in order to calculate the required data in $q'_{1/2}$.
Using the bases at $q_{1/2}$, chosen to be orthonormal, the resulted ``approximate basis'' at $q'_{1/2}$ is ``almost-orthonormal'' by the properties of the $P^{-}$ map.

In order to approximate the operator $P^{-}$ itself, one need to approximate the ``translating map'' $i^{-1}_{q'_{1/2}}$ (c.f. \ref{prop:i-factorization}).
Then one considers this map as a vector in a bundle consisting of linear maps between two bundles, this bundle is smooth along stables, using Theorem~\ref{thm:constructive-holonomies} we may approximate it.

Given an Oseledets regular point $q'\in W^{s}(q)$, we calculate all the required data in order to approximate the operator $P^{-}(q_{1/2},q'_{1/2})$.
This would result in a section $F_{q}(q')$.

Detailed proofs of the various constructions described in this section appear in Appendix~\ref{app:factorization-details}.

\section{Bilipschitz estimates}\label{sec:bilip}
We consider the maps $A(q_1,u,\ell,t)$ which were constructed in the previous section. Notice that $A(q_1,u,\ell,t):V(q) \to \mathcal{R}(g_{t}.u.q_{1})$.
We restrict $A$ to the subspace spanned by $F_{q}(q')$. Note that as $F_{q}(q')\in~ V(q)$,  this subspace is one-dimensional subspace of $V(q)$.
We endow $V(\star), \Rc(\star)$ with the Lyapunov norms defined in subsection~\ref{sub:lyap-norm}.
This give rise to a cocycle over $\mathcal{R}(\star)$ defined as
\begin{equation}\label{eq:lambda-2-def}
    \lVert g_{t}. v \rVert_{\mathcal{R}(x)} = e^{\lambda_{1}(x,t)}\cdot \lVert v \rVert_{\mathcal{R}(x)}.
\end{equation}
We define 
\begin{equation}\label{eq:frak-A-defn}
    \mathfrak{A}(q_1,u,\ell,t) = \lVert A(q_1,u,\ell,t) \rVert,
\end{equation} where the norm is the operator norms between the two spaces equipped with the Lyapunov norms.

\begin{defn}
We define for almost all $q_1\in M$, $u.q_1\in W^{uu}_{\text{loc}}(q)$ and any $\ell>0$ the function $\tau_{1,(\epsilon)}(q_1,u,\ell)$ as follows
\begin{equation}\label{eq:stopping-time-def}
    \tau_{1,(\epsilon)}(q_1,u,\ell) = \sup\left\{ t\geq 0 \mid \mathfrak{A}(q_1,u,\ell,t)\leq\epsilon \right\}.
\end{equation}
In words, $\tau_{1,(\epsilon)}$ measures the time $t$ for which the value of $A(q_1,u,\ell,t)F_{q}(q')$ reaches to size $\varepsilon$.
\end{defn}
Note that by the factorization theorem, $A(q_1,u,\ell,t)F_{q}(q')$ approximates to a high degree the value of $hd_{g_{t}.u.q_{1}}(W^{uu}(g_{t}.u.q_{1}),W^{uu}(g_{t}.q'_{1}))$ hence one should think of $\tau_{1,(\varepsilon)}$ as the time this distance grows to length $\varepsilon$ (c.f. Corollary~\ref{cor:approximation-of-stopping-time}).
With an appropriate choice of $F_{q}(q')$, one is able to compare $A(q_{1},u,\ell,t)F_{q}(q')$ with the operator norm $\left\lVert A(q_{1},u,\ell,t) \right\rVert$.

\begin{lem}[Bilipschitz estimate]\label{lem:bilip-estimate}
For almost all $q_{1}\in M$, $u.q_{1}\in~W^{uu}_{\text{loc}}(q_{1})$ any $\ell\gg 0, s>0$ we get
\begin{equation*}
    \tau_{1,(\epsilon)}(q_1,u,\ell) + \kappa_{1}\cdot s \leq \tau_{1,(\epsilon)}(q_1,u,\ell+s) \leq \tau_{1,(\epsilon)}(q_1,u,\ell)+\kappa_{2}\cdot s
\end{equation*}
where $\kappa_{1},\kappa_{2}$ are related to the Lyapunov spectrum of $g_{t}$ on $M$ and the induced cocycle on the vector bundle $V$ and the constants appearing in Lemma~\ref{lem:norm-comp}.
\end{lem}
The proof follows the spirit of the proof of Eskin-Lindenstrauss and Eskin-Mirzakhani, with one modification over the contracted part to handle the fact that the domain of the operator $A(q,u,\ell,t)$ is the vector bundle $V$, we produce here for the sake of completeness.

We start with an observation.
\begin{obs}\label{obs:restricted-cocycle-lyapunov}
Consider $\mathcal{L}(q)=\text{span} \left\{F_{q}(\mu^{s}_{q})\right\} \leq V(q)$ as a $g_{t}$-invariant subspace, we have a cocycle
\begin{equation}\label{eq:restricted-cocycle-def}
\lVert (g_{s}).F_{q}(q') \rVert_{V(g_{s}.q)} = e^{\lambda_{V\mid_{F}}(q,s)}\cdot \lVert F_{q}(q') \rVert_{V(q)},    
\end{equation}
for almost every $q'\in W^{s}_{\text{loc}}(q)$.
Then for $s>0$ we have \\ $\lambda_{V\mid_{F}}(q,s) \leq~-\kappa_{V\mid_{F}} \cdot~s$ for some $\kappa_{V\mid_{F}}>0$.
\end{obs}
The proof of this Observation follows from the contraction properties of the induced $g_{t}$-action on the auxiliary subspace constructed during the proof of factorization, see Lemma~\ref{lem:contracting-action} in the appendix.

\begin{lem}\label{lem:factorization-approx-after-s-change}
For any $s>0$, we have 
\begin{equation*}
    \left\lVert A(q_{1},u,\ell+s,t) - A(q_{1},u,\ell,t).g_{-\ell+s}\right\rVert \ll e^{-\alpha\cdot \ell},
\end{equation*}
where $\alpha$ is as in Theorem~\ref{thm:factorization}.
\end{lem}
The proof of the Lemma above follows as both expressions \emph{factorize} the same distance, namely both of them give $O(e^{-\alpha\cdot \ell})$ approximation to $hd_{g_{t}.q_{1}}(W^{uu}_{\text{loc}}(g_{t}.u.q_1), W^{uu}_{\text{loc}}(g_{t}.q'))$.

\begin{proof}[Proof of Lemma~\ref{lem:bilip-estimate}]
We have the following representation for $A(q_1,u,\ell,t)$:
\begin{equation}
\begin{split}
    A(q_1,u,\ell+s,t+\sigma) &= \left(g^{g_{t}.u.q_{1}}_{\sigma}\right).A\left(q_1,u,\ell+s,t\right).
\end{split}
\end{equation}
As we endowed $\Rc(\star)$ with the Lyapunov norm, we have that
\begin{equation*}
    \lVert \left(g^{g_{t}.u.q_{1}}_{\sigma}\right).v \rVert_{\mathcal{R}(g_{t}.u.q_{1})} =e^{\lambda_{1}(g_{t}.u.q_{1},\sigma)} \cdot \lVert v \rVert_{\mathcal{R}(g_{t}.u.q_{1})},
\end{equation*}
for any $v\in \Rc$.
Furthermore, as we endowed $V$ with the Lyapunov norms as well, by the observation from before
\begin{equation*}
     \left\lVert \left(g_{s}^{g_{-(\ell+s)}.q_1}\right).v \right\rVert_{V(g_{-s}.q)} = e^{\lambda_{V\mid_{F}}(g_{-(\ell+s)}.q_1,s)} \cdot \lVert v \rVert_{V(g_{-(\ell+s).q_1})}
\end{equation*} 
for the cocycle $\lambda_{V\mid_{F}}$ defined in~\eqref{eq:restricted-cocycle-def} for all $v\in V(g_{-(\ell+s)}.q_{1})$.

So in general one may deduce that
\begin{equation*}
   \mathfrak{A}(q_1,u,\ell+s,t+\sigma) \leq e^{\lambda_{1}(g_{t}.u.q_{1},\sigma)}\cdot \mathfrak{A}(q_1,u,\ell,t+s).
 \end{equation*}
  
In view of Lemma~\ref{lem:factorization-approx-after-s-change} we get
\begin{equation*}
   \mathfrak{A}(q_1,u,\ell+s,t+\sigma) \leq e^{\lambda_{1}(g_{t}.u.q_{1},\sigma)}\cdot\left( \mathfrak{A}(q_1,u,\ell,t)\cdot e^{\lambda_{V}(g_{-(\ell+s)}.q_1,s)} + e^{-\alpha\cdot \ell} \right).
 \end{equation*}
 
Choose $t$ such that $\tau_{1,(\varepsilon)}(q_{1},u,\ell)=t$ we have that 
\begin{equation}\label{eq:lower-bound-bilip}
    \mathfrak{A}(q_1,u,\ell+s,t+\sigma) \leq \varepsilon \cdot e^{\lambda_{1}(g_{t}.u.q_{1},\sigma)+\lambda_{V}(g_{-(\ell+s)}.q_1,s)}+e^{\lambda_{1}(g_{t}.u.q_{1},\sigma)-\alpha\cdot \ell}.
\end{equation}
We note that in the proof of the main theorem, we will take $\ell\to~\infty$. As a result, $t$ will go to infinity.
When $\ell\to\infty$ the term $e^{\lambda_{1}(g_{t}.u.q_{1},\sigma)-\alpha\cdot\ell}$ decays to zero exponentially in $\ell$.
Now fix $\sigma$ such that $\tau_{1,(\varepsilon)}(q_{1},u,\ell+s)=t+\sigma$.
If so, in view of~\eqref{eq:lower-bound-bilip}, as $\ell\to\infty$ we must have
\begin{equation*}
    \lambda_{1}(g_{t}.u.q_{1},\sigma)+\lambda_{V\mid _{F}}(g_{-(\ell+s)}.q_{1},s) \geq 0.
\end{equation*}

As we assume that $g_{t}.u.q_{1}$ and $q_{1}$ in a good Oseledets set, by the bound over the growth of the Lyapunov norms and the observation from above
\begin{equation*}
    \lambda_{1}(g_{t}.u.q_{1},\sigma) \leq \kappa_{2}\cdot\sigma, \  \lambda_{V\mid_{F}}(g_{-(\ell+s)}.q_{1},s) \leq -\kappa_{V\mid_{F}}\cdot s.
\end{equation*}

Hence we deduce that
\begin{equation*}
    \kappa_{2}\cdot \sigma -\kappa_{V\mid_{F}}\cdot s \geq 0
\end{equation*}
or equivalently
\begin{equation*}
\begin{split}
    \tau_{1,(\varepsilon)}(q_{1},u,\ell+s) - \tau_{1,(\varepsilon)}(q_{1},u,\ell) &= \sigma \\ 
    &\geq \kappa_{V\mid_{F}}\cdot\kappa_{2}^{-1}\cdot s.
\end{split}
\end{equation*}

For the upper bound, we have the following representation
\begin{equation}
\begin{split}
    A(q_1,u,\ell,t) &= \left(g_{-\sigma}^{g_{t+\sigma}.u.q_{1}}\right). A\left(q_1,u,\ell,t+\sigma\right).
\end{split}
\end{equation}
Hence we get the inequality about the norms using the Lyapunov norms
\begin{equation}\begin{split}
        \mathfrak{A}(q_1,u,\ell,t) \leq \lVert \left(g_{-\sigma}^{g_{t+\sigma}.u.q_{1}}\right) \rVert_{\mathcal{R}(g_{t+\sigma}.u.q_{1})} \cdot  \mathfrak{A}\left(q_1,u,\ell+s,t+\sigma\right).
    \end{split}
\end{equation}
Again, using Lemma~\ref{lem:factorization-approx-after-s-change} allows us to change $\mathfrak{A}(q_{1},u,\ell+s,t+\sigma)$ to $\mathfrak{A}(q_{1},u,\ell,t+\sigma)$ at the cost of $e^{-\alpha\cdot \ell}$ as before.
This yields the equation
\begin{equation*}
    \varepsilon \leq e^{\lambda_{1}(g_{t+\sigma}.u.q_{1})}\cdot \left( \mathfrak{A}(q_{1},u,\ell,t+\sigma)+e^{-\alpha\cdot \ell}\right).
\end{equation*}
Choosing again $t,\sigma$ as before and utilizing the Lyapunov norms we are led to the inequality
\begin{equation*}
    \lambda_{1}(g_{t+\sigma}.u.q_{1},-\sigma) + \lambda_{V\mid F}(g_{-\ell}.q_{1},-s) \geq 0.
\end{equation*}
Using the growth bounds estimates for the Lyapunov cocycles we get
\begin{equation*}
    -\kappa_{2}\cdot \sigma + \kappa_{V\mid F}\cdot s \geq 0,
\end{equation*}
or equivalently
\begin{equation*}
    \sigma \leq \kappa_{V\mid F}\cdot \kappa_{2}^{-1}\cdot s.
\end{equation*}
\end{proof}

\section{The Eskin-Mirzakhani scheme}\label{sec:8-pts}
We recall the following Definition~\ref{def:wasserstein}:
\begin{defn*}
We define the (normalized) \emph{Wasserstein metric} $d_{W}$ between two conditional measures (of bounded support) as
\begin{equation*}
    d_{W}(\mu_{1},\mu_{2})=\sup_{h:M\to\R \text{ is Lipschitz with }Lip(h)\leq 1 }\left\{\left\lvert\int_{M}h(x)\left(\frac{d\mu_{1}(x)}{\mu_{1}(M)}-\frac{d\mu_{2}(x)}{\mu_{2}(M)}\right)\right\rvert\right\}.
\end{equation*}
\end{defn*}

We will use this distance in order to measure the distance between conditional measures $\mathbf{f}_1$ as defined in \S\ref{sec:conditional-measures}.

The main result of this section is the following proposition
\begin{prop}\label{prop:main}
For some small $\delta_{0}< 1$, there exists a compact subset $\mathcal{M}\subset M$ with $\mu(\mathcal{M})>1-\delta_0$ such that $\mathbf{f_{1}}$ is uniformly continuous over $\mathcal{M}$ and some $C=C(\mathcal{M},\delta)>1$ such that for every $\varepsilon>0$ there exists a subset $\mathcal{M}'\subset\mathcal{M}$ with $\mu(\mathcal{M}')>\delta_0$ such that for every $x\in \mathcal{M}'$ there exists some $y\in \mathcal{M'}\cap W^{u}(x)$ such that $$C^{-1}\cdot\varepsilon \leq hd_{x}(W^{uu}_{\text{loc}}(x),W^{uu}_{\text{loc}}(y))\leq C\cdot \varepsilon$$ and  such that
\begin{equation*}
    d_{W}(\mathbf{f}_{1}(x), \mathbf{f}_{1}(y)) \ll \varepsilon,
\end{equation*}
where $hd_{x}$ is the local Hausdorff distance at $x$ defined in Definition~\ref{defn:local-haus-distance} and $d_{W}$ is the Wasserstein distance as in Definition~\ref{def:wasserstein}.
\end{prop}

We start with the following definition.
\begin{defn}
A \emph{$Y$-configuration} of points $q$, $q_{1}=g_{\ell}.q$, $u.q_{1}$, \\  $q_{2}=~g_{\tau_{1,(\varepsilon)}(q_{1},u,\ell)}.(u.q_{1})$, $q_{3}=g_{t_{1}}.q_{1}$ depending on parameters $q, u,\ell$ is a set of points such that all the points belong to some Oseledets' good set which admits a good splitting and moreover
\begin{equation*}
    \lambda_{1}(q_{1},t_{1}) = \lambda_{1}(u.q_{1}, \tau_{1,(\varepsilon)}(q_{1},u,\ell)),
\end{equation*}
where $t_{1},\tau_{1,(\varepsilon)}$ are the quantities defined in~\eqref{eq:def-of-t_2} and~\eqref{eq:stopping-time-def} respectively.
\end{defn}
\texttt{Standing assumption - our $Y$ configurations will be always chosen in a way where the point $q_{1/2}$ is also Oseledets good point and $q_{1/2}, g_{-\ell/2}.u.q_{1}$ are point in $\mathcal{P}$, namely satisify the QNI condition}.

In order to apply the Eskin-Mirzakhani scheme we will need to generate sets of points in two $Y$-configurations which will be synchronized in the sense that all lengths of the corresponding legs are the same, as can be seen in Figure~\ref{fig:8-pts}.
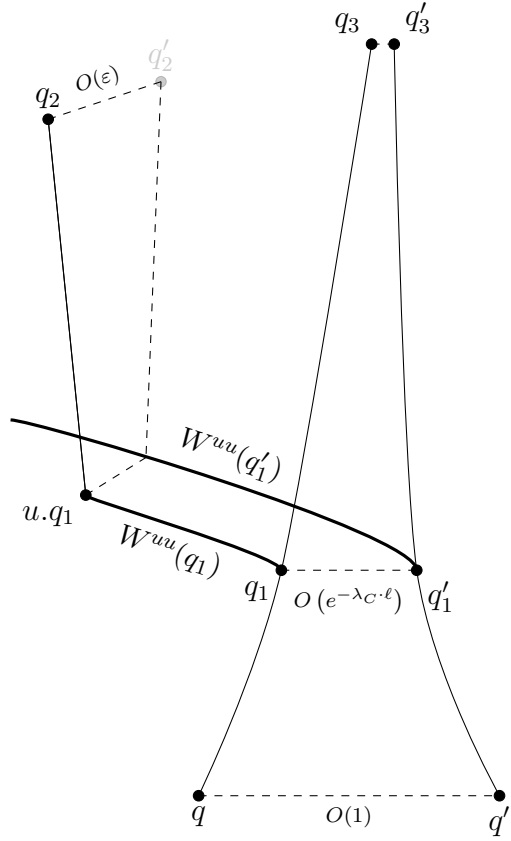
\begin{figure}[H]
    \centering
    \scalebox{1}{
\begin{tikzpicture}

\coordinate (q) at (-2,0); 
\coordinate (q_1) at (-0.9,3);
\coordinate (q_3) at (0.3,10); 
\coordinate (q') at (2,0);
\coordinate (q'_1) at (0.9,3);
\coordinate (q'_3) at (0.6,10);
\coordinate (uq_1) at (-3.5,4);
\coordinate (q_2) at (-4,9);
\coordinate (q'_2) at (-2.5,9.5);
\coordinate (uq'_1) at (-4.5,5);
\coordinate (z) at (-2.7,4.5);


\draw plot [smooth] coordinates {(q) (q_1) (q_3)};
\draw plot [smooth] coordinates {(q') (q'_1) (q'_3)};
\draw plot [smooth] coordinates{(uq_1) (q_2)};
\draw (uq_1) -- (q_2);

\draw[very thick] (q_1) .. controls +(up:0.2cm) and +(left:0.2cm) .. (uq_1) node[pos=0.5, below, sloped]{\smaller $W^{uu}(q_{1})$};
\draw[very thick] (q'_1) .. controls +(up:0.5cm) and +(right:0.3cm)  .. (uq'_1) node[pos=0.5, above, sloped]{\smaller $W^{uu}(q'_{1})$};


\draw[dashed] (q) -- (q') node[pos=0.5, below]{\tiny $O(1)$};
\draw[dashed] (q_1) -- (q'_1) node[pos=0.5, below=0.1 cm]{\tiny $O\left(e^{-\lambda_{C}\cdot\ell}\right)$};
\draw[dashed] (q_3) -- (q'_3);
\draw[dashed] (uq_1) -- (z);
\draw[dashed] (q_2) -- (q'_2) node[pos=0.5,above, sloped]{\tiny $O(\varepsilon)$};
\draw[dashed] (q'_2) -- (z);

\filldraw (q) circle (2pt) node[below]{$q$};
\filldraw (q') circle (2pt) node[below]{$q'$};
\filldraw (q_1) circle (2pt) node[below left]{$q_1$};
\filldraw (q'_1) circle (2pt) node[below right]{$q'_1$};
\filldraw (q_3) circle (2pt) node[above left]{$q_3$};
\filldraw (q'_3) circle (2pt) node[above right]{$q'_3$};
\filldraw (uq_1) circle (2pt) node[below left]{$u.q_1$};
\filldraw (q_2) circle (2pt) node[above]{$q_2$};
\filldraw[opacity=0.2] (q'_2) circle (2pt) node[above, opacity=.2,text opacity=0.5]{$q'_2$};
\end{tikzpicture}
}
    \caption{Illustration of the points chosen in \S\ref{sec:8-pts}}
    \label{fig:8-pts}
\end{figure}

\subsubsection*{Organization of the choices}
We start by setting up two different subsets - $M_{base}$ and $M_{\text{rec}}$.
$M_{base}$ is a set composed of points with controlled Osceledets splitting. This is the set for which we want our endpoint to belong.
$M_{\text{rec}}$ is a set of points which spend most of their time inside $M_{base}$ in a highly controlled fashion.
As we want to control synchronization between the $Y$-configurations, we will construct the points $q,q',q_1,q'_1,q_2,q'_2,q'_3$ in the recurrent set $M_{\text{rec}}$ and then remedy the situation on \S\ref{sub:coupled}.
The proof consists of three parts.
The proof begins by considering, for a given $q_{1}, u.q_{1}$ the sets $E$ for which there exists a time in the past, $\ell$ such that $q,q_1,q_2,q_3$ all belong the the recurrent set.
Using the bilipschitz estimates and Fubini argument, one can prove that there exists a universal set of times $D_{good}$ of large density and a set of points $M_{\text{good start}}$ of large measure such that for any $\ell\in D_{good}$, $q_1\in M_{\text{good start}}(\ell)$ and many $u.q_{1}\in W^{uu}_{\text{loc}}(q_{1})\cap B_{1}(q_1)$, the associated $Y$ configuration lands in $M_{\text{rec}}$, this is done in~\S\ref{sub:existence-y-config}.
In the second part, one needs to synchronize the two different $Y$-configurations which we get.
By considering the relative divergence of the curves, one may conclude that the two different stopping times $\tau_{1,(\varepsilon)}(q_{1},u,\ell)$ and $\tau_{1,(\varepsilon)}(q'_{1},u',\ell)$, one shows that they are only differ by a bounded constant as in Proposition~\ref{prop:match-times}.
Then one may use the recurrence property to correct the times. This is done in~\S\ref{sub:coupled}.
Then one can use a strategy similar to the one used by Benoist-Quint~\cite{B-Q} in order to show that by letting $\varepsilon$ go to $0$, one can indeed conclude extra invariance, this is done in~\S\ref{sub:endgame}.

\subsection*{Initial choices}
Fix some $\varepsilon>0$.
Let $\delta$ be an arbitrarily small constant.
Let $\mathfrak{P}$ denote the points satisfying QNI for some $\alpha$, by our assumptions $\mu(\mathfrak{P})>0$. 
We are assuming that $\delta$ is smaller than $\mu(\mathfrak{P})$.

By Lusin's theorem, there exists a compact subset $M_{\text{uni}}$ of measure $\mu(M_{\text{uni}})>1-\delta$ such that $f_{1}$ is uniformly-continuous over $M_{\text{uni}}$.
Fix some $\varepsilon'>0$ depending only on the Lyapunov spectrum.
Using Osceledets theorem, we may find a time  $T'=T'(\delta)>0$ and a set of Oseledets regular points $M_{\text{Os-reg},\varepsilon',T'}$ of measure greater than $1-\delta$.
We define $M_{base}=~M_{\text{uni}}~\cap~ M_{\text{Os-reg},\varepsilon',T'}$.
We will define an additional subset $M_{\text{rec}}$ as a set of measure greater than $1-\delta$ such that there exists some $T''(\delta)>0$ such that for all $T>T''(\delta)$ and $x\in M_{\text{rec}}$
\begin{equation}\label{eq:rec-density}
    \lvert \left\{ t\in [0,T] \mid g_{t}.x \in M_{base} \right\}\rvert \geq 0.99T.
\end{equation}

While we are going to show that there exist two good $Y$-configurations composed of the points in $M_{base}$, in order to couple them as in~\S~\ref{sub:coupled}, one needs to modify the side lengths a bit after the points were chosen (the issue stems from Proposition~\ref{prop:match-times} and appears in~\eqref{eq:distortion-q'_2},\eqref{eq:distortion-q'_3}, relating the distortions of the sides of the tagged $Y$-configuration). The definition of $M_{\text{rec}}$ comes to the rescue here as the uniform recurrence rate ensures us that by a minor modification of the side lengths, one may correct both $Y$-configurations at once to be composed of points in $M_{base}$, as needed. This also shows that choice of any density strictly bigger than $1/2$ in~\eqref{eq:rec-density} would have worked.

In view of Lemma~\ref{lem:norm-comp}, without loss of generality we may assume that $M_{\text{rec}}$ is the set for which Lemma~\ref{lem:norm-comp} is applicable.

Given $q_{1}, u.q_{1}$ in an Osceledets' good set, we define $t_{1}=t_{1}(q_{1},u,t)$ to be the number which solves the equation
\begin{equation}\label{eq:def-of-t_2}
    \lambda_{1}(q_{1},t_{1}) = \lambda_{1}(u.q_{1},t),
\end{equation}
where $\lambda_{1}$ is the cocycle defined in~\eqref{eq:lambda-2-def}.
By the continuity property of the Lyapunov cocycle, $t_{1}$ is bilipschitz in $t$ for fixed $q_{1},u.q_{1}$.

Let $\tau_{1,(\varepsilon)}(q_{1},u,\ell)$ be the stopping time defined in~\eqref{eq:stopping-time-def}.
Define the sets
\begin{equation}\label{eq:E-2}
    E_{\text{good starting times, left branch}}(q_{1},u) = \left\{\ell \mid g_{\tau_{1,(\varepsilon)}(q_{1},u,\ell)}.u.q_{1} \in M_{\text{rec}} \right\},
\end{equation}
and
\begin{equation}\label{eq:E-3}\begin{split}
    &E_{\text{good starting times, both branches}}(q_{1},u.q_{1}) \\
    &\ = \left\{\ell \in E_{\text{good starting times, left branch}}(q_1,u) \mid g_{t_{1}}.q_{1} \in M_{\text{rec}} \right\}.
\end{split}
\end{equation}

The set $E_{\text{good starting times, both branches}}(q_1, u.q_1)$
 allows us to choose set of ``starting times'' (amounting to the initial choices of $q$) such that the points $q_{2},q_{3}$ both belong to the recurrent set $M_{\text{rec}}$.

 \subsection{Existence of $Y$-configurations}\label{sub:existence-y-config}
 \begin{prop}\label{prop:good-configuartion}
 There exists some time $\ell_{min}>0$ and a subset $M_{\text{good start}}\subset M_{\text{rec}}$ of measure greater than $1-c_{1}(\delta)$ for some $c_{1}(\delta)$ which goes to $0$ with $\delta$, such that for any $q_1\in M_{\text{good start}}$ we have a subset $Q_{\text{good stop}}(q_{1})\subset W^{uu}_{\text{loc}}(q_{1}) \cap B_{1}(q_{1})$ such that $$\mathsf{m}_{x}^{uu}(Q_{\text{good stop}}(q_{1}))>~1-~c_{2}(\delta),$$ for some $c_{2}(\delta)$ which goes to $0$ with $\delta$ and if $q_{1}\in M_{\text{good start}}$, \\ $ u.q_{1}\in~Q_{\text{good stop}}(q_{1})$ and any $\ell>\ell_{min}$ we have that
 \begin{equation*}
     \lvert E_{\text{good starting times, both branches}}(q_{1},u.q_{1}) \cap [0,\ell] \rvert \geq (1-c_{3}(\delta))\cdot\ell,
 \end{equation*}
 for some $c_{3}(\delta)$ which goes to $0$ as $\delta\to 0$.
 \end{prop}
This proposition assures us that by choosing points $q_{1}$ from an appropriate large set $M_{\text{good start}}$ and many points $u.q_{1} \in Q_{\text{good stop}}(q_{1})\subset~W^{uu}_{\text{loc}}(q_{1})$, for most choices of large $\ell$, we would have that all the points $q_{1}$,$u.q_{1}$,\\$ q_{2}=~g_{\tau_{1,(\varepsilon)}(q_{1},u,\ell)}.u.q_{1}$, $q_{3}=~g_{t_{1}(q_{1},u,\ell)}.q_{1}$ and $q=g_{-\ell}.q_{1}$ all belong to a good set.

\begin{proof}
By the ergodic theorem for any $\delta>0$ there exists some $\ell_{\text{recurrence}}\in~\mathbb{R}$ and a set $M_{\text{initial}}$ of measure greater than $1-\delta$ such that for any $x\in M_{\text{initial}}$ and any $L>\ell_{\text{recurrence}}$ we have that $$\frac{\lvert \left\{t\in [0,L] \mid  g_{t}.x \in M_{\text{rec}}\right\}\rvert}{L}~\geq~1-\delta.$$

Define $$Q_{\text{good stop}}(q) = W^{uu}_{\text{loc}}(q)\cap M_{\text{initial}}$$ and
consider the set
\begin{equation*}
    M_{\text{good start}} = \left\{q\in M \ \middle\vert \ \frac{\mathsf{m}^{uu}_{q}(Q_{\text{good stop}}(q)\cap B_{1}(q))}{\mathsf{m}^{uu}_{q}(W^{uu}_{\text{loc}}(q)\cap B_{1}(q))} > 1-\delta\right\}.
\end{equation*}
Suppose now that $q_{1} \in M_{\text{good start}}$ and $u.q_{1} \in Q_{\text{good stop}}(q_1) \subset M_{\text{initial}}$.
Define
$$E_{bad}~=~\left\{ t \mid g_{t}.u.q_{1} \notin M_{\text{rec}} \right\}.$$
As $u.q_{1} \in M_{\text{initial}}$, we have that the density of the set $E_{bad}$ is less or equal to $\delta$.
Note that by
\begin{equation*}
    E_{\text{good starting times, left branch}}(q_{1},u)^{c} = \left\{ \ell \mid g_{\tau_{1,(\varepsilon)}(q_{1},u,\ell)}.(u.q_{1}) \notin M_{\text{rec}} \right\}
\end{equation*}
Propagating the bad times from $E_{bad}$ to $E_{\text{good starting times, left branch}}(q_{1},u)^{c}$ by means of the bilipschitz estimates of Lemma~\ref{lem:bilip-estimate}, we see that the density of $\ell$'s corresponding to the bad times $t\in E_{bad}$ is bounded by $4\Theta\cdot\delta$ for large enough $\ell$'s (namely $\ell>\Theta\cdot\ell_{\text{recurrence}}$ so the density statement will remain valid).
Note that $\ell \mapsto t_{1}(q_{1},u,\tau_{1,(\varepsilon)}(q_{1},u,\ell))$ is $\Theta^2$ bilipschitz, as by Lemma~\ref{lem:bilip-estimate} $\tau_{1,(\varepsilon)}$ is bilipschitz in $\ell$ and $t_{1}$ is bilipschitz by its definition in~\eqref{eq:lambda-2-def} and the estimates regarding the cocycle growth using the Lyapunov norm in \S~\ref{sub:lyap-norm}, hence in a similar manner by propagating the times from $E_{\text{good starting times, left branch}}(q_{1},u)$ we can see that the density of $E_{\text{good starting times, both branches}}(q_{1},u)$ is larger than $1-4\Theta^{2}\cdot\delta$.
\end{proof}

In the previous proposition, the constructed subsets $M_{\text{good start}}$ and $Q_{\text{good stop}}(q_{1})$ were independent of $\ell$, but as a result, the set of times $E_{\text{good starting times, left branch}}(q_{1},u)$ were only of positive proportion.
The next corollary rectify the situation, makes the sets $M_\text{good start},Q_\text{good stop}$ dependent over $\ell$.

\begin{cor}\label{cor:plenty-of-Y-config}
There exists a set of times $D_{good}\subset\mathbb{R}$ and a number $\ell'>0$ such that for $\ell>\ell'$, the density of $D_{good}$ in $[0,\ell]$ is greater than $1-c_{4}(\delta)$ and for any given number $\ell\in D_{good}$ such that $\ell>~\ell'$, there exists a subset $M_{\text{good start}}(\ell)~\subset ~M_{\text{good start}}$ such that for any $q_{1}~\in~M_{\text{good start}}(\ell)$ there exists a subset \ $Q_{\text{good stop}}(q_{1},\ell)\subset ~W^{uu}_{\text{loc}}(q_1)$
 which satisfy
\begin{equation}\label{eq:good-stop-density}
    \frac{\mathsf{m}^{uu}_{q_1}(Q_{\text{good stop}}(q_{1},\ell))}{\mathsf{m}^{uu}_{q_1}(W_{\text{loc}}^{uu}(q_1))} > 1-c_{4}'(\delta)
\end{equation} such that for all $q_{1} \in M_{\text{good start}}(\ell)$, $u.q_{1} \in Q_{\text{good stop}}(q_{1},\ell)$ we have
\begin{equation*}
    \ell \in E_{\text{good starting times, both branches}}(q_{1},u).
\end{equation*}
\end{cor}

\begin{proof}
Using Fubini's theorem
\begin{multline*}
   \int_{q\in M}\int_{u.q\in W_{\text{loc}}^{uu}(q)}\int_{\ell\in \mathbb{R}}\chi(q_{1},u.q,\ell) dLeb(\ell) d\mu^{uu}_{q}(u.q)d\mu(q) \\ =\int_{\ell\in\mathbb{R}}\int_{q\in M}\int_{u.q\in W^{uu}_{\text{loc}}(q)}\chi(q_{1},u.q,\ell) d\mu^{uu}_{q}(u.q)d\mu(q)dLeb(\ell),
\end{multline*}
where $\chi$ is a characteristic function for the set of points $(q,u.q,\ell)$ which are in suitable position.
\end{proof}

Furthermore, as $Q_{\text{good stop}}(q_{1},\ell)$ is of large density in $W^{uu}_{\text{loc}}(q_1)$, by proportionality of the conditional measures, we may assume that this set (assuming $q_{1/2}=g_{-\ell/2}.q_{1}$ belongs to a  set of points satisfying QNI) contains $g_{-\ell/2}.u.q_{1}$ which satisfies QNI.

\texttt{Standing assumption}
We will choose $q_{1}\in g_{\ell/2}.\mathfrak{P}$, namely $q_{1/2}\in\mathfrak{P}$ will be part from a dynamical quadrilateral satisfying QNI.

We are now ready to pick the ``bottom half'' of the configurations, depending on $q_{1},\ell$.
\subsection*{Choice of $q,q',q'_1$}
Define $\mathfrak{M}_{u}$ to be the subspace defined as in Lemma~\ref{lem:badsubspace} applied to the map $A(q_{1},u,\ell,\tau_{1,(\varepsilon)}(q_{1},u,\ell))$ inside $V(q)$.
Apply Lemma~\ref{lem:subspace-avoidance} with $M^{\dagger}=g_{-\ell}.M_{\text{good start}}(\ell)$ to get $M^{\dagger}_{\text{avoidance}}(\ell)$. The subset $M^{\dagger}_{\text{avoidance}}(\ell)$ comprises of points in $g_{-\ell}.M_{\text{good start}}(\ell)$ such that for all points \\ $p\in M^{\dagger}_{\text{avoidance}}(\ell)$ there exists a nearby point $p'\in W^{s}(p)$ which also belong to $M^{\dagger}(\ell)$.

By the Lemma, we have $\mu(M^{\dagger}_{\text{avoidance}}(\ell)) \geq 1-c(\delta)$.
Now define $$M_{\text{good start,avoidance}}(\ell)=g_{\ell}.M^{\dagger}_{\text{avoidance}}(\ell).$$
Suppose $\ell\in~D_{good}$ and $q_{1}\in M_{\text{good start, avoidance}}(\ell)$.
Choose
\begin{itemize}
    \item $q=g_{-\ell}.q_{1}.$
\end{itemize}
By the definition of $M_{\text{good start, avoidance}}(\ell)$ we may choose
\begin{itemize}
    \item $q'\in W^{s}(q)\cap g_{-\ell}.M_{\text{good start}}(\ell).$
\end{itemize}such that $\rho'(\delta)<~dist(q,q')\ll~1$ and $dist(F_{q}(q'),\mathfrak{M}_{u})>\rho(\delta)$ for most $u\in~Q_{\text{good start}}(q_{1},\ell)$.
In addition we set
\begin{itemize}
    \item $q'_1 = g_{\ell}.q'.$
\end{itemize}
Note that $q'_1 \in M_{\text{good start}}(\ell)$ by the choice of $q'$.
Moreover, as we assumed that there exists a subset of points $\mathfrak{S}(q_{1})\subset W^{s}(q_{1/2})$ of positive $\mu^{s}_{q_{1/2}}$ density and we have the following property of conditional measures - $g_{\ell/2}.\mu^{s}_{q_{1/2}} \propto \mu^{s}_{q_1}$, we see that $g_{\ell/2}.\mathfrak{S}(q_{1/2})$ is of positive $\mu^{s}_{q}$ density, hence we may assume in addition in our choices that
\begin{itemize}
    \item $q'\in g_{-\ell/2}.\mathfrak{S}(q_{1/2})\cap g_{-\ell}.M_{\text{good start}}(\ell),$
\end{itemize}
and in addition
\begin{itemize}
    \item $q'_{1}=g_{\ell}.q'\in g_{\ell/2}.\mathfrak{S}(q_{1/2})\cap M_{\text{good start}}(\ell).$
\end{itemize}

Here we briefly indicate how to modify the choices for the QNI condition of Definition~\ref{def:QNI2}. Taking as a first choice $q_{1/2}$ instead of $q_{1}$, as long as $q_{1/2}\in\mathcal{P}$ will not make any changes to the actual proofs given in Proposition~\ref{prop:good-configuartion} and Corollary~\ref{cor:plenty-of-Y-config} and they will work for $q_{1/2}$ instead of $q_{1}$ verbatim.

Notice the following easy Lemma.
\begin{lem}
Fix some $\epsilon,\nu>0$. Get $C,k_0$ as in Definition~\ref{def:QNI2}.
For any $X\subset M$ a measurable set with $\mu(X)>1-\epsilon$, there exists a set of points $X'\subset X$ with $\mu(X')>1-c_{\text{rec}}(\epsilon)$ with $c_{\text{rec}}(\epsilon)\to 0$ as $\epsilon\to 0$, such that for each $x\in X'$ there exists a subset of times $\ell\in \mathbb{R}$ of (Banach) density greater than $1-c'_{\text{rec}}(\nu)$ such that
$g_{\ell/2}.x, g_{-\ell/2}.x \in X'$.
\end{lem}
The proof follows immediately by considering the set of uniformly generic points in $X$, which is of almost full measure. Refining it further we may assume lower bound over the uniformity and the recurrence times of the orbits to $X$.
Now the set of forward return times is of density close to $\mu(X)$. Similarly, the density of the set of backwards returns to $X$ is also of density close to $\mu(X)$. Comparing the set of times, we get a set of small density (up to say factor of $2$) where such multiple recurrence does not hold.

Considering the set $X$ to be the intersection of $\mathcal{X}$ with the set of uniformly generic points (of given uniformity), we see that we have plenty of choices where $q_{1},q_{1/2},q$e all in $\mathcal{X}$ and generic.

As a result, refining the choices of those triplets $q_{1},q_{1/2},q$ with the set $M_{\text{good start}}$ we see that we may find those triplets in a good configuration.

Moreover, given the set $g_{-\ell/2}.Q_{\text{good stop}}(q_{1},\ell)\subset W^{uu}_{\text{loc}}(q_{1/2})$, which is of large density, and the set $U_{q_{1/2}}$ which is of large density, we may intersect them both and ensure that $u.q_{1/2}\in U_{q_{1/2}}$.
Furthermore, we may assume that $q'\in M_\text{good start, avoidance}(\ell)$ is also chosen such that $q'_{1/2}\in S_{q_{1/2}}$ as Lemma~\ref{lem:subspace-avoidance} ensures us a set of large density at $W^{s}(q)$ of ``good vectors for factorization'' and by the definition~\ref{def:QNI2}, the set $S_{q'_{1/2}}$ is of large density as well.

\subsection{Existence of synchronized $Y$-configurations}\label{sub:coupled}
Using the choices made above with Corollary~\ref{cor:plenty-of-Y-config}, for $\ell\in D_{\text{good}}$ we have that $q,q_{1}$ and most of $u.q_{1}$ and $q',q'_{1}$ and most of $u'.q'_{1}$ both are forming good $Y$ configurations.
\begin{obs}
By simply unfolding the definition of the set $D_{\text{good}}$ and the sets $E_{\text{good starting times, both branches}}(q_{1},u)$, $E_{\text{good starting times, both branches}}(q'_{1},u')$ we get
\begin{equation*}
    g_{\tau_{1,(\varepsilon)}(q_{1},u,\ell)}.u.q_{1}, \  g_{\tau_{1,(\varepsilon)}(q'_{1},u',\ell)}.u'.q'_{1} \in M_{\text{rec}}
\end{equation*}
\end{obs}

Suppose that $u.q_{1}$ is chosen such that $u.q_{1}\in Q_{\text{good stop}}(q_{1},\ell)$. \\
Set $t=\tau_{1,(\varepsilon)}(q_1,u,\ell)$ and define $t_{1}$ by the cocycle equation~\eqref{eq:def-of-t_2}.

Since $\ell\in D_{\text{good}}$, by construction we get $\ell\in E_{\text{good starting times, both branches}}(q_{1},u)$ and so $g_{t_1}.q_{1}\in M_{\text{rec}}$.

Similarly, if $u'.q'_1 \in Q_{\text{good stop}}(q'_{1},\ell)$, and we define $t'_{1}$ in an analogous manner, we get that $g_{t'_{1}}.q'_1 \in M_{\text{rec}}$.

Define $$\nu(u)=A(q_{1},u,\ell,\tau_{1,(\varepsilon)}(q_1,u,\ell)).F_{q}(q').$$
\begin{prop}\label{prop:A-avoidance}
There exists a subset $$Q_{\text{good stop, avoidance}}(q_{1},q'_{1},\ell)\subset~ Q_{\text{good stop}}(q_{1},\ell)$$ with $m_{uu}(Q)>1-c_{5}(\delta)$ with $c_{5}(\delta)\to 0$ as $\delta\to 0$ and a number $\ell'=~\ell'(\delta,\varepsilon)$ such that for all $\ell>\ell'$, $q_{1} \in M_{\text{good start}}(\ell)$ and $u.q_{1}~\in~ Q_{\text{good stop, avoidance}}(q_1,q'_1,\ell)\subset W^{u}(q_{1})$ we have
\begin{equation*}
    C(\delta)^{-1}\cdot\varepsilon\leq \lVert \nu(u) \rVert \leq C(\delta)\cdot\varepsilon.
\end{equation*}
Moreover, we may take $C(\delta)$ to be bounded bellow and above in the interval $(0,\infty)$ in $\delta$ as $\delta\to 0$.
\end{prop}
\begin{proof}
We already picked $q_{1},q'_{1}$ and $\ell$, so we have fixed $q,q'$ by that.
As a result, we have fixed a vector $F_{q}(q')\in V(q)$ by the construction of the vector bundle $V$ in~\S\ref{sub:approx-holonomies}.
Now consider some $u.q_1 \in Q_{\text{good stop}}(q_{1},\ell)$.
If $F_{q}(q')$ avoids $\mathfrak{M}_{u}$, we are done as the inequalities follow at once from the estimates of Lemma~\ref{lem:badsubspace}.
Fixing $Q_{\text{good stop}}(q_{1},\ell)$ in $W^{uu}_{\text{loc}}(q_{1})$ as in Corollary~\ref{cor:plenty-of-Y-config}.
Due to the avoidance Lemma~\ref{lem:subspace-avoidance}, we may find a subset $Q_{\text{good stop, avoidance}}=Q_{\text{good stop, avoidance}}(q,q',\ell)\subset Q_{\text{good start}}(q_{1},\ell)$ such that for all $u\in Q$ we have
\begin{equation*}
    dist(F_{q}(q'),\mathfrak{M}_{u}) \geq \rho(\delta).
\end{equation*}
By the assumptions regarding the density of $Q_{\text{good stop}}(q_{1},\ell)$ as in~\eqref{eq:good-stop-density} and the subset resulting from Lemma~\ref{lem:subspace-avoidance}, we have that \\ $Q_{\text{good stop, avoidance}}(q_{1},q_{1}',\ell)\subset~W^{uu}_{\text{loc}}(q_{1})$ is a set of density large than some $1-c_{5}(\delta)$ as needed.
\end{proof}
\begin{prop}\label{prop:packing-all-data}
There exists a subset $M'(\ell)\subset M$ such that $\mu(M'(\ell))>~1-~c'_{6}(\delta)$ and for each $q_{1},q'_1\in M'$ a subset $Q=Q(q_{1},q'_1,\ell)\subset Q_{\text{good stop, avoidance}}(q_{1},q'_{1},\ell)$ such that
\begin{equation*}
    \frac{\mathsf{m}^{uu}_{q_1}(Q)}{\mathsf{m}_{uu}(W^{uu}_{\text{loc}}(q_1))}>1-c_{6}(\delta)
\end{equation*}
and a number $\ell_{6}$ such that for all $\ell>\ell_6$, $q_{1} \in M', u.q_{1} \in Q$ we have
\begin{equation}\label{eq:nu-estimate}
    c'_{1}(\delta)\varepsilon \leq \lVert \nu(u) \rVert \leq c_{2}'(\delta)\varepsilon,
\end{equation}
\begin{equation}\label{eq:hd-estimate-final}
    c''_{1}(\delta)\cdot\varepsilon \leq hd_{g_{\tau_{1,(\varepsilon)}(q_{1},u,\ell)}.u.q_{1}}(W^{uu}_{\text{loc}}(g_{\tau_{1,(\varepsilon)}(q_{1},u,\ell)}.u.q_{1}),W^{uu}_{\text{loc}}(g_{t'_1}.q'_1)) \leq c''_{2}(\delta)\cdot\varepsilon
\end{equation}
and
\begin{equation}\label{eq:tau-estimate}
    \alpha_{3}^{-1}\cdot\ell\leq \tau_{1,(\varepsilon)}(q_{1},u,\ell) \leq \alpha_{3}\cdot\ell,
\end{equation}
with $c_{6}(\delta),c'_{6}(\delta)$ tend to $0$ as $\delta$ goes to $0$.
\end{prop}
\begin{proof}
Consider the set $M_{\text{good start}}(\ell)$ which was constructed before and intersect it with the set $K$ as described in the factorization theorem, Theorem~\ref{thm:factorization} to define $M'$.
Furthermore, for each $q_{1},q'_{1}\in M'$ we define $Q$ to be the subset of $Q_{\text{good stop, avoidance}}(q_{1},q'_{1},\ell)$ refined such that its points satisfy the quantitative non-integrability condition.

Equation~\eqref{eq:nu-estimate} follows immediately from the definition of $\nu$ and $\tau_{2,(\varepsilon)}$ and the estimates proven in Proposition~\ref{prop:A-avoidance}.
Equation~\eqref{eq:hd-estimate-final} follows from ~\eqref{eq:nu-estimate} and the factorization theorem, which gives
\begin{equation*}
    c'_{1}(\delta)\cdot\varepsilon-e^{-\alpha\cdot\ell}\leq hd_{g_{t}.u.q_{1}}(W^{uu}_{\text{loc}}(g_{t}.u.q_{1}),W^{uu}_{\text{loc}}(g_{t}.q'_1)) \leq c'_{1}(\delta)\cdot\varepsilon + e^{-\alpha\cdot\ell}.
\end{equation*}
While the estimates in~\eqref{eq:tau-estimate} follows from the upper bound given in the a-priori growth estimate in~\ref{sub:a-priori-stopping} and the lower bound given by the definition of the Lyapunov norms.
\end{proof}

Let $\varkappa$ denote the minimum between the densities of $Q_{QNI}(q_1,q_1')$ in $W^{uu}_{\text{loc}}(q_1)$ and $Q_{QNI}(q'_1,q_1)$ in $W^{uu}(q'_1)$.
From now on we will assume in addition that $\delta$ is small enough such that both densities of $Q(q_{1},q_{1}',\ell)$ and $Q(q'_1,q_1,\ell)$ are strictly greater than $1-\varkappa$, namely $c_{6}(\delta)<\varkappa$ and in addition, $\delta$ is small enough such that $\mu(M'(\ell))>1-\mu(\mathfrak{P})$ namely $c_{6}(\delta)<\mu(\mathfrak{P})$.

This assumption allow us to conclude that $M'(\ell)$ contains points $q_{1}$ which satisfy QNI and also for each such point, $Q(q_{1},q'_{1},\ell)$ contains points $u.q$ which satisfy QNI and $Q(q'_1,q_1,\ell)$ contains $u'.q'$ which satisfy QNI.

\begin{prop}\label{prop:match-times}
Fix some $q_{1}\in~M'(\ell)\cap \mathfrak{P}$, $q_{1}'\in \mathfrak{S}(q_1)\cap M'(\ell)$.
Suppose that $u.q_{1} \in~Q(q_{1},q'_{1},\ell)\cap~Q_{QNI}(q_{1},q'_1)$, $u'.q'_{1}\in Q(q'_{1},q_{1},\ell)\cap~Q_{QNI}(q'_1,q_1)$ for some $\ell>\ell_6$ and~\eqref{eq:hd-estimate-final} holds for both, then there exists some $C_{0}(\delta)>0$ such that
\begin{equation}
    \left\lvert \tau_{1,(\varepsilon)}(q_1, u, \ell) - \tau_{1,(\varepsilon)}(q'_1, u',\ell) \right\vert \leq C_{0}(\delta).
\end{equation}
\end{prop}

\begin{proof}
Notice that we have $u'.q_{1}'\in W^{uu}_{\text{loc}}(q'_{1})$. Hence we get
\begin{equation}
\begin{split}
    hd_{g_{t'}.u'.q'_1}(W^{uu}_{\text{loc}}(g_{t'}.u'.q'_1),W^{uu}_{\text{loc}}(g_{t'}.u.q_1)) &\leq \lVert A(q'_{1},u',\ell,t')F_{q'}(q)\rVert+e^{-\alpha\cdot\ell} \\
    &\leq C(\delta)\cdot\mathfrak{A}(q'_1,u',\ell,t') \\
    &= C(\delta)\cdot \varepsilon,
\end{split}
\end{equation}
for some $C(\delta)>1$.
This inequality means that $W^{uu}_{\text{loc}}(g_{t'}.u'.q'_1)\cap B_{1}(g_{t'}.u'.q'_1)$ of distance of $O_{\delta}(\varepsilon)$ from $W^{uu}_{\text{loc}}(g_{t'}.u.q_{1})$. Hence contracting the pieces a bit by flowing backwards with $g_{t}$ (of length $O_{\delta}(1)$, depending on the Lyapunov spectrum), one gets that $$hd_{g_{t'-O_{\delta}(1)}.u'.q'_1}(W^{uu}_{\text{loc}}(g_{t'-O_{\delta}(1)}.u'.q'_1), W^{uu}_{\text{loc}}(g_{t'-O_{\delta}(1)}.u.q_1))=\varepsilon.$$
In particular, by definition of $t$, one must have $t'\leq t+O_{\delta}(1)$.

Now for the inverse inequality, we follow~\cite[Claim~$12.7$]{eskin-mirzakhani}.
We assume that $t'$ is significantly smaller than $t$, in the sense of $t-t'\geq O_{\delta}(1)$.
We apply the avoidance Proposition~\ref{prop:A-avoidance} (c.f. also Lemma~\ref{lem:subspace-avoidance}) to $A(q_{1}',u',\ell,t')$ and get a point $q_{1}''\in W^{s}_{\text{loc}}(q_{1}')$ such that 
\begin{equation}\label{eq:avoidance-estimate-q''}
    C(\delta)^{-1}\cdot \mathfrak{A}(q_{1}',u',\ell,t')\leq A(q_{1}',u',\ell,t')F_{q_{1}'}(q_{1}'').
\end{equation}

In view of Proposition~\ref{prop:packing-all-data} we have
\begin{equation*}
    hd_{g_{t'}.u'.q_{1}'}\left(W^{uu}_{\text{loc}}(g_{t'}.u'.q_{1}), W^{uu}_{\text{loc}}(g_{t'}.q_{1}'')\right) \geq c_{1}''(\delta)\cdot \varepsilon.
\end{equation*}
By the factorization theorem~\ref{thm:factorization} we get
\begin{equation*}
    hd_{g_{t'}.u'.q_{1}'}\left(W^{uu}_{\text{loc}}(g_{t'}.u'.q_{1}'), W^{uu}_{\text{loc}}(g_{t'-t}.g_{t}.u.q_{1})\right) \ll \varepsilon\cdot C\cdot e^{-\beta\cdot (t-t')} +  e^{-\alpha\cdot \ell},
\end{equation*}
where $\beta$ depends on the Lyapunov spectrum of $(M,g_{t},\mu)$, and the factor of $e^{-\beta\cdot (t-t')}$ arises from flowing $g_{t'-t}$ from $g_{t}.u.q_{1}$.

By the triangle inequality we get
\begin{equation*}
    hd_{g_{t'}.u'.q'_{1}}\left(g_{t'}.u'.q'_{1}, g_{t'}.q'' \right) \gg \varepsilon \cdot \left(c_{1}''(\delta)- C\cdot e^{-\beta\cdot (t-t')}\right) - e^{-\alpha\cdot \ell}.
\end{equation*}
Combining with ~\eqref{eq:avoidance-estimate-q''} we get
\begin{equation*}
    e^{-\beta\cdot (t-t')} \gg C(\delta)^{-1}-c_{1}''(\delta).
\end{equation*}
Observing that $C(\delta)$ is bounded, and $c_{1}''(\delta)\to 0$ in $\delta$, we get a contradiction if $t-t'>O_{\delta}(1)$.
\end{proof}

Now we have that
\begin{equation}\label{eq:distortion-q'_2}
    \lvert \lambda_{1}(u.q_1,\tau_{1,(\varepsilon)}(q_{1},u,\ell)) - \lambda_{1}(u'.q'_1,\tau_{1,(\varepsilon)}(q'_{1},u',\ell)) \rvert \leq C_{4}''(\delta).
\end{equation}
This is equivalent to
\begin{equation}\label{eq:distortion-q'_3}
    \lvert \lambda_{1}(q_1,t_{1}) - \lambda_{1}(q'_1,t_{1}') \rvert \leq C_{4}''(\delta).
\end{equation}
\texttt{Obtaining matched pairs of points $u.q$, $u'.q'$}
Two $Y$ configurations $\mathcal{Y}=\left\{q,q_1,u.q_1,q_2,q_3\right\}$ and $\mathcal{Y}'=\left\{q',q'_1,u'.q'_1,q'_2,q'_3\right\}$ composed of generic points, forming dynamical quadrilaterals \emph{satisfying QNI} are called $\varepsilon-$\emph{matched} if
\begin{equation}\label{eq:dist-req-matching}
hd_{g_{t}.u.q_{1}}\left(W^{uu}_{\text{loc}}(g_{t}.u.q_1),W^{uu}_{\text{loc}}(g_{t'}.u'.q'_1)\right)<\varepsilon, \     hd_{g_{t'}.u'.q_{1}}\left(W^{uu}_{\text{loc}}(g_{t'}.u'.q'_1),W^{uu}_{\text{loc}}(g_{t}.u.q_1)\right)<\varepsilon.
\end{equation}
We are now trying to find $\varepsilon-$\emph{matched} $Y$-configurations.

This is done in a manner analogous to \cite[Lemmata $12.8, 12.9$]{eskin-mirzakhani}.

As we will work over $W_{\text{loc}}^{uu}(q_{1}), W_{\text{loc}}^{uu}(q_{1})$, we identify them locally with $\mathbb{R}^{n}$.

Recall the sets $Q_{\text{QNI}}(q_1,q_{1}') \subset W^{>1}_{\text{loc}}(q_1), Q_{\text{QNI}}(q_1',q_{1}) \subset W^{>1}_{\text{loc}}(q_1')$.

Each of these sets composed of points $u.q_{1}$, (respectively  $u'.q'_1$)  for which the projection under the central stable holonomy to $W^{>1}_{\text{loc}}(q_1')$ (respectively $W^{>1}_{\text{loc}}(q_1)$) is composed of a generic point which will form a dynamical quadrilateral.

Previously we have shown that one may find plenty of \emph{good} $Y$ configurations originating from either $q_{1}$ or $q_1'$, which satisfy one of the distance inequalities in~\eqref{eq:dist-req-matching}, but not necessarily having the matching point $u.q_1$ or $u'.q'_1$ being a generic point. 
The following lemma allows us to choose both good $Y$ configurations simultaneously.

\begin{lem}\label{lem:matching-lem}
In the same assumptions and notations as Propositions~\ref{prop:packing-all-data}, \ref{prop:match-times}, there exists a subset $\tilde{Q}(q_1,q'_1,\ell)\subset Q(q_1,q'_1,\ell)$ with $\frac{\mu^{uu}_{q_1}(\tilde{Q}(q_1,q'_1,\ell))}{\mu^{uu}_{q_1}(Q(q_1,q'_1,\ell))}>c_{7}(\delta)>0$ for some $c_{7}(\delta)$ such that for all $u.q\in \tilde{Q}(q_1,q'_1,\ell)$, its central-stable projection lands inside $Q_{\text{QNI}}(q'_1,q_1)$.  

\end{lem}

This Lemma depends heavily on \cite[Lemma~$12.9$]{eskin-mirzakhani}.

We start with the following proposition which is analogous to ~\cite[Proposition~$6.14$]{eskin-mirzakhani}).
This lemma uses factorization of normal forms together with distortion estimate in order to compare certain sizes of sets in $W^{uu}_{\text{loc}}$. 
\begin{prop}\label{prop:distortion-est}
    Under the previous notation, assume $t$ is large so that
    \begin{equation*}
        hd_{g_{t}.u.q_1}\left(W^{uu}_{\text{loc}}(g_{t}.u.q_{1}), W^{uu}_{\text{loc}}(g_{t}.q_{1}') \right) \leq \epsilon,
    \end{equation*}
    for some $u\in W^{uu}_{\text{loc}}(g_{t}.q_{1})$.
    Suppose that $x\in W^{uu}_{\text{loc}}(g_{t}.u.q_{1}) \cap B(g_{t}uq_{1}, O(1)) $.
    Define
    \begin{equation*}
        A_{t} = W^{uu}_{\text{loc}}(g_{t}.u.q_{1})\cap B(x,\epsilon_0), \ A'_{t} = W^{uu}_{\text{loc}}(g_{t}.q'_{1})\cap B(x,\epsilon_0),
    \end{equation*}
    for some $\epsilon_0<\epsilon$.
    
    Then there exists $\kappa>0$ such that
    \begin{equation*}
        \kappa^{-1} \cdot \left\lvert g_{-t}.A_{t}\right\rvert \leq \left\lvert g_{-t}. A'_{t}\right\rvert \leq \kappa\cdot \left\lvert g_{-t}.A_{t}\right\rvert
    \end{equation*}
    where the Haar measure is defined as the normalized measure induced from the Haar measure over $W^{uu}_{\text{loc}}(q_{1}) \cap \mathcal{B}_{0}(q_{1})$ and $W^{uu}_{\text{loc}}(q'_{1}) \cap \mathcal{B}_{0}(q'_{1})$ respectively.
\end{prop}

This proposition follows from a distortion estimate, as $W^{uu}_{\text{loc}}(u.q_{1}), W^{uu}_{\text{loc}}(q_{1}')$ are close for all times in $[0,t]$, together with the estimates in~\cite{ABV} regarding H\"{o}lder continuity of the Lyapunov subspaces.
The proof of Lemma~\ref{lem:matching-lem} follows verbatim from \cite[Lemma~$12.9$]{eskin-mirzakhani}, using our  partition $\mathcal{B}$, replacing their Claim~$12.6$ with our Proposition~\ref{prop:packing-all-data}.

\subsection*{Choice of $u.q_{1}, u'.q'_{1}, q_{2}, q'_{2}, q_{3}, q'_{3}$}
Pick some
\begin{itemize}
    \item $\ell\in D_{\text{good}}$
\end{itemize}
such that $\ell>\ell_{6}$ and
\begin{itemize}
    \item $q_{1}\in M'(\ell)\cap g_{\ell/2}.\mathfrak{P}$, which have non-zero intersection in view of our choice of $\delta$.
    \item $u.q_{1}\in Q(q_{1},q'_{1},\ell)$ such that $g_{-\ell/2}.u.q_{1}\in Q_{QNI}(q_{1/2},q'_{1/2})$, which have non-empty intersection in view of our choice of $\delta$.
    \item $u'.q'_{1} \in Q(q'_{1},q_1,\ell)$ such that $g_{-\ell/2}.u'.q'_{1}\in Q_{QNI}(q'_1,q_1)$, which exists by Lemma~\ref{lem:matching-lem}.
\end{itemize}

As a result we fix the following points:
\begin{itemize}
    \item  $q_{2}=g_{\tau_{1,(\varepsilon)}(q_{1},u,\ell)}.u.q_{1},$
    \item $q_{3} = g_{t_1}.q_{1},$
    \item $q'_{2} =g_{\tau'_{1,(\varepsilon)}(q'_{1},u',\ell)}.u'.q'_{1},$
    \item $q'_{3} = g_{t'_1}.q'_{1},$
\end{itemize}
where $\tau_{1,\varepsilon},\tau'_{1,(\varepsilon)},t_1,t'_1$ as defined before in~\eqref{eq:stopping-time-def},\eqref{eq:def-of-t_2}.

All those points belong to $M_{\text{rec}}$ by the choices we made, but the $Y$-configuration sides are not of the same length!

Clearly we have
$$g_{\tau_{1}}.u'.q'_{1} = g_{\tau_{1}-\tau'_{1}}. g_{\tau'_1}.u'.q'_{1} \in g_{\tau_{1}-\tau'_{1}}M_{\text{rec}}.$$
By Proposition~\ref{prop:match-times} we get that
$g_{\tau'_1 - \tau_1}.M_{\text{rec}} \subset g_{[0,C(\delta)]}.M_{\text{rec}}$ by~\eqref{eq:distortion-q'_2}
Similarly, we get that $g_{t_1}.q_{1} \in~ M_{\text{rec}}$, but \\ $g_{t_1}.q'_{1} = g_{t_1-t'_1}.g_{t'_1}.q'_{1} \in~ g_{[0,C(\delta)]}.M_{\text{rec}}$ by ~\eqref{eq:distortion-q'_3}.

Notice that by the definition of $M_{\text{rec}}$ via the recurrence density~\eqref{eq:rec-density}, for any point $x\in M_{\text{rec}}$, for the majority of times $t>T''(\delta)$, $g_{t}.x\in~M_{base}$.
In particular, one case find a number $s$ (bounded by a constant which depends on $\delta$) such that at the time $\tilde{\tau}=\tau_{1}+s$ both points $q_{2}=g_{\tilde\tau}.u.q_{1}$, $q'_2=g_{\tilde\tau}.u'.q'_1$ belong to $M_{base}$.
Similarly, there exists a number $s'$ such that at the time $\tilde t = t_{2}+s'$ both points $q_{3}=g_{\tilde t}.q_{1}$ and $q'_3=g_{\tilde t}.q'_{1}$ both belong to $M_{base}$.

\begin{rem}
With those choices of the endpoints, one may not guarantee exactly that the expansion rate and contraction rate will cancel each other, but one deduces that by changing the conditional measures between $q_{2}, q_{3}$ and $q_{2}', q_{3}'$ the measure change is reflected by the movement along the $E^{\lambda_1}$ direction composed with a small dilation (bounded as a function of the exact choices of $s, s'$).
\end{rem}

\subsection{Endgame}\label{sub:endgame}
\subsubsection*{Transfer of the conditional measures along the configurations}
\begin{prop}\label{prop:transfer-of-conditional-measures}
There exists some $\Delta=\Delta(\delta,\ell)$ such that $\Delta\to 0$ as $\ell \to \infty$ and 
\begin{equation*}
    d_{W}(\mathbf{f}_{1}(q_{2}),\mathbf{f}_{1}(q'_{2})) \leq \Delta.
\end{equation*}
\end{prop}

We recall the following standard fact
\begin{lem}
Suppose $T:M\to M$ preserves $\mu$ and also for almost every $x\in M$ we have
$$ \mathcal{B}_{0}(T.x)\cap T.\mathcal{B}_{0}(x) = T.\mathcal{B}_{0}(x)\cap \mathcal{B}_{0}(T.x). $$
Then
$$ \mathbf{f}_{1}(T.x) \propto T.\mathbf{f}_{1}(x) $$
in the sense that the restriction of both measures to the set $\mathcal{B}_{0}(T.x)\cap~ T.\mathcal{B}_{0}(x)$ where they both defined is the same up to normalization.
\end{lem}
For proof, see~\cite[Lemma~$4.2(iv)$]{einsiedler-lindenstrauss}.

In view of the normalized Wasserstein distance we actually have
\begin{equation*}
    d_{W}(\mathbf{f}_{1}(T.x),T_{\star}\mathbf{f}_{1}(x)) = 0.
\end{equation*}

Recall that we fixed a normal forms coordinates over our manifolds in a way such that the points were chosen in a set with bounded structure. 
Under this identification, for every $u\in E^{>1}(x), t,s\in\mathbb{R}$ we can define a map $S_{t,u,x,s}$ in normal forms such that
$S_{t,u,x,s}.x=g_{t}u.(g_{-s}.x) \in~W^{u}(g_{t-s}.x)$.

In view of the construction of the the conditional measures given in \S\ref{sec:conditional-measures} and Ledrappier's result~\cite[Lemmma~$3.8$]{Ledrappier-Sinai} we get the following Corollary.
\begin{cor}
We have that for almost every $x$, any $s,t\in \mathbb{R}$ and $u \in W^{uu}_{\text{loc}}(x)$
\begin{equation*}
     d_{W}(\mathbf{f}_{1}(S_{t,u,x,s}.x) , (S_{t,u,x,s})_{\star}\mathbf{f}_{1}(x)) = 0.
\end{equation*}
\end{cor}
By definition have that $\mathbf{f}_{1}(g_{-s}.x) \propto (g_{-s})_{\star}\mathbf{f}_{1}(x)$, based on the construction of $\mathbf{f}_{1}$.
Now by absolute continuity of the Lebesgue measure under translations, we get
\begin{equation*}
    d_{W}\left(\mathbf{f}_{1}(u.g_{-s}.x),\mathbf{f}_{1}(g_{-s}.x)\right) = 0 ,
\end{equation*}
and the Corollary follows.

Assume now that the eight points $q_{1},q'_{1},u.q_{1}, u'.q'_{1},,q_{2},q'_{2},q_{3},q'_{3}$ are all chosen to be good point as in the previous section.

In particular we have $q_{3}=g_{\tilde{t}}.q_{1}$,$q_{2}=g_{\tilde{\tau}}.u.q_{1}$.
Using normal forms coordinates over $W^{u}(q_{1})$ we can define a sub-resonant map $S_{\tilde{\tau},u,q_{1},\tilde{t}}$ such that $S_{\tilde{\tau},u,q_{1},\tilde{t}}(q_{2}) = q_{3} \in W^{u}(g_{\tilde{\tau}}.q_{1})$. Similarly we have a sub-resonant map $S_{\tilde{\tau},u',q'_{1},\tilde{t}}$ such that
$S_{\tilde{\tau},u',q'_{1},\tilde{t}}(q'_{2}) = q'_{3} \in W^{u}(g_{\tilde{\tau}}.q_{1})$

In view of the discussion above
\begin{equation}\label{eq:conditional-propto}
    \begin{split}
        \mathbf{f}_{1}(q_{2}) &\propto (S_{\tilde{\tau},u,q_{1},\tilde{t}})_{\star}\mathbf{f}_{1}(q_{3}), \\
        \mathbf{f}_{1}(q'_{2}) &\propto (S_{\tilde{\tau},u',q'_{1},\tilde{t}})_{\star}\mathbf{f}_{1}(q'_{3}).
    \end{split}
\end{equation}

\begin{lem}
We have that 
\begin{equation*}
    d_{W}(\mathbf{f}_{1}(q_{3}),\mathbf{f}_{1}(q'_3)) \ll \varepsilon.
\end{equation*}
\end{lem}
\begin{proof}
As $q_{3}$ and $q'_{3}$ are \emph{stably-related} and in an Oseledets' good set, we may recover $\Rc$ from the $\tilde{\Rc}$ which in turn is related to $\tilde{\Qc}$ as been done during the proof of the factorization theorem.
As $\tilde{\Qc}$ is the forward flag, it is smooth along \emph{stable leaves} (c.f. Ruelle's theorem~\ref{thm:smoothness-leaves}) and so we get that $dist(\tilde{\Qc}(q_{3}).\tilde{\Qc}(q'_{3})) = O(dist(q,q'_{3}))$, hence in particular we get $$\lvert \Qc(q_{3})-\Qc(q'_{3}) \rvert = O(dist(q_{3},q'_{3})).$$
As both $q_{3},q'_{3}$ are in the Oseledets' good set $M_{base}$, the translation maps $I^{0}_{q_{3}}, I^{0}_{q'_3}$ as defined in~\eqref{eq:I-def} give that
\begin{equation*}
    \lvert \Rc(q_{3})-\Rc(q'_{3}) \rvert = O(dist(q_{3},q'_{3})).
\end{equation*}
As $\mathbf{f}_{1}$ was normalized according to the cocycle $\Rc$, the result follows as long as $\tilde{t}\to\infty$ where $q_{3}=g_{\tilde{t}}.q_{1}$.
\end{proof}

\begin{cor}\label{cor:almost-invariance}
\begin{equation*}
    d_{W}(\mathbf{f}_{1}(q_2),\mathbf{f}_{1}(q'_2))\ll \varepsilon.
\end{equation*}
\end{cor}
We have by the triangle inequality
\begin{equation*}
        d_{W}(\mathbf{f}_{1}(q_2), \mathbf{f}_{1}(q'_2)) \leq d_{W}(\mathbf{f}_{1}(q_{2}), \mathbf{f}_{1}(q_{3})) + d_{W}(\mathbf{f}_{1}(q_{3}),\mathbf{f}_{1}(q'_{3})) + d_{W}(\mathbf{f}_{1}(q'_{3}),\mathbf{f}_{1}(q'_{2})),
\end{equation*}
which can be written as
\begin{equation*}
    d_{W}(\mathbf{f}_{1}(q_{2}), (S_{\tilde{\tau},u,q_{1},\tilde{t}})_{\star}.\mathbf{f}_{1}(q_{2})) + d_{W}(\mathbf{f}_{1}(q_{3}),\mathbf{f}_{1}(q'_{3})) + d_{W}((S_{\tilde{\tau},u',q'_{1},\tilde{t}})_{\star}.\mathbf{f}_{1}(q'_{2}),\mathbf{f}_{1}(q'_{2})).
\end{equation*}
and in view of~\eqref{eq:conditional-propto} and the normalized Wasserstein distance, we actually get
\begin{equation*}
    d_{W}(\mathbf{f}_{1}(q_2),\mathbf{f}_{1}(q'_2)) \leq d_{W}(\mathbf{f}_{1}(q_3),\mathbf{f}_{1}(q'_3)) \ll \varepsilon.
\end{equation*}

\subsubsection*{Generating limit points}
\texttt{Taking the limit at $\ell \to \infty$.}

Assume that $\ell_{k}\to~\infty$.
Without loss of generality, we may assume that $q_{i}(\ell_{k})\to~\tilde{q}_{i}$, $q'_{i}(\ell_{k})\to~\tilde{q}'_{i}$, $x(\ell_{k})\to \tilde{x}$, $z(\ell_{k})\to \tilde{z}$ as $M$ is compact, for $i=1,2,3$.
Moreover, due to the choices defined above, we may assume that the limit points all belong to the compact set $M_{base}$ consisting of Oseledets' good points. Moreover, we may assume those points are generic with respect to $\mu$.
In view of the factorization theorem, we see that the associated points $dist(u.\tilde{q}_{2}(\infty),W^{u}(\tilde{q}'_{2}))=~0 $
or equivalently we have 
$\tilde{q}_{2}\in W^{u}(\tilde{q}'_2)$.
But we still have $\varepsilon \ll_{\delta} d(\tilde{q}_{2},\tilde{q}'_{2}) \ll_{\delta} \varepsilon$.
Moreover, in view of Theorem~\ref{thm:factorization} and the explicit construction given there, we have that
$dist(x(\ell_{k}),z(\ell_{k})) \ll e^{-\alpha\cdot\ell_{k}}$, and also for the appropriately chosen stopping time $\tilde{\tau}$ we get 
$$ dist(q_{2}(\ell_{k}),g_{\tilde{\tau}}.z(\ell_{k})) \ll e^{-\alpha\cdot\ell_{k}},  $$
and in particular, as $W^{u}(g_{\tilde{\tau}}.z(\ell_{k})) = W^{u}(q'_{2}(\ell_{k})$, we have that
$$ dist(q_{2}(\ell_{k}),W^{u}(q'_{2}(\ell_{k})) \ll e^{-\alpha\cdot\ell_{k}}. $$
As we take $\ell_{k}\to \infty$, we get that $dist(\tilde{q}_{2},W^{u}(\tilde{q}'_{2})) =0$ so
$$ \tilde{q}_{2} \in W^{u}(\tilde{q}'_{2}). $$
In view of Corollary~\ref{cor:almost-invariance}, we have that
$$ d_{W}(\mathbf{f}_{1}(\tilde{q}_{2}), \mathbf{f}_{1}(\tilde{q}'_{2}))=0 $$

This concludes the proof of Proposition~\ref{prop:main}.

\subsection*{Obtaining extra invariance via computation in normal forms coordinates}

\texttt{Taking the limit as $\varepsilon \to 0$.}

Assume $\left\{\varepsilon_{m} \right\}_{m\in \mathbb{N}}$ is some sequence of positive numbers converging to $0$.
We apply Proposition~\ref{prop:main} for each $\varepsilon_{m}$ to get sets $\mathcal{M}'_{m}$.
Define
\begin{equation*}
    \mathcal{F} = \bigcap_{k=1}^{\infty}\bigcup_{m=k}^{\infty}\mathcal{M}'_{m},
\end{equation*}
namely $\mathcal{F} = \limsup_{m\to\infty}(\mathcal{M}'_{m})$.
Note that as each $\mathcal{M}'_{m}$ is of strictly positive measure $\mu(\mathcal{M}'_{m})\geq \delta_{0}$, we must have that $\mu(\mathcal{F})>0$.
Assume now $x\in \mathcal{F}$. So $x\in E_{m_i}$ for infinitely $m_{i}'s$, hence there exist points $y_{m_i}\in W^{u}(x)$ such that $\varepsilon_{m_i} \ll dist(x,y_{m_i}) \ll \varepsilon_{m_i}$ and $f_{2}(x)\propto f_{2}(y_{m_i})$.

We write $S_{m_i}$ to be the matching sub-resonant map moving $x$ to $y_{m_i}$ in the normal forms coordinates.
We note that in the normal forms coordinates, we may identify $y_{m_i}$ with a vector given by $\underline{y}_{m_i}$ given by $\underline{y}_{m_i} = S_{m_i}(\underline{w}_{m_i})$ for some vector $\underline{w}_{m_i}$.
Denote by $\text{Inv}(\mu)$ the \emph{invariance group} of the measure $\mu$. This is a closed subgroup of the sub-resonant group $G^{\chi}_{x}$. We aim to prove the following:
\begin{lem}\label{lem:Leb}
    The group $\text{Inv}(\mu)$ contains the group $\mathcal{U}_{\lambda_{1}}$ - the group of translations over $E^{\lambda_{1}}$, as a subgroup. 
\end{lem}
As $\mu$ was assumed to be generalized $u$-Gibbs state, $\text{Inv}(\mu)$ contains all the sub-resonant maps that keep the last coordinate fixed. We denote this (normal) subgroup of $G^{\chi}_{x}$ by $\text{Inv}_{\text{Gibbs}}$.
We will consider the quotient group $G^{\chi}_{x}/\text{Inv}_{\text{Gibbs}}$.
In view of the normal forms coordinates, $G^{\chi}_{x}/\text{Inv}_{\text{Gibbs}}$ is identified as a subgroup of the $1$-dimensional affine group and in our assumptions, as $\dim E^{\lambda_{1}}>0$, this is the full affine group.

Let $\langle S\rangle$ denote the closed subgroup of $\text{Inv}(\mu)$ generated by the maps $\left\{S_{m_i}\right\}$. 
We have that $f_{1}$ is invariant under elements from $\langle S\rangle$.
If $\langle S\rangle$ contains $\mathcal{U}_{\lambda_{1}}$, we are done.

\begin{obs}
The set of maps $\left\{S_{m_i}\right\}$ is \emph{uniformly bounded}.
\end{obs}
This follows as $dist(x,y_{m_i})\ll \varepsilon_{m_i} \ll 1$.

In particular, taking a limit of $S_{m_i}$, we see that that $\langle S\rangle$ contains some element $D(x)$ for which its projection to the affine group $G^{\chi}_{x}/\text{Inv}_{\text{Gibbs}}$ is a dilation of the form $x_{\lambda_{1}}\mapsto c\cdot x_{\lambda_{1}}$, for some non-trivial dilation.
In particular, after applying $D(x)^{-1}$, the image of the group $\langle S \rangle$ inside the quotient group $G^{\chi}_{x}/\text{Inv}_{\text{Gibbs}}$ contains the identity as a limit point.
Therefore, the projection of $\text{Inv}(\mu)$ into $G^{\chi}_{x}/\text{Inv}_{\text{Gibbs}}$ contains a one-parameter subgroup of the affine group. We denote this subgroup by $\mathcal{J}(x)$.
\begin{obs}
The one parameter subgroups of the affine subgroups are either the subgroup of translations, or a subgroup conjugated to the subgroup of dilations.
\end{obs}

The later case gives rise to multiplicative character from the $\mathcal{J}(x)$ to $\mathbb{R}$, which we denote by $e^{\beta(\mathcal{J}(x))}$.
Moreover we have that the character associated with $\mathcal{J}(g_{t}.x)$ equals to $e^{\beta(\mathcal{J}(x))\cdot t}$.
This is a measurable function over $M$.
Let $C\subset M$ be a subset of positive measure consisting of good points, such that $\beta$ is a bounded function.
For $\mu$-almost every $x\in C$, the ergodic theorem implies that $g_{-t}.x\in C$ for infinitely many times $t$.
In particular we have that $x=g_{t}.g_{-t}.x$ Hence the character associated to $g_{t}.g_{-t}.x$ equals to $e^{\beta(\mathcal{J}(g_{-t}x))\cdot t}.$
In particular we have $\beta(\mathcal{J}(g_{-t}x)) \cdot t = \beta(\mathcal{J}(x))$.
As $g_{-t}.x\in B$ for arbitrarily large $t$, we must conclude that $\beta(\mathcal{J}(x))\equiv 0$.
Therefore the transformation $\mathcal{J}(x)$ is the identity, contradicting the assumption that $\mathcal{J}(x)$ is a (non-trivial) one-parameter subgroup.

Therefore, the projection of $\langle S\rangle$ contains the group of translations. In particular, $\text{Inv}(\mu)$ contains the product group $\text{Inv}_{\text{Gibbs}}\times \mathcal{U}_{\lambda_{1}}$.
A Fubini argument shows then that $\mu$ is absolutely continuous (by disintegrating over $W^{>1}$ and concluding that the conditional measure along $W^{\lambda_{1}}$ is the Lebesgue measure) and hence a-posteriori invariant under the whole group $G^{\chi}_{x}$.
This concludes the proof of Lemma~\ref{lem:Leb}.

\section{Examples and an application to equidistribution}\label{sec:applications}
\subsection{Examples originating from homogeneous dynamics}
We begin with some basic examples from homogeneous dynamics which illustrates the QNI condition and and application of our main theorem. 

\begin{ex}\label{ex:asl2}
Consider $ASL_{2}(\mathbb{R}) = SL_{2}(\mathbb{R})\rtimes \mathbb{R}^{2}$ where we identify $ASL_{2}(\mathbb{R})$ as a subgroup of $SL_{3}(\mathbb{R})$ by the subgroup of matrices
\begin{equation*}
    ASL_{2} = \begin{pmatrix}
    a & b & x \\
    c & d & y \\
    0 & 0 & 1
    \end{pmatrix}.
\end{equation*}
Let $\Gamma\leq ASL_{2}(\mathbb{R})$ by a torsion free lattice and consider $M=ASL_{2}(\mathbb{R})/\Gamma$ (this is a torus bundle over the embedded hyperbolic surface).
Define $a_{t} = \begin{pmatrix} e^{t} & 0 & 0 \\
0 & e^{-t} & 0 \\
0 & 0 & 1
\end{pmatrix}$.
One easily calculate
\begin{equation*}
    a_{t}.g.a_{-t}=\begin{pmatrix} a & e^{2t}b & e^{t}x \\ e^{-2t}c & d & e^{-t}y \\ 0 & 0 & 1 \end{pmatrix}.
\end{equation*}
Therefore one can deduce the following Lyapunov splitting:
\begin{equation*}
    T_{p}M=E^{4}\oplus E^{3} \oplus E^{0} \oplus E^{2}\oplus E^{1}, \ E^{s}=E^{4}\oplus E^{3}, \ E^{u}=E^{2}\oplus E^{1}, E^{uu}=E^{1},
\end{equation*}
where $E^{4} = \log\begin{pmatrix} 1 & 0 & 0 \\ c & 1 & 0 \\ 0 & 0 & 1 \end{pmatrix}$, $E^{3} = \log\begin{pmatrix} 1 & 0 & 0 \\ 0 & 1 & y \\ 0 & 0 & 1 \end{pmatrix}$, $E^{2} = \log\begin{pmatrix} 1 & 0 & x \\ 0 & 1 & 0 \\ 0 & 0 & 1 \end{pmatrix}$ and $E^{1} = \log\begin{pmatrix} 1 & b & 0 \\ 0 & 1 & 0 \\ 0 & 0 & 1 \end{pmatrix}$.
Moreover, one may see that the system $(M,g_{t})$ equipped with any absolutely continuous measure over $E^{1}$ satisfy QNI as
\begin{equation*}
     \begin{pmatrix} 1 & 0 & 0 \\ c & 1 & y \\ 0 & 0 & 1 \end{pmatrix} \cdot \begin{pmatrix} 1 & b & 0 \\ 0 & 1 & 0 \\ 0 & 0 & 1 \end{pmatrix} \cdot \begin{pmatrix} 1 & 0 & 0 \\ -c & 1 & -y \\ 0 & 0 & 1 \end{pmatrix} = \begin{pmatrix} 1-bc & b & -by \\ -bc^2 & 1+bc & -bcy \\ 0 & 0 & 1 \end{pmatrix},
\end{equation*}
so one may get an $E^{2}$ component as long as the conditional measure along $E^{3}$ is non-trivial (this can be verified say by an entropy assumption, assuming $h_{\mu}(a_{-t})>2$).
By the measure rigidity theorem, this shows that  any $a_{t}$-invariant measure with absolutely continuous conditional measure over $E^{1}$ must be horospherical.
We remark here that such an entropy assumption would force the measure to be the Haar measure in view of Ratner's theorem, as because of the conditional measure along the $E^{2}$-direction, the measure cannot interact with any proper tube, but nevertheless the usage of the measure rigidity theorem suppresses the usage of the linearization technique in that case.
\end{ex}

The following example provides an example of HQNI system.
\begin{ex}\label{ex:borel-smale}
We will consider $G=\mathbb{R} \rtimes (N \times N)$ as a Lie group, where $N$ is the $3$-dimensional Heisenberg group.

We identify the Heisenberg group with the triplets $(x,y,z)$ such that $[x,y]=z$, and so we may identify $N\times N$ with the six-dimensional space which we write as $$v=\begin{pmatrix} x_{1} &x_{2} \\ y_{1} &y_{2} \\ z_{1} &z_{2} \end{pmatrix}.$$

We we will consider the Borel-Smale Anosov example (c.f.~\cite[Section $3$]{wilki}\cite{LAURET2003201}).
This is a linear automorphism $A$ of $N\times N/\Gamma$ which is Anosov, such that
\begin{equation*}
    \begin{split}
        A.x_{1} = \lambda^{a}.x_{1}, &A.x_{2}=\lambda^{-a}.x_{2}, \\
        A.y_{1}=\lambda^{b}.y_{1}, &A.y_{2}=\lambda^{-b}.y_{2}, \\
        A.z_{1}=\lambda^{a+b}.z_{1}, &A.z_{2}=\lambda^{-a-b}.z_{2},
    \end{split}
\end{equation*}
where
$$A.v=\begin{pmatrix} A &0 \\0 &1\end{pmatrix}.v =\begin{pmatrix} A(x_{1} &x_{2}) \\ A(y_{1} &y_{2}) \\ A(z_{1} &z_{2}) \end{pmatrix}.$$ 

Picking $a=3,b=-2$, one may see that we have the following Lyapunov spectrum
\begin{equation}\label{eq:Anosov-spec}
    3\log\lambda>2\log\lambda>\log\lambda>-\log\lambda>-2\log\lambda>-3\log\lambda.
\end{equation}
The associated unstable and stable manifolds are defined by $$\log W^{u} =~\text{span}\left\{x_{1},y_{2},z_{1} \right\}\text{ and } \log W^{s}=\text{span}\left\{x_{2},y_{1},z_{2} \right\}.$$
Moreover, every generalized $u$-Gibbs measure will have entropy greater or equal to $5\log\lambda$.

Consider the group $G=\mathbb{R}\rtimes (N\times N)$.
Let $M$ be the solvemanifold $G/\Gamma$ for some $\Gamma\leq G$ a lattice. 
As the Laypunov exponent associated to $x_{2},z_{1}$ only add up to $-4\log\lambda$, by an entropy calculation, at-least the two most negative Lyapunov spaces admit a non-zero conditionals.
If so, we see that one may found non-trivial  $y_{1}$ coordinate.
In particular, by taking commutators, one can see that we have QNI as one may recover the $z_1$-coordinate, with a quadratic estimate.
Using the theorem we conclude that such a measure must be SRB measure which here amounts to the Haar measure.
\end{ex}

We continue developing the previous example, giving rise to a non-algebraic HQNI flow by sufficiently small perturbations of the example above along the $z_{1},z_{2}$ directions.
\begin{ex}\label{ex:borel-smale-pert}
Consider the space $(N\times N)/\Gamma$ -  the product space of the Heisenberg groups from the previous example.
Notice that the map $A$ keep the subspace spanned by $z_{1},z_{2}$ fixed.
As a result, we may realize this system as a vector bundle over $\mathbb{T}^{2} \times \mathbb{T}^{2}$.
We will deform the map over the fiber in the $C^{\infty}$-topology.

We define a small perturbation map $\tilde{B}$ over the fiber as follows
\begin{equation}
    \tilde{B}(v) = \begin{pmatrix} (x_{1}, x_{2}).Id \\
    (y_{1}, y_{2}).Id \\
    (z_{1}, z_{2}).\tilde B\end{pmatrix},
\end{equation}
where $\tilde B$ is a small perturbation of the identity map acting on the fiber which is the so-called vertical torus defined by $(z_{1},z_{2})$ in the nilmanifold.
This leads to the action of
$$A\tilde{B}(v) = \begin{pmatrix} A.\begin{pmatrix} x_{1} &x_{2} \\ y_{1} &y_{2} \end{pmatrix}.Id \\
A(z_{1} , z_{2}).\tilde{B}\end{pmatrix}=\tilde{B}A.(v).$$
Using the stability of Anosov maps, the perturbed map over the fiber, $A.\tilde B$ is also Anosov.

As this is an Anosov map over the $2$-torus, using the entropy formula and Yomdin's theorem regarding continuity of the entropy, we see that the eigenvalues associated to $A.\tilde{B}$ are varying continuously with respect to eigenvalues of $A\mid_{z_1,z_2}$.
As a result, we have the following Lyapunov spectrum of $A\tilde{B}$ in analogy with~\eqref{eq:Anosov-spec}
\begin{equation}
    3\log\lambda>2\log\lambda>\log\mu_{\tilde{B}}>-\log\mu_{\tilde{B}}>-2\log\lambda>-3\log\lambda.
\end{equation}
assuming the perturbation is small enough.
Considering the commutator argument from the previous example, as it pertained only to horizontal torus, we still have a quadratic $z_1$ deviation, say of size greater or equal to $\varepsilon^{2}$ where $z_1$ is the direction of the non-perturbed system.
In the perturbed system, the vertical torus is foliated by new curves, amounting to $\tilde{B}$, which are $C^{1}$-close (say of distance at most $\Theta$) to $z_1,z_{2}$.
We may estimate the divergence of the foliations from the previous coordinates by
$\Theta\cdot\varepsilon^2$.
Hence as long as $\Theta<1/2$ we still salvage a quadratic deviation in the new least-positive Lyapunov direction of the unstable leaf.
\end{ex}

\begin{rem}
The previous procedure may be applied in the context of Example~\ref{ex:asl2} as well.
One may realized this space as a torus bundle over $SL_{2}/\Gamma$, where the induced action over the fiber is given by the diagonal flow.
Perturbing the fiber dynamics, still yields an Anosov flow. As we had an entropy gap, small perturbation would retain this property and keep the basis' foliation fixed.
As such, similar considerations applied and one may see that as the new foliations over the fiber are $C^{1}$-close to the axis, one may recover the $HQNI$ property as well.
\end{rem}

\begin{ex}\label{ex:SL3}
Let $G=SL_{3}(\mathbb{R})$ and $\Gamma\leq G$ a \emph{uniform lattice}. Consider the action over the homogeneous space $X=G/\Gamma$ of the split Cartan element
\begin{equation*}
    a_t=\exp\left(t\cdot \begin{pmatrix} 
2 & 0 & 0 \\
0 & 1 & 0 \\
0 & 0 & -3
\end{pmatrix} \right).
\end{equation*}
We clearly have
\begin{equation*}
\begin{split}
    a_{t}. \begin{pmatrix} 
1 & x & z \\
0 & 1 & y \\
0 & 0 & 1
\end{pmatrix}. a_{-t}  = \begin{pmatrix} 
1 & e^{t}\cdot x & e^{5t}\cdot z \\
0 & 1 & e^{4t}\cdot y \\
0 & 0 & 1
\end{pmatrix}
\end{split}.
\end{equation*}
Hence one can see that the unstable manifold of the flow can be identified with $N=\exp(\underline{N})$ where $\underline{N}=span\left\{\underline{x},\underline{y},\underline{z}\right\}$ with $\underline{x}=\left(\begin{smallmatrix} 0 & \star & 0\\ 0 & 0 & 0\\ 0& 0 &0 \end{smallmatrix}\right)$, $\underline{y}=\left(\begin{smallmatrix} 0 & 0 & 0\\ 0 & 0 & \star\\ 0& 0 &0 \end{smallmatrix}\right)$ and $\underline{z}=\left(\begin{smallmatrix} 0 & 0 & \star\\ 0 & 0 & 0\\ 0& 0 &0 \end{smallmatrix}\right)$.
Moreover, we have that $E^{1}=span\left\{\underline{z} \right\}$, $E^{2}=span\left\{\underline y \right\}$ and $E^{3} = span\left\{\underline z \right\}$.
Moreover, we see the following quantitative non-integrability property:
for $g$ and $g'=s.g$ we have
$Proj_{s}^{g \mapsto g'}(E^{\leq 2}) = Ad(s).(E^{1}+E^{2})$ which must have an $\underline x$ component as we can do the following computation in the Lie algebra level
\begin{equation*}
    \begin{split}
    \left[
        \begin{pmatrix}
        0 & 0 & 0\\
        s_{1} & 0 & 0 \\
        s_{2} & s_{3} & 0
        \end{pmatrix}, \begin{pmatrix}
        0 & 0 &  z\\
        0 & 0 &  y \\
        0 & 0 & 0
        \end{pmatrix}
        \right] = 
        \begin{pmatrix}
        -s_{2}\cdot z & -s_{3}\cdot z & 0 \\
        -s_{2} \cdot z & -s_{3} \cdot y & s_{1}\cdot z \\
        0 & 0 & s_{2}\cdot z +s_{3}\cdot y
        \end{pmatrix}
    \end{split}
\end{equation*}
So one may see that the projection of this element, over $\underline{N}$ is the nilpotent matrix
\begin{equation*}
    \begin{pmatrix}
    0 & -s_{3}\cdot z & 0 \\
    0 & 0 & s_{1}\cdot z \\
    0 & 0 & 0
    \end{pmatrix} = -s_{3}\cdot z \cdot \underline{x} + s_{1}\cdot z \underline{y}.
\end{equation*}
As $\underline{y} \in E^{1}+E^{2}$, we get that the deviation from $E^{1}+E^{2}$ in $\underline{N}$ is reflected by $-s_{3}\cdot z \cdot \underline{x}$.
We note that $s_{3}$ is comparable the distance between $g,g'$ (chosen in a generic way).
The conclusion of our theorem will say that any probability measure $\mu$ which is ergodic and invariant under $L$ where $L=\langle a_{t},b_{t},\exp(t\cdot\underline{y}),\exp(t\cdot\underline{z}) \rangle$ must be invariant under all of $N$.
Where $b_{t}$ is any other Cartan element which is not a power of $a_{t}$.

In that case it is well known (i.e.~\cite[Theorem~$1.11$]{katz}) that measures which are invariant under the horospherical flow associated to $a_{t}$ (namely SRB), must be the Haar measure, with a quantitative and effective equidistribution statement.
Nevertheless, our result gives a related result (the assumption about the $b_{t}$-invariance is used only as the construction of $z$ in the factorization theorem, might have a $b_{t}$ component involved in the definition of $x_1$, which we can remove if we assume $b_{t}$-invariance and in particular, we do not need to assume $b_{t}$-ergodicity say).
\end{ex}
A related situation can be demonstrated in the case of taking \\ $G=~SL_{2}(\mathbb{R}) \times~ SL_{2}(\mathbb{R})$ when one take quotient by a irreducible lattice $\Gamma\leq G$, when one can choose a Cartan element say of the form
\begin{equation*}
    a_{t}=\left(\exp\left(t\cdot \begin{pmatrix} 
2 &  0 \\
0 & -2 
\end{pmatrix} \right), \exp\left(t\cdot \begin{pmatrix} 
1 & 0 \\
0 & -1 
\end{pmatrix} \right)\right).
\end{equation*}
In that case, the expanding space is $2$-dimensional, with the fast subspace corresponding to the first upper unipotent subgroup, and the second expanding subspace corresponding to the second upper unipotent subgroup.

We may deduce a result in the spirit of a theorem of Kleinbock-Shi-Weiss~\cite{Kleinbock-Shi-Weiss} as a corollary - 
\begin{ex}
Let $G=PSL_{3}(\mathbb{R})$ and $\Gamma \leq G$ a \emph{uniform lattice}.
Define $a_t=\exp\left(t\cdot \begin{pmatrix} 
2 & 0 & 0 \\
0 & 1 & 0 \\
0 & 0 & -3
\end{pmatrix} \right)$, $b_t=\exp\left(t\cdot \begin{pmatrix} 
-3 & 0 & 0 \\
0 & 2 & 0 \\
0 & 0 & 1
\end{pmatrix} \right)$  and let  $u(x,y,z)= \begin{pmatrix}
    1       & x & y \\
    0       & 1 & z\\
    0       & 0 & 1 
\end{pmatrix}$.
Then $U=\langle u(x,y,z)\rangle$ is the unstable horospherical subgroup $G_{a}^{-}$, namely $$\langle u(x,y,z)\rangle=\left\{ g\in G \mid a_{t}ga_{-t} \to e \text{ as $t\to -\infty$} \right\}.$$
Let $\rho(y,z)dydz$ be an absolutely continuous finite density with respect to the Riemannian measure of the the embedded submanifold $W^{\geq 2} ~=~\exp\begin{pmatrix} 
0 & 0 & \underline{y} \\
0 & 0 & \underline{z} \\
0 & 0 & 0
\end{pmatrix}$.
Define the measure $\nu$ as any weak-$\star$ limit of $\frac{1}{T}\int_{t=0}^{T}b_{t}.\rho(y,z)dydz$.
Then \emph{for every $x\in G/\Gamma$}, for $\nu$ almost-every $h$, we have that
$$\frac{1}{T}\int_{t=0}^{T}f(a_{t}.h.x)dt \to \int_{G/\Gamma}fdm_{\text{Haar}}.$$
\end{ex}
The proof follows at once from Example~\ref{ex:SL3}. One should note that it is not always the case that the weak-$\star$ limits $\nu$ are actually non-zero measures.
A similar situation occurs in the case of $SL_{2}(\mathbb{R})\times SL_{2}(\mathbb{R})/\Gamma$.
Again those kind of results can be easily deduced by known techniques in homogeneous dynamics such as Ratner theorems or the work by Einsiedler-Katok-Lindenstrauss and its generalizations~\cite[Theorem~$1.3$]{el-tori}, but the proof presented here is simpler in nature and doesn't use those explicit techniques.

\subsection{Examples related to $3D$ Anosov diffeomorphisms}
Recall that a diffeomorphism of a manifold $M$ is called \emph{partially hyperbolic} if for every $x\in M$ there exists a splitting of the tangent bundle to $M$ at $x$, $T_{x}M$ to three sub-bundles $$T_{x}M=E^{s}(x)\oplus E^{c}(x)\oplus E^{u}(x)$$ such that the splitting is dominated, $E^{s}$ is uniformly contracted, $E^{u}$ is uniformly expanded and at-least one of them is non-trivial.

Consider a smooth partially hyperbolic diffeomorphism $f$ of the $3$-tours $\mathbb{T}^3$.
In that case, every sub-bundle $E^{s}, E^{c}, E^{u}$ is one-dimensional.
We further assume that $\lambda_{1}<1<\lambda_{2}<\lambda_{3}$.

The following is an upcoming result of Avila-Crovisier-Eskin-Potrie-Wilkinson-Zhang~\cite{ACEPWZ}:

\begin{thm}: If $f:\mathbb{T}^3\to \mathbb{T}^3$ is Anosov and partially hyperbolic, with a two dimensional unstable foliation, then either  the system $(\mathbb{T}^3,f,\mu)$ satisfies the QNI condition for any measure $\mu$ which is a generalized $u$-Gibbs state with respect to the dominated splitting $E^{c}\oplus E^{u}$, or the system $(\mathbb{T}^{3},f)$ is conjugate to a linear system $(\mathbb{T}^{3},\tilde{f})$  , where the conjugacy maps the strong unstable leaf to the strong linear foliation.
\end{thm}

\subsection{Equidistribution in HQNI systems}
We now turn to prove a version of ``spherical equidistribution'' as stated in Theorem~\ref{thm:equi}, motivated by an analogous result by Chaika-Eskin~\cite{chaika-eskin} in the moduli space settings.

We note this kind of statement is stronger than a construction of $u$-Gibbs measures given by Pesin and Sinai~\cite{PesinSinai1982} to construct $SRB$ measures as they are averaging absolutely-continuous densities along $W^{u}$ and not just $W^{uu}$.

The main example to keep in mind is the Borel-Smale one, Example~\ref{ex:borel-smale}.

Let $\phi:~M\to~\mathbb{R}$ be a Lipschitz test function.
As the Oseledets splitting was assumed to be dominated, we may extend the Lyapunov exponents continuously to the whole manifold.
We note that for any $h\in E^{j}\subset \oplus_{i>1}E^{\lambda_i}(x)\simeq \mathbb{R}^{n}$ we have the following cocycle (coming from dominated splitting)
\begin{equation}\label{eq:cocycle-def}
    g_{t}.h=e^{\lambda_{j}(x,t)}.h.
\end{equation}
Moreover, the various cocycles $\lambda_{j}(x,t)$ grow linearly in $t$.

\begin{prop}\label{prop:law-large-num}
Fix some $h\in E^{j}$, define 
\begin{equation*}
    f_{t}(u)=\phi(g_{t}.(u.x)) - \phi(h.g_{t}.u.x),
\end{equation*}
where both points $g_{t}.(u.x), h.g_{t}.u.x$ belong to $W^{>1}_{\text{loc}}(g_{t}.x)$ and we interpret the $h$-translation by means of the normal forms coordinates.
Then for almost every $u$ we have
\begin{equation*}
    \lim_{T\to\infty}\frac{1}{T}\int_{t=0}^{T}f_{t}(u)dt = 0.
\end{equation*}
\end{prop}
This proposition implies that for a generic $u$, the measures constructed in the Theorem are $h$-invariant. Using the proposition repeatedly for various $h\in E^{j}$, we get that the measure are generalized $u$-Gibbs measures.

We will need the following auxiliary Lemma.
\begin{lem}
There exists some $\kappa>0$ and $C\geq 1$ such that for any $t>s>0$ we have
\begin{equation}\label{eq:corr-bound}
   \left\lvert \int_{u\in[0,1]}f_{t}(u)f_{s}(u)du \right\rvert \leq C\cdot e^{-\kappa\lvert t-s\rvert}.
\end{equation}
\end{lem}
\begin{proof}

Let $r>0$ which will be specified later.
Consider intervals of the form $\mathfrak{I}(i)=\left[i-e^{-r},i+e^{-r} \right]$.
We divide $[0,1]$ (which we think about as part of $E^{j}$) into disjoint consecutive intervals of the form $\mathfrak{I}(i)$.
Note that as $\phi$ is bounded, $\text{length}(\mathfrak{I}(i))\leq~2e^{-r}$, the last part of $[0,1]$ not of the form $\mathfrak{I}(i)$ does not interfere with the estimate, as long as $r$ is small enough.
Now as $\phi$ is a Lipschitz function, we have that
\begin{align*}
    \left\lvert f_{s}(y+i) -f_{s}(i)\right\rvert \ll_{\phi} e^{\lambda_{j}(x,s)}\cdot \lvert y \rvert.
    \end{align*}
Hence for for any $\lvert y \rvert \leq e^{-r}$ we have that

$$\lvert f_{s}(y+u)-f_{s}(u)\rvert \ll_{\phi} e^{-r+\cdot \lambda_{j}(x,s)},$$ as
\begin{equation*}
\begin{split}
    \lvert f_{s}(y+u)-f_{s}(u)\rvert &= \lvert \phi(g_{s}.(y+u))-\phi(g_{s}.u)\rvert \\
    &= \lvert \phi(\left( e^{\lambda_{j}(x,s)}\cdot y\right)+g_{s}.u)-\phi(g_{s}.u) \rvert \\
    &\ll_{\phi} \left\lVert e^{\lambda_{j}(x,s)}\cdot y \right\rVert \\
    &\ll_{\phi} e^{-r+\lambda_{j}(x,s)},
\end{split}    
\end{equation*}
and $\phi$ is a Lipschitz function.
Therefore we may estimate the correlation over any interval $\mathfrak{I}(i)$ as
\begin{equation}\label{eq:mathfrakI(i)-corr}
    \left\lvert \frac{1}{\lvert \mathfrak{I}(i)\rvert}\int_{u\in \mathfrak{I}(i)}f_{t}(u)f_{s}(u)du\right\rvert = \frac{\left\lvert f_{s}(i)\right\rvert}{\lvert \mathfrak{I}(i)\rvert}\left\lvert\int_{u\in \mathfrak{I}(i)}f_{t}(u)du\right\rvert+O_{f}\left(e^{-\lvert t-s \rvert}\right).
\end{equation}
Using the definition of $f_{t}$ we have that
\begin{equation*}
    \frac{1}{\lvert \mathfrak{I}(i)\rvert}\int_{u\in \mathfrak{I}(i)}f_{t}(u)du = \frac{1}{\lvert \mathfrak{I}(i) \rvert}\int_{u\in \mathfrak{I}(i)}\phi(g_{t}.u.x)-\phi(g_{t}h_{e^{-\lambda_{j}(x,t)}}.u.x)du.
\end{equation*}
Hence one see that there is cancellation in the interval except a neighborhood of the boundary of $\mathfrak{I}(i)$ which is of measure $2e^{-\lambda_{j}(x,t)}$ 

So we may bounded the integral over $\mathfrak{I}(i)$ as
\begin{align*}
    \left\lvert\frac{1}{\lvert\mathfrak{I}(i)\rvert}\int_{u\in \mathfrak{I}(i)}f_{t}(u)du \right\rvert &\ll_{f} e^{-\lambda_{j}(x,t)} \cdot \lvert \mathfrak{I}(i)\rvert^{-1}  \\
    &= e^{-\lambda_{j}(x,t)+r}. 
\end{align*}
Using that in~\eqref{eq:mathfrakI(i)-corr} give
\begin{equation*}
    \left\rvert \int_{u\in \mathfrak{I}(i)}f_{t}(u)f_{s}(u)du \right\rvert \ll_{f} e^{-\lambda_{j}(x,t)+r} +e^{-r+\lambda_{j}(x,s)} .
\end{equation*}
Now approximating the integral over the whole interval by the localized integrals give
\begin{equation}\label{eq:estimate-corelation}
\begin{split}
    \int_{u\in [0,1]}f_{t}(u)f_{s}(u)du &= \sum_{\mathfrak{I}(i)}\int_{\mathfrak{I}(i)}f_{s}(u)f_{t}(u)du + O_{f}\left(\length(\mathfrak{I}(i))^{-1}\right) \\
    &\ll_{f} e^{-\lambda_{j}(x,t)+r} +e^{-r+\lambda_{j}(x,s)} + e^{-r}.
\end{split}
\end{equation}
Notice we have the following cocycle property of $\lambda_{j}$:
\begin{equation}\label{eq:cocycle-prop}
    \lambda_{j}(x,t)=\lambda_{j}(x,s)+\lambda_{j}(y,t-s),
\end{equation}
where $y=g_{s}.x$.
We may rewrite the first summand of ~\eqref{eq:estimate-corelation} as
\begin{equation*}
    e^{-\lambda_{j}(x,t)+r}=e^{-\lambda_{j}(x,s)+r-\lambda_{j}(y,t-s)}.
\end{equation*}
Taking $r$ so that
$r=\lambda_{j}(x,s)+\frac{1}{2}\lambda_{j}(y,t-s)$ yields decay rate of $e^{-\frac{1}{2}\lambda_{j}(y,t-s)}$ in the above estimate.
Define $m=\min_{z\in M}\lambda_{j}(z,1)>0$.
Repeatedly applying the cocycle property~\eqref{eq:cocycle-prop} gives
\begin{equation*}
    \lambda_{j}(y,t-s) \geq m\cdot \lvert t-s\rvert.
\end{equation*}
Therefor taking any $\kappa\leq \frac{m}{2}$ would do.

The proof now follows by using Fubini's theorem over $[0,1]^{m}$ and specializing to the $E^{j}$-axis.
\end{proof}
\begin{rem}
The only usage in the identification of $W^{>m}$ with $\mathbb{R}^{n}$ here comes from the splitting process and calculating the change between the various sub-intervals $\mathfrak{I}(i)$.
Using normal forms coordinates, and the maps $H_{x}^{-1}:TW^{>m} \to TW^{>m}$ one may do a similar calculation (with a more complicated boundary estimate, but this estimate would not have a significant impact over the proof).
\end{rem}

We quote here the following law of large numbers:
\begin{lem}[\cite{chaika-eskin} Lemma~$3.4$]
Suppose $f_{t}:[0,1]\to\mathbb{R}$ are bounded functions satisfying~\eqref{eq:corr-bound} and also $f_{t}$ are $2M$-Lipschitz for any $t$, then Proposition~\ref{prop:law-large-num} holds.
\end{lem}
We remark here that Eskin-Chaika assumed in addition that the integrals over $[0,1]$ of $f_{t}$ vanish. This is not used in any part of their argument.

We conclude that for almost every $u$, using the result of Chaika-Eskin, the measures induced from taking weak-$\star$ limits in~\eqref{eq:empherical-def} are absolutely continuous over the foliation given by $W^{\geq m}$.
Furthermore, by their definition as averaging along $g_{t}$, any weak-$\star$ limit will be $g_{t}$-invariant.
\begin{obs}
The set of $u$-Gibbs states is a weak-$\star$ closed convex subset of $g_{t}$-invariant probability measures over $M$.
\end{obs}
For proof see~\cite[\S11]{bonatti2006dynamics}.

In view of the HQNI property of the manifold, we have the following
\begin{obs}
Any ergodic component of the limiting measure $\eta$ measure achieved as above must satisfy QNI.
\end{obs}

Given such a limiting measure $\eta$, we may consider its ergodic decomposition $d\eta=\int d\eta^{\varepsilon}$. 
Using the measure classification result of Theorem~\ref{thm:measure-classification}, we have that that each ergodic component $\eta^{\varepsilon}$ must be SRB.
As we assume that there is only one SRB measure, $\nu$, we have that $\eta^{\epsilon}=\nu$ for all the ergodic components $\eta^{\epsilon}$.
Hence the theorem follow.

\begin{rem}
The above proof (and construction showing HQNI and the example about $ASL_{2}$) suggests that even in the case where the flow does not satisfy QNI, if for limiting measure $\eta$ has large enough entropy, using the affinity of entropy over the ergodic decomposition, one may show that at-least one of its ergodic components satisfy QNI (we note here that the set of ergodic $u$-Gibbs states is countable), hence such a measure $\eta$ must have a basin of an SRB measure inside its support, recovering a topological analogue.
It will be interesting to continue developing Margulis' inequality techniques as in~\cite[Proposition~$2.13$]{emm} to replace the linearization technique in the analysis of the various ergodic components of the limiting measure.
\end{rem}

\newpage
\appendix

\section{Technical construction of factorization}\label{app:factorization-details}
\subsection*{Basic facts and definitions}

We assume throughout that $V$ is a finite-dimensional vector bundle over $M$, may be measurable, which is smooth along stable leaves.
Moreover, we fix some $x\in M$ and $y\in~W^{s}(x)$.
We will assume that $x$ is generic, in light of~\cite[Theorem~$2.5$]{kalinin}, we may endow $T(W^{s}(x))=E^{s}(x)$ with \emph{normal forms coordinates}.
If $y\in~W^{s}(x)$ is sufficiently close to $x$ we can write $y=exp_{x}(\underline{y})$ for some vector $\underline{y} \in E^{s}(x)=T(W^{s}(x))$.
Using the normal forms coordinates, we may calculate $g_{\ell}.y$ by calculating the action of $g_{\ell}.\underline{y} \in E^{s}(g_{\ell}.x)$ by of sub-resonant polynomials.
Throughout the section we will write $g_{\ell}.y$ as the result of this computation.
We note that this coordinates immediately generalize to higher tangent spaces, $T^{k}(E^{s}(x))$, for all $k\geq 1$.

\begin{defn}\label{defn:op-factorization}
Assume $V$ is a finite dimensional vector bundle over $M$. 
Let $x\in M$ be given and $y\in W^{s}(x)$.
Assume that $T(y):~V(x)\to~V(y)$ is a family of linear operators. Moreover we assume that $x,y$ both belong to the same local trivialization of $V$.

We say that $T$ is \emph{factorizable} if for any $\beta>0$ there exists a vector bundle $V_{ext}$, a map $F_{x}:W^{s}(x)\to V_{ext}(x)$  

and a linear map \\ 
$A(x,\ell):~V_{ext}(x)\to~Hom(V(g_{\ell}.x),V(g_{\ell}.x))$,
$A=A(T)$, 
 such that 
\begin{equation}
    \left\lVert \xi \circ T(y)-A(x,\ell)F_{x}(y) \right\rVert_{V(g_{\ell}.x),op} \ll e^{-\beta\cdot\ell},
\end{equation}
provided that $g_{\ell}.x,g_{\ell}.y$ lie in the same local trivialization, so we can do the following identification of vector spaces $V(g_{\ell}.x)\simeq V(g_{\ell}.y)$ by means of the isomorphism $\xi$.
\end{defn}

\begin{thm}[Smooth holonomy map is factorizable]\label{thm:hol-factorization}
Assume that $V$ is a vector bundle over $M$ which has \emph{smooth} stable holonomy map $H_{V}(x,y):V(x)\to V(y)$, then $H_{V}$ is factorizable.
\end{thm}
\begin{proof}
As the holonomy is smooth, let $A(x,\ell)$ be the \emph{prolongation} of $H_{V}(g_{\ell}.x,\star)$, of some order $k$ which will be determined later, as a function with domain being $W^{s}(g_{\ell}.x)$.
Using Taylor expansion over $E^{s}(g_{\ell}.x)$ we may write
\begin{equation}
    H_{v}(g_{\ell}.x,z) = \sum_{0\leq \lvert \alpha \rvert \leq k}\partial^{\alpha}H_{V}(g_{\ell}.x,g_{\ell}.x)\cdot \log(z)^{\alpha}+E_{g_{\ell}.x,k}(z)
\end{equation}
where $\log(z)\in E^{s}(g_{\ell}.x)\leq TM(g_{\ell}.x)$ is the inverse image of $z$ under the exponential map at $g_{\ell}.x$ (given in terms of the normal forms coordinates) and $E_{g_{\ell}.x,k}(z)$ stands for the reminder term.
Using Lagrange's reminder formula, noticing that $M$ is compact, we may bound
\begin{equation*}
    \lVert E_{g_{\ell}.x,k}(z) \rVert \ll_{M} dist(g_{\ell}.x,z)^{k+1}.
\end{equation*}
Specializing to the case of $z=g_{\ell}.y$ one gets
\begin{equation*}
\begin{split}
    \lVert E_{g_{\ell}.x,k}(z) \rVert &\ll_{M} dist(g_{\ell}.x,g_{\ell}.y)^{k+1}\\
    &\ll_{M} e^{-(k+1)\lambda_{C}\cdot\ell},
\end{split}
\end{equation*}
for some $\lambda_{C}$ which depends on the Lyapunov spectrum of $M$.
Taking $k$ large enough so that 
\begin{equation*}
    k+1>\frac{\beta}{\lambda_{C}}
\end{equation*}
yields the result, where the map $F_{x}(y)$ amounts to calculating the values of the various $\log(y)^{\alpha}$, $V_{ext}=T^{\leq k}E^{s}$ - the $k$'th higher-order tangent space to $E^{s}$, and $A(x,\ell)$ matches the $k$'th prolongation of $H_{V}(g_{\ell}.x,\star)$ with $g_{\ell}.F_{x}(y)$.
\end{proof}
\begin{defn}
Let $V$ be a finite dimensional vector bundle over $M$, and endows $V$ with the Pesin inner-product.
Assume $S\leq V$ is a subspace.
Given some $\ell>0$ we say that a basis $\mathsf{B}(g_{\ell}.x)$ of $S(g_{\ell}.x)$ is called \emph{almost orthonormal} if for any $v\in \mathsf{B}(g_{\ell}.x)$ we have $$\lvert \lVert v \rVert -1\rvert\leq o(\ell)$$ and for any two distinct vectors $v,v' \in \mathsf{B}(g_{\ell}.x)$, $$\lvert\langle v,v' \rangle\rvert \leq o(\ell).$$ 
\end{defn}

\begin{defn}\label{defn:factorization}
A subspace $S\leq V$ is \emph{factorizable} if for any $\beta>0$ there is a bundle $V_{ext}$ and a \emph{measurable} map $F_{x}:W^{s}(x)\to V_{ext}(x)$, 
and a linear map
$$A(x,\ell):V_{ext}(x)\to~\oplus_{i=1}^{\dim S}V(g_{\ell}.x),$$ $A=A(S)$ such that
\begin{equation}
    \lVert \mathfrak{s} - A(x,\ell)F_{x}(y) \rVert_{\oplus_{i=1}^{\dim S}V(g_{\ell}.x)} \ll e^{-\beta\cdot\ell},
\end{equation}
for some almost orthonormal basis $\mathfrak{s}$ of $S(g_{\ell}.y)$, where we assume $g_{\ell}.x, g_{\ell}.y$ lie in the same local trivialization.

We will extend the definition of factorization given above to say that a subspace is factorizable if it is factorizable over a subset of arbitrarily large measure of $M$.
\end{defn}

\begin{cor}[Preliminary factorization]\label{thm:basic-factorization-theorem}
Assume that $V$ is a bundle admitting smooth stable holonomies which we denote $H_{V}(x,y)$.
Assume that $S\leq~V$ is a subspace which is preserved under stable holonomies.
Then $S$ is factorizable over a subset of arbitrarily large measure of $M$.
\end{cor}
\begin{proof}
Fix $G_{S}:M\to \oplus_{i=1}^{\dim S}V$ some measurable choice of orthonormal bases for $S$ according to the Pesin inner-product.
The map $G$ has the following equivarience property $g_{\ell}.G_{S}(x):=G_{S}(g_{\ell}.x)$.
On the level of subspaces we have $$H_{V}(g_{\ell}.x,g_{\ell}.y)S(g_{\ell}.x)=S(g_{\ell}.y)$$
hence on the level of bases we get that $H_{V}(g_{\ell}.x,g_{\ell}.y).G_{S}(g_{\ell}.x)$ is a basis for $S(g_{\ell}.y)$.
Moreover, by a theorem of of Araujo-Bufetov-Filip~\cite{filip}, 
\begin{equation}\label{eq:filip-holonomy-estimate}
    \lVert H_{V}(x,y)-Id \rVert \ll dist(x,y)^{\gamma},
\end{equation}
for some $\gamma$ depending on the Lyapunov spectrum of $M$ on a set of arbitrarily large measure.
Consider this set.
Hence the basis given by $H_{v}(g_{\ell}.x,g_{\ell}.y).G_{S}(g_{\ell}.x)$ of $S(g_{\ell}.y)$ is an almost orthonormal basis.
As $H_{V}$ is smooth map, it is factorizable by Theorem~\ref{thm:hol-factorization}.
Combining the above yields the result. We remark here that the dependence in $\beta$ (namely the error of the factorization) relates to the order of the Taylor expansion of $H_{v}$ been used, while the error of the almost orthonormal basis $\mathfrak{s}$ produced is related both to $\beta$ and to the estimate regarding $\gamma$ in~\eqref{eq:filip-holonomy-estimate}.
\end{proof}

\begin{cor}[Future factorization]\label{cor:future-factorization}
Given some fixed  $\alpha>0$, define $T_{\max}=\alpha\cdot \ell$. Using the previous notations, we may factorize $g_{t}.S(g_{\ell}.y)$ in the following manner - for all  $\beta,\ell, t\leq T_{\max}(\ell)$, there exists a bundle $V_{ext}$, a map  $F_{x}:W^{s}(x)\to V_{ext}(x)$ and a linear map $A(x,\ell):V_{ext}(x)\to~\oplus_{i=1}^{\dim S}V(g_{\ell}.x)$ such that

\begin{equation}
    \lVert g_{t}.\mathfrak{s}-g_{t}A(x,\ell)F_{x}(y) \rVert_{\oplus_{i=1}^{\dim S}V(g_{t+\ell}.y)} \ll e^{-\beta\cdot \ell}, 
\end{equation} 
for some ``almost orthonormal'' basis $\mathfrak{s}$ of $S(g_{\ell}.y)$.
\end{cor}
\begin{proof}
Apply Corollary~\ref{thm:basic-factorization-theorem} with some $\beta'>0$ to be specified later.  

Consider the vector $u=\mathfrak{s}-A(x,\ell)F_{x}(y) \in \oplus_{i=1}^{\dim S}V(g_{\ell}.y)$. 

By Oseledets theorem 
$$\lVert g_{t}.u \rVert \leq e^{\lambda_{M}\cdot t}\cdot \lVert u\rVert \ll_{v} e^{\lambda_{M}\cdot T_{\max}}\cdot e^{-\beta'\cdot \ell}$$
for all $0\leq t\leq T_{\max}$, for some $\lambda_{M}>0$.

Given $\beta>0$, one may take $\beta'>0$ large enough so that  
\begin{equation*}
    \lambda_{M}\cdot T_{max}-\beta'\cdot~\ell~\leq~-\beta\cdot\ell.
\end{equation*}
With that choice of $\beta'$, we recover the required factorization. 
\end{proof}

\subsection*{The $P^{-}$ operator}

Assume $V$ is a vector bundle over $M$, smooth along stables. For a set of full measure $\hat{M}\subset M$ we have a splitting of $V$ by Oseledets theorem as  $V=\oplus_{i} V^{\lambda_{i}}$ .
\begin{defn}\label{defn:attached-bundle}
The \emph{attached bundle} to $V$, $\hat{V}$, is the (measurable) bundle $\hat{V}=\oplus_{i} V^{\leq \lambda_{i}}/V^{<\lambda_{i}}$ .
In view of Oseledets theorem, for any biregular point $x\in \hat{M}$ we have the (measurable) \emph{translating map} $i_{x}:~V(x)\to \hat{V}(x)$.
Writing the Osceledets splitting as $V(x)=\oplus_{i} V^{\lambda_{i}}$, we can write each $v\in V(x)$ as $v=\sum_{i} v^{\lambda_{i}}$ with $v^{\lambda_{i}}\in V^{\lambda_{i}}$ for all $i$.
Hence we can we may send each subspace to $V^{\lambda_{i}}$ to its associated quotient, namely $V^{\lambda_{i}} \mapsto V^{\leq \lambda_i}/V^{<\lambda_i}$. In particular for each $v^{\lambda_{i}}$ we have $$i_{x}(v^{\lambda_i})=v^{\lambda_i}+V^{<\lambda_{i}} = v^{\lambda_{i}}+\oplus_{j<i}V^{\lambda_j}.$$
\end{defn}
As each component of $\hat{V}$ has a single Lyapunov exponent, the attached bundle $\hat{V}$ admits smooth stable holonomies (c.f. \cite[Proposition~$3.4$]{avila}, \cite[Proposition~$4.2$]{kalinin-sadovskaya-singe-lyapunov}) which we denote as  $H_{\hat{V}}(x,y):\hat{V}(x)\to~\hat{V}(y)$.

\begin{defn}\label{def:P-def} Let $V$ be a finite dimensional vector bundle over $M$ which is smooth along stables. For any two points $x,y\in M$ such that $y\in W^{s}(x)$ we define the measurable operator $P^{-}_{V}(x,y):V(x)\to~V(y)$ between two biregular points $x,y$ to be the composition
\begin{equation}\label{eq:P-def}
P^{-}(x,y) = i^{-1}_{y} \circ H_{\hat{V}}(x,y) \circ i_{x},  
\end{equation}
where $i_{x}:V(x)\to \hat{V}(x)$ and $i_{y}^{-1}:\hat{V}(y)\to V(y)$ are the translating maps given by Oseledets theorem as in Definition~\ref{defn:attached-bundle}.
\end{defn}
This operator is analogous to the operator defined in \cite[\S4.2]{eskin-mirzakhani}.

A fundamental result due to F. Ledrappier~\cite[Theorem~$1$]{Ledrappier} shows that $P^{-}_{V}$ preserves any measurable equivariant subspace $S\leq V$. 

\begin{thm}[General subspace factorization theorem]\label{thm:general-factorization}
Assume $V$ is a vector bundle smooth along stables.
Let $S\leq V$ be any measurable equivariant subspace.
Then $S$ is factorizable.
\end{thm}
The factorization can be achieved by considering the equation
\begin{equation}\label{eq:P-factorization}
    S(g_{\ell}.y)=P^{-}_{V}(g_{\ell}.x,g_{\ell}.y)S(g_{\ell}.x),
\end{equation}
Which follows from Ledrappier's theorem.
We now need to show that we may factorize the operator $P^{-}_{V}(g_{\ell}.x,g_{\ell}.y)$. 
Examining the definition of $P^{-}$ in~\eqref{eq:P-def}, we see that 
\begin{equation*}
    P^{-}_{V}(g_{\ell}.x,g_{\ell}.y) = i_{g_{\ell}.y}^{-1} \circ H_{\hat{V}}(x,y)\circ i_{g_{\ell}.x}.
\end{equation*}
The term $H_{\hat{V}}(x,y)\circ i_{g_{\ell}.x}$ is factorizable (as an operator with its image in the attached bundle $\hat{V}$, which has \emph{smooth holonomies}) by Corollary~\ref{thm:basic-factorization-theorem}, as $H_{\hat{V}}$ is smooth along stables, and $i$ preserves spaces of dynamical definition.

We consider $i^{-1}_{\star}$ as a vector inside a bundle in the following manner
$$ i^{-1}_{\star}\in Hom(\hat{V}(\star),V(\star)),$$
or alternatively, one may consider $i^{-1}$ as section $M\to Hom(\hat{V}(\star),V(\star))$, hence belong to the vector space of such sections.

Unfortunately, the bundle of sections, does not necessarily admit smooth holonomies. Nevertheless, it is smooth along stables (as $V$ is). The next subsection introduces a technique based on a theorem of Brown-Eskin-Filip-Rodriguez-Hertz to overcome this issue and factorize the section $i^{-1}_{\star}$.

\subsection*{Overcoming non-existence of holonomies}
We will need to discuss cases when the vector bundle $V$ is only smooth along stables and does not necessarily admit smooth holonomies.
We note the following theorem.
\begin{thm}[Brown-Eskin-Filip-Rodriguez-Hertz~\cite{brown_eskin_filip}, \S $A.4$, Corollary $A.4.6$]\label{thm:brown-eskin-filip} Let $V$ be a bundle which is smooth along stables, there exists an embedding $j_{\star}:~V(\star)\to~V'(\star)$ such that $V'$ admits smooth stable holonomies. Moreover the embedding $j_{\star}$ is $g_{t}$-equivarient.
\end{thm}

As part of the proof of the above theorem, using \emph{cocycle normal forms}, we have the following crucial observation:
\begin{obs}[Explicit construction of cocycle normal forms]\label{obs:j-inverse}
The map $j_{\star}$ is an explicit analytic map.
\end{obs}
Due to this nature of the map $j_{\star}$, one may recover factorizability of $S$, by applying the simple factorization theorem~\ref{thm:basic-factorization-theorem} for the image of the vector $\mathfrak{s}$, $j_{g_{\ell}.y}(\mathfrak{s})\in \oplus_{i=1}^{\dim S}V'(g_{\ell}.y)$, and inverting $j_{g_{\ell}.y}$ by calculating a polynomial approximation by applying the Lagrange-Good inversion formula for power series (c.f. \cite{lagrange-good}).

The following proposition finishes the proof of Theorem~\ref{thm:general-factorization}.
\begin{prop}\label{prop:i-factorization}
$i_{g_{\ell}.y}^{-1}:\hat{V}(g_{\ell}.y) \to V(g_{\ell}.y)$ is factorizable as a vector, namely consider the subspace $L(g_{\ell}.y)$ spanned by $i_{g_{\ell}.y}^{-1}$, then this space is factorizable.
\end{prop}
The bundle $Hom(\hat{V}(\star),V(\star))$ is smooth along stables, as the attached bundle $\hat{V}$ is.
Using Theorem~\ref{thm:brown-eskin-filip} of Brown-Eskin-Filip-Rodriguez-Hertz, we may consider the image of $i_{g_{\ell}.y}^{-1}$ inside a bundle which admits smooth stable holonomies $j_{g_{\ell}.y}(i^{-1}_{g_{\ell}.y}) \in~Hom(\hat{V}(g_{\ell}.y),V(g_{\ell}.y))'$.
We note the following equation
\begin{equation}\label{eq:hol-factorization}
    j_{g_{\ell}.y}(i^{-1}_{g_{\ell}.y}) = H(g_{\ell}.x,g_{\ell}.y).g_{\ell}.H(y,x)j_{y}(i^{-1}_{y}),
\end{equation}
where $H$ is the stable holonomy of the bundle  $Hom(\hat{V},V)'$.
We may define a map $G_{x}(y)$ to be \begin{equation}\label{eq:i-section-defn} G_{x}(y)=H(y,x)j_{y}(i^{-1}_{y}).\end{equation}
The map $G_{x}(y)$ can be seen a section from $W^{s}(x)$ to some bundle $Hom(\hat{V}(x),V(x))'$.
Using Theorem~\ref{thm:basic-factorization-theorem}, expanding the smooth holonomy $H(g_{\ell}.x,\star)$, we may factorize the subspace spanned by $j_{g_{\ell}.y}(i^{-1}_{g_{\ell}.y})$ in the form of the factorization of the holonomy $H$, acting on the section $G_{x}(y)$. 

Using Observation~\ref{obs:j-inverse}, we may factorize (the one-dimensional subspace defined by) $i^{-1}_{g_{\ell}.y}$, 
by using the Lagrange-Good formula in order to extract $i_{g_{\ell}.y}^{-1}$ out of the $j_{\star}$ embedding. This concludes the proof of Proposition~\ref{prop:i-factorization}.
Combining the factorizations of $i^{-1}_{g_{\ell}.y}$ from Proposition~\ref{prop:i-factorization} and $H_{\hat{V}}(g_{\ell}.x,g_{\ell}.y)\circ i_{g_{\ell}.x}$ which follow from Corollary~\ref{thm:basic-factorization-theorem} yields Theorem~\ref{thm:general-factorization}.

\begin{obs}\label{obs:i-almost-isomtry}
For any vector bundle $V$, the translating map defined by Oseledets' theorem, $i^{-1}_{x}:~\hat{V}(x)~\to~V(x)$ between the attached bundle and $V$ is ``almost-isometry'' in the following meaning: for every $v\in \hat{V}(x)$ and $\lvert t \rvert \gg 0$ we get
\begin{equation}
    \frac{\left\lVert g_{t}.i^{-1}_{x}(v) \right\rVert_{V(g_{t}.x)}}{\left\lVert g_{t}.v \right\rVert_{\hat{V}(g_{t}.x)}} \ll e^{\varepsilon \cdot \lvert t\rvert}.
\end{equation}
\end{obs}
This follows immediately from the definition of the quotients and Pesin norms used.

\begin{cor}[Approximation of $E^{u}(g_{\ell}.y)$]\label{cor:factor-subspace}
The subspace $E^{u}(g_{\ell}.y)$ is factorizable, for any generic point $g_{\ell}.y$.
\end{cor}
\begin{proof}
Consider the tangent bundle $V=T_{p}M$.
The subspace $E^{u}\leq~TM$ is measurable equivariant, hence by Theorem~\ref{thm:general-factorization} it is factorizable, by approximating 
$P^{-}_{V}(g_{\ell}.x,g_{\ell}.y)\mathsf{B}(g_{\ell}.x)$, where $\mathsf{B}(g_{\ell}.x)$ is any orthonormal basis of $E^{u}(g_{\ell}.x)$.
\end{proof}

\subsection*{Approximation of a Taylor polynomial}
Consider the unstable bundle $(M,E^{u})$.
Assume that $f:W^{u}(p)\to \mathbb{R}^{n}$ is some smooth function.
We will be interested in studying its Taylor expansion (over $W^{u}(p)$), for example in order to do the computation discussed in \S~\ref{subsec:approximate-stable}.
One may think of the Taylor coefficients of such a function as a element in some \emph{jet bundle}.

Ideally, we would to read of the coefficients from the jet bundle $J(E^{u}(p))$.
As the unstable bundle $(M,E^{u})$ is not smooth along stables but only Holder-continuous, we may not be able to use the previous construction.
Using the partially hyperbolic splitting we have an embedding of bundle $E^{u}\hookrightarrow TM$ and as a result we get an embedding of the corresponding jet bundles $J(E^{u}(p))\hookrightarrow J(M)$.

\begin{prop}
For any finite $k$, the sub-bundle of jets of order less or equal to $k$ $J^{\leq k}(M)$ of $J(M)$ is factorizable.
\end{prop}
\begin{proof}
The full jet bundle $J(M)$ is smooth along stables, as $TM$ is, so is $J^{\leq k}(M)$. Hence the result follows from Theorem~\ref{thm:general-factorization}, considering $S=V=J^{\leq k}(M)$.
\end{proof}
Once we have a vector in $V$ in hand, one may evaluate the jet in \emph{any given direction} of $TM$ in order to recover the needed derivatives.

\begin{cor}
Combining the factorizability of $E^{u}$ and $J(M)$, we may factorize the various Taylor coefficients of any smooth function defined over $W^{u}(g_{\ell}.y)$.
\end{cor}

In view of the formula for directional derivative, we have the following:
\begin{lem}\label{lem:Taylor-coeff-fact}
Fix some smooth $f:W^{u}(p)\to \mathbb{R}^{n}$. 
Assume that $v,v'\in~TM$ then
$\lVert D_{v}f(p)-D_{v'}f(p) \rVert \ll_{f}~ \lVert v-v' \rVert$,
where the dependence of $f$ is by some $C^{\star}$-norm, locally around $p$, where the order of the norm relates to the order of the differential.
\end{lem}

This has an immediate corollary showing we may actually recover the Taylor approximation, as we have a good approximation to both the coefficients and the actual subspace.

\begin{cor}[Factorization of Taylor polynomial]\label{cor:taylor-fact}
Let $TP_{f,R}$ denote the Taylor polynomial of $f:W^{u}(g_{\ell}.y)\to~\mathbb{R}^{n}$ of degree $R$ and $\widetilde{TP}_{f,R}$ denote the factorized Taylor polynomial, then
\begin{equation*}
    \lVert TP_{f,R}(z)-\widetilde{TP}_{f,R}(z) \rVert \ll e^{-\alpha\cdot\ell}\cdot \max\{1,\lVert z\rVert^R\}.
\end{equation*}
\end{cor}
As all the coefficients' differences are bounded by $e^{-\alpha\cdot\ell}$, the corollary follows at once.

\subsection*{Approximation of $W^{u}$ and the transfer function}
Consider $W^{u}$ as a function to some ambient Euclidean space $f:W^{u}\to \mathbb{R}^{n}$.
Then the Taylor polynomial for $W^{u}(q'_1)$ is factorizable due to Corollary~\ref{cor:taylor-fact}, where the degree of the polynomial is determined by the Lyapunov exponent with respect to the error estimate needed, as explain in \S~\ref{subsec:approximate-stable}.

Moreover, the Taylor polynomial for $W^{cs}(u.q_{1/2})$ is trivially factorizable (but in terms of an operator $A(q,u,\ell)$ depending on $u$ as well), as it depends only on $u.q$ ($W^{u}(q)$) and is independent from $q'$, given $u$.

We may now solve, approximately, the equation $W^{cs}(u.q_{1/2})\cap W^{u}(q'_{1/2})$ by applying the Lagrange-Good inversion formula~\cite{lagrange-good} to the factorized expression, in order to calculate approximation of $z$ in terms of $u.q_{1/2},q'_{1/2},\ell/2$ by polynomials up to the required precision.

\begin{cor}
The point $z$ is factorizable in the following sense -  we identify $z\in W^{u}(q'_1)$, by means of the exponential map, with a vector $z\in E^{u}(q'_1)\leq T_{q'_1}M$ dependent on $q,q',u,\ell$. For all $\beta>0$ there exists a vector bundle $V$, a map $F_{q}:W^{u}(q)\to~V_{ext}(q)$, a linear operator $A(q,u,\ell):~V_{ext}(q)\to~T_{q'_1}M$ such that
\begin{equation*}
\lVert z-A(q,u,\ell)F_{q}(q') \rVert_{T_{q'_1}M} \ll e^{-\beta\cdot\ell}.
\end{equation*}
\end{cor}

For the transfer function $T^{u}_{q'_{1/2}\to z}$, we may consider the backward flag as an element of the bundle $\oplus_{i}\left(\bigwedge^{\dim V^{\lambda_i}} TM\right)$.
Using Lemma~\ref{lem:Taylor-coeff-fact} again for $E^{u}$ and the jet bundle $J(\oplus_{i}\bigwedge^{\dim V^{\lambda_{i}}}TM)$, yields factorization of the Taylor polynomial for the transfer function.

Knowing the backward flag allows us to measure the distance between $E^{2}\oplus E^{1}$ from $E^{1}$ (in some generalized Grassmannian). 
As this distance function is smooth, we recover a factorized expression for the distance. Using the procedure described in~\S\ref{subsection:A-construction}, as both of our bundles of interest are of dimension $1$.

\begin{rem}
On a more general framework, one may calculate the ``distance'' $E^{1}\oplus E^{2}/E^{1}$ in terms of the normal forms coordinates of $W^{u}(q'_1/2)$.
This requires showing that the normal forms coordinates themselves are factorizable, and will be done in the upcoming work of Brown-Eskin-Filip-Rodriguez-Hertz~\cite{brown_eskin_filip}.
\end{rem}

\begin{cor}
The distance $hd_{z}(W^{uu}(q'_1), W^{uu}(q_1))$ is factorizable, namely there exists a bundle $V_{ext}$, a map $F_{q}:W^{s}(q)\to V_{ext}(q)$ and  an operator \\ $A(q,u,\ell):~V_{ext}(q_{1})\to~\mathbb{R}$ such that that
\begin{equation*}
    \lvert hd_{u.q_1}(W^{uu}(q'_1),W^{uu}(q_1)) - A(q,u,\ell)F_{q}(q') \rvert \ll e^{-\beta\cdot\ell}. 
\end{equation*}
\end{cor}

Using the corollary~\ref{cor:future-factorization} we may calculate the needed distance for any $t\leq \alpha\cdot \ell$.
We may identify the range of the operator $A(q,u,\ell,t)$ with $\Qc$ as defined in~\eqref{eq:Q-R-construction}.
This concludes the proof of the factorization theorem~\ref{thm:factorization} in \S~\ref{sec:factorization}.

\subsection*{Contraction property}
Fix a biregular point $q\in M$.
Recall we fixed some subspaces ($E^{u}$ and $\text{span}\left(T^{u}_{q'_{1/2}}(z)\right)$) which we factorized, along with their Taylor expansions associated to them.
Let $V_{total}$ be a finite dimensional vector bundle, which consists of the data needed for the calculation (which includes the Oseledets decomposition and the related copies of the jet bundles).
Define a vector bundle $V_{ext}$ to be a bundle consisting the required to factorize $P^{-}_{V_{total}}$. This bundle consists of the higher tangent bundles of $M$ together with the data required in order to calculate the translation map over the embedding of the extended vector bundle of $V_{total}$, given by the theorem of Brown-Eskin-Filip-Rodriguez-Hertz as in~\eqref{eq:hol-factorization},\eqref{eq:i-section-defn}.

Let $F_{q}:~W^{s}(q)~\to~V_{ext}(q)$ be the measurable map amounting to expanding the data needed to expand operator $P^{-}_{V_{total}}(g_{\ell}.q,\star)$ into a Taylor polynomial of the required degree and evaluating it at $\star=g_{\ell/2}.q'=q'_{1/2}$.
The operator $A$ then will amount to appending the operator $P^{-}_{V_{total}}$ with a measurable choice of a subspace for the required subspace we are interested in factorizing in the following sense
$A(q,u,\ell)F_{q}(q') = P^{-}_{V_{total}}(g_{\ell/2}.q,g_{\ell/2}.q').\mathbf{B}(g_{\ell/2}.q)+\text{ERROR}$, where the error is bounded exponentially in $\ell$ and $\mathbf{B}$ is some measurable choice of basis.

\begin{lem}
The following estimate holds:
$$
\lvert A(q,u,\ell,t)F_{q}(q) \rvert \ll e^{-\beta\cdot\ell}.$$
\end{lem}
This follows naturally by plugging $q'_1=q_1$ into the expression for the Hausdorff distance $hd_{u.q_1}(W^{uu}_{\text{loc}}(q'_1),W^{uu}_{\text{loc}}(q_1))$ give $0$.
Essentially, this estimate basically says that in the case that $q=q'$, the computation described above (without the errors arising from the truncations of the Taylor expansions) gives that \begin{equation*}
g_{-\ell/2}.z=W^{cs}(u.q_{1/2})\cap W^{u}(q_{1/2})=u.q_{1/2}.
\end{equation*}

Then we have the immediate corollary from the triangle inequality:
\begin{cor}
$$ \lvert hd_{g_{t}.u.q_1}(W^{uu}_{\text{loc}}(g_{t}.q'_1),W^{uu}_{\text{loc}}(g_{t}.q_1))- A(q,u,\ell,t)\left(F_{q}(q')-F_{q}(q)\right) \rvert \ll e^{-\alpha\cdot\ell}.$$
\end{cor}
This is analogous to the construction appearing in Eskin-Mirzakhani~\cite[Proposition~$6.11$, Proof of Lemma~$5.1$]{eskin-mirzakhani}.

\begin{lem}\label{lem:contracting-action}
The $g_{t}$ action over the subspace of $V_{ext}$ defined by $\text{ess-span}\left(\bigcup_{q'\in W^{s}_{\text{loc}}(q)}\left(F_{q}(q')-F_{q}(q)\right)\right)$ is contracting (when we think about those two vectors in a given local trivialization), hence the $g_{-t}$ action is expanding.
\end{lem}
\begin{proof}
By definition of $\sigma$, it consists of data in the higher tangent bundle and some Oseledets splittings of various bundles.
The image of $\sigma(q)$ in the higher tangent bundle amounts to $\overline{0}$, while the image of $\sigma(q')$ amounts to a vector in the stable part of space hence contracted (in an exponential manner) under the $g_{t}$-action, namely
\begin{equation}\label{eq:contraction-tangent-bundle}
\begin{split}
    \lVert F_{g_{t}.q}(g_{t}.q')\mid_{T^{\leq k}M}-F_{g_{t}.q}(g_{t}.q)\mid_{T^{\leq k}M} \rVert_{V_{ext}} &\ll dist(g_{t}.q',g_{t}.q) \\
    &\ll e^{-\lambda_{C}\cdot t}.
\end{split}
\end{equation}
For the Oseledets data, using the theorem by~\cite{filip}, we see that the Oseledets decomposition is dominated, in a Holder fashion, by the distance, namely 
\begin{equation}\label{eq:contraction-oseledets}
\begin{split}
    \lVert F_{g_{t}.q}(g_{t}.q')\mid_{\text{Oseledets}}-F_{g_{t}.q}(g_{t}.q)\mid_{\text{Oseledets}} \rVert_{V_{ext}} &\ll dist(g_{t}.q',g_{t}.q)^{\eta} \\
    &\ll e^{-\lambda_{C}\cdot t\cdot\eta},
\end{split}
\end{equation}
for some $\eta>0$ depending on the Oseledets decomposition.
\end{proof}

\section{Proofs of Lemma~\ref{lem:subspace-avoidance} and auxiliary lemmata }\label{app:avoiding-tech}
We start with an easy Lemma.
\begin{lem}\label{lem:badsubspace}
For any $\rho > 0$ there is a constant $c(\rho)>0$ with the following
property:  
Let $A: V \to W$ be a linear map between Euclidean spaces.
Then there exists a proper subspace
$\mathfrak{M} \subset V$ such that for any $v$ with
$\|v\|=1$ and $d(v,\mathfrak{M}) > \rho$, we have 
\begin{displaymath}
\|A\| \ge \|A v\| \ge c (\rho) \|A\|. 
\end{displaymath}
\end{lem}
\begin{proof}[Proof of Lemma~\ref{lem:badsubspace}]
The matrix $A^t A$ is symmetric, so it has a complete
orthogonal set of eigenspaces $W_1, \dots, W_m$ corresponding to
eigenvalues $\mu_1~>~\mu_2~>~\dots~>~\mu_m$. Let $\mathfrak{M} = W_1^{\perp}$. 
\end{proof}

\subsection*{Proof of Lemma~\ref{lem:subspace-avoidance}}
Let $\mu_{\star}^{s}$ denote the conditional measure along the partition $\mathcal{B}^{s}$.
We push this measure forward by the measurable map $F_{x}$ to a measure $(F_{x})_{\star}(\mu^{s}_{x})$ supported in $F_{x}(W^{s}(x))\subset V(x)$.
We define $\mathcal{L}(x)$ as the linear span of $\supp \left( F_{x}(\mu^{s}_{x})\right)\subset V(x)$.

We note that as $W^{s}(x)$ is a sub-bundle of $V(x)$ and the projection of $F_{x}(W^{s}(x))$ is the identity map, the map $F_{x}$ is actually injective, hence choosing an appropriate point in $F_{x}(W^{s}(x))$ in particular yields a point in $W^{s}(x)$.

\begin{rem}
As we assume that $\mu^{>1}_{x}$ is absolutely continuous with respect to the Riemannian volume on $W^{>1}_{loc}(x)$, using Ledrappier-Young entropy formula~\cite{led-young-2} we see that $h_{\mu}(g_{t})>0$.
We also have that $h_{\mu}(g_{t})=~h_{\mu}(g_{-t})$ as $g_{t}$ is invertible flow.
Hence $h_{\mu}(g_{-t})>0$ as well.
Using the Ledrappier-Young entropy formula again, we see that we must have that $\dim_{H}(\mu^{s}_{x})>0$ for $\mu$ almost every $x$.
In particular, we see that for $\mu$ almost every $x$, $\mathcal{L}(x)$ is \emph{not} the trivial subspace, namely $\dim \mathcal{L}(x)>0$, as $W^{s}(x)$ is embedded isometrically into $V(x)$.
\end{rem}

\begin{lem}\label{lem:5-3}
For $\mu$ almost every $x\in M$, for any $\epsilon>0$, the restriction of $(F_{x})_{\star}(\mu^{s}_{x})$ to the ball $B_{\epsilon}(\overline{0})\subset V(x)$ is not supported on a finite union of proper affine subspaces.
\end{lem}
\begin{proof}[Proof of Lemma~\ref{lem:5-3}]
Suppose not. Denote by $N(x)$ the minimal number of subspaces such that for some $\epsilon=\epsilon(x)>0$, the restriction of $(F_{x})_{\star}\mu^{s}_{x}$ to the ball $B_{\epsilon}(0)$ is supported on $N$ affine subspaces with non-zero $W^{s}$ components.

As the projection to $F_{x}(W^{s}(x))$ to $W^{s}(x)$ is the identity, and the action of $g_{-t}$ is expanding over $W^{s}(x)$, we have that the $g_{-t}$ action over $\mathcal{L}(x)$ is expanding in the following sense:
For every affine subspace $U$ of $V(x)$ with non-zero $W^{s}(x)$ component, we have that $g_{-t}.U \to \infty$.
Then $N(x)$ is invariant under $g_{-t}$, as for $g_{-t}.x$ we may pick $\epsilon(g_{-t}.x):=C(t)\cdot\epsilon(x)$ for some $C\geq 1$ which amounts to the expansion factor.
Hence $N(x)$ is constant almost-surely as in invariant function over an ergodic system.
Now we claim that $N(x)=0$ almost surely.
Assume not, so $N(x)=k>0$ almost surely.
Take $y\in W^{s}(x)$ such that $F_{x}(y)$ is inside some affine subspace which is contained inside the support of $F_{x}.\mu^{s}_{x}$.
Hence as $N(y)=N(x)$ we must have that there are $k$ subspaces which are supported by $(F_{y})_{\star}\mu^{s}_{y}$.
As $\mu^{s}_{x}=\mu^{s}_{y}$, and as $x$ and $y$ are stably-related we have $V(x)=V(y)$.
Now we see that the conditional measure is supported over at-least $k+1$ subspaces, which yields a contradiction.

Therefore we see that $(F_{x})_{\star}.\mu^{s}_{x}$ can only be supported over subspaces with $0$ as their $W^{s}$ component, which are also proper.
This contradicts the definition of $\mathcal{L}(x)$.
\end{proof}

\begin{lem}\label{lem:5-4} For every $\omega>0$ and $N>0$ there exists $\beta_{1}=\beta_{1}(\,N)>~0$, $\rho_{1}=\rho_{1}(\,N)>0$ and a compact subset $K_{\omega,N}\subset M$ of measure at-least $1-\omega$ such that for all $x\in K_{\omega,N}$ and any proper subspaces $\mathfrak{M}_{1},\ldots,\mathfrak{M}_{N}$ we have
\begin{equation*}
   (F_{x})_{\star}\mu^{s}_{x}\left(F_{x}.\mathcal{B}^{s}[x] \setminus\bigcup_{k=1}^{N}Nbhd(\mathfrak{M}_{k},\rho_{1})  \right) \geq \beta_{1}\cdot (F_{x})_{\star}\mu^{s}_{x}\left(F_{x}.\mathcal{B}^{s}[x] \right)
\end{equation*}
\end{lem}

\begin{proof}[Proof of Lemma~\ref{lem:5-4}]
By Lemma~\ref{lem:5-3}, for $\mu$ almost every $x$ there exists $\beta_{x}(N)>0$ and $\rho_{x}(N)>0$ such that for any subspaces $\mathfrak{M}_1(x),\ldots,\mathfrak{M}_{N}(x)$ we have
\begin{equation}\label{eq:avoid-local}
    (F_{x})_{\star}\mu^{s}_{x}\left(F_{x}.\mathcal{B}^{s}[x] \setminus \bigcup_{k=1}^{N}Nbhd(\mathfrak{M}_{k},\rho_{x})\right) \geq \beta_{x}\cdot (F_{x})_{\star}\mu^{s}_{x}(F_{x}.\mathcal{B}^{s}[x]).
\end{equation}
Taking the union over decreasing values of $\rho,\beta$ give that
\begin{equation*}
    \mu\left(\bigcup_{\rho>0, \beta>0}\left\{ x \in M \mid \text{\eqref{eq:avoid-local} holds for $x$} \right\}\right)=1.
\end{equation*}
Now for any $\omega$, choosing sufficiently small $\rho_{1},\beta_{1}$ yields the subset $K_{\omega,N}$ of $\bigcup_{\rho>0, \beta>0}\left\{ x \in M \mid \text{\eqref{eq:avoid-local} holds for $x$} \right\}$ of measure larger than $1-\omega$.
\end{proof}

\begin{lem}\label{lem:5-5}
For every $\omega>0$ and $\epsilon_{1}>0$ there exists $\beta=\beta(\omega,\epsilon_{1})>~0$ and $K_{\omega}\subset M$ a compact subset of measure larger than $1-\omega$ and $\rho=~\rho(\omega,\epsilon_{1})>0$ such that the following holds - suppose that for each $u\in \mathcal{B}$ we have an proper subspace $\mathfrak{M}_{u}(x)$ of $V(q)$.
Let
\begin{equation*}
\begin{split}
    J_{\text{good}}(x) = \left\{y\in F_{x}.\mathcal{B}^{s}(x) \bigg\vert \text{at least $(1-\epsilon_{1})$ fraction of $u$ in $\mathcal{B}$, } d_{0}(y,\mathfrak{M}_{u}(x))>\rho/2 \right\}.
\end{split}
\end{equation*}
Then for any $x\in K_{\omega}$ 
\begin{equation*}
    (F_{x})_{\star}\mu^{s}_{x}(J_{\text{good}}(x)) \geq \beta\cdot (F_{x})_{\star}\mu^{s}_{x}(F_{x}.\mathcal{B}^{s}[x]).
\end{equation*}
\end{lem}

\begin{proof}[Proof of Lemma~\ref{lem:5-5}]
Let $n=\dim V(q)$.
By considering determinants, it is easy to show that
for any $C > 0$ there exists a constant $c_n = c_n(C) > 0$ depending on $n$ and $C$ such
that for any $\omega > 0$ and any points $v_1, \ldots , v_n$ in a ball of radius $C$ with the property
that $\lVert v_1\rVert \geq \omega$ and for all $1 < i \leq n$, $v_i$
is not within $\omega$ of the subspace spanned by
$v_{1},\ldots , v_{i-1}$, then $v_1,\ldots,v_n$ are not within $c_{n}\cdot\omega^{n}$ of any $n-1$ dimensional subspace.
Let $k_{\max}\in \mathbb{N}$ be the smallest integer greater than $1+n/\epsilon_{1}$ and let $N=N(\epsilon_{1})=\binom{k_{\max}}{n-1}$.
Let $\beta_{1},\rho_1$ and $K_{\omega,N}$ as in Lemma~\ref{lem:5-4}.
Let $\beta=\beta(\omega,\epsilon_{1})=\beta_{1}(\omega,N(\epsilon_{1}))$, $\rho=\rho(n,\epsilon_{1})=c_{n}\cdot\rho_{1}^{n}$, $K_{\omega}(\epsilon_{1})=K_{\omega,N(\epsilon_{1})}$.
Define
\begin{equation*}
    J_{\text{bad}}(x)=\mathcal{B}^{s}[x]\setminus J_{\text{good}}(x).
\end{equation*}
We claim that $J_{\text{bad}}(x)$ is contained in the union of $\rho_{1}$-neighborhoods of at-most $N$ subspaces.
Suppose this is not true. Then, for $1 \leq k \leq k_{max}$ we can inductively pick points $v_1,\ldots , v_k \in J_{\text{bad}}(x)$, such that $v_{j}$ is not inside a $\rho_{1}$-neighborhood of any of the subspaces spanned by $v_{i_1},\ldots,v_{i_r}$ for $1\leq i_{1}< \cdots \leq i_{r} \leq j-1$.
Then, any $r$-tuple of the points $v_{i_1},\ldots, v_{i_r} $ is not contained within $\rho=c_{n}\rho_{1}$ of a single subspace.
As $v_{i}\in J_{\text{bad}}(x)$ there exists $U_{i}\subset\mathcal{B}$ with $\lvert U_{i} \rvert \geq \epsilon_{i}\lvert \mathcal{B}\rvert$ such that for all $u\in U_{i}$, $d_{0}(v_{i},\mathfrak{M}_{u})<\rho/2$.
We claim that for any $1\leq i_{1}< \cdots \leq i_{r} \leq j-1$ we have
\begin{equation*}
    U_{i_1}\cap\cdots\cap U_{i_r} = \emptyset.
\end{equation*}
Assume that $u\in U_{i_1}\cap\cdots\cap U_{i_r}$, then each of $v_{i_1},\ldots,v_{i_r}$ is within $\rho/2$ of the single subspace $\mathfrak{M}_{u}$, in contradiction to the choice of the $v_{i}$'s.
Now we calculate
\begin{equation*}
    \epsilon_{1}\cdot k_{\max}\cdot\lvert \mathcal{B}\rvert \leq \sum_{i=1}^{k_{\max}} \lvert U_{i} \rvert \leq n\cdot \lvert \bigcup_{i=1}^{k_{\max}} U_{i} \rvert \leq n\cdot\lvert \mathcal{B} \rvert,
\end{equation*}
Which is a contradiction to the choice of $k_{\max}$.

Now by applying Lemma~\ref{lem:5-4} we get
\begin{equation*}
   \begin{split}
       (F_{x})_{\star}\mu^{s}_{x}(J_{\text{good}}(x)) &\geq (F_{x})_{\star}\mu^{s}_{x}\left(F_{x}.\mathcal{B}^{s}[x] \setminus \bigcup_{k=1}^{N}Nbhd(\mathfrak{M}_{k},\rho_{1}) \right) \\
       &\geq \beta\cdot (F_{x})_{\star}\mu^{s}_{x}\left(F_{x}.\mathcal{B}^{s}[x] \right)
   \end{split}
\end{equation*}
\end{proof}

We are now ready for the proof of Lemma~\ref{lem:subspace-avoidance}.
\begin{proof}[Proof of Lemma~\ref{lem:subspace-avoidance}]
Define
\begin{equation*}
    M'_{\text{dense}}=\left\{x\in M' \mid \mu^{s}_{x}(M'\cap \mathcal{B}^{s}[x]) \geq (1-\delta^{1/2})\mu^{s}_{x}(\mathcal{B}^{s}[x]) \right\}.
\end{equation*}
Since $\mathcal{B}^{s}$ is a partition, we must have that $\mu(M'_{dense})\geq (1-\delta^{1/2})$, by Markov's inequality.
For $x\in M''$ we have that
\begin{equation*}
    (F_{x})_{\star}\mu^{s}_{x}(F_{x}.(M'\cap \mathcal{B}^{s}[x])) \geq (1-\delta^{1/2})(F_{x})_{\star}\mu^{s}_{x}(F_{x}.\mathcal{B}^{s}[x]) .
\end{equation*}
Let $\beta(\omega,\epsilon_1)$ as in Lemma~\ref{lem:5-5}.
Set 
\begin{equation*}
    c(\delta)=\delta+\inf\left\{(\omega^2+\epsilon_{1}^{2})^{1/2} \mid \beta(\omega,\epsilon_{1})\geq 8\cdot\delta^{1/2} \right\}.
\end{equation*}
Clearly we have that $c(\delta)\to 0$ as $\delta\to 0$.
By definition of $c(\delta)$, we may choose $\omega=\omega(\delta)<c(\delta)$ and $\epsilon_{1}=\epsilon_{1}(\delta)<c(\delta)$ such that $\beta(\omega,\epsilon_{1})\geq~8\cdot~\delta^{1/2}$.
By Lemma~\ref{lem:5-5}, for $x\in K_{\omega}$ we have
\begin{equation*}
    (F_{x})_{\star}\mu^{s}_{x}(J_{\text{good}}(x)) \geq \beta\cdot(F_{x})_{\star}\mu_{x}^{s}(F_{x}.\mathcal{B}^{s}[x]) \geq 8\cdot\delta^{1/2}\cdot(F_{x})_{\star}\mu_{x}^{s}(F_{x}.\mathcal{B}^{s}[x]).
\end{equation*}
Let $M''=M'_{\text{dense}}\cap K_{\omega}$. We have
$\mu(M'')\geq 1-\delta-\delta^{1/2}-~c(\delta)$, so $\mu(M'')\to 1$ as $\delta\to 0$.
Also if $q\in M''$ then we have that \\ $M''\cap~J_{\text{good}}(q)~\neq~\emptyset $, by measure considerations. Hence we may choose $q'=s.q\in \mathcal{B}^{s}[q]$ such that $F_{q}(s.q)\in J_{\text{good}}(q)$.
The upper bound for $dist(q,q')$ follows trivially by the bound over the diameter of each atom in the partition.
The lower bound follows from the fact that up to using the exponential map and its inverse (from $W^{s}(q)$ to the tangent space), we see that $\lVert s \rVert^{2} \gg dist(F_{q}(s.q),p)^{2} + dist(p,\overline{0}) \geq \rho^{2}(\delta)$, where $p$ stands for the orthogonal projection from $F_{q}(s.q)$ to $\mathfrak{M}_{u}(q)$.
\end{proof}

\section{QNI condition}\label{app:QNI-cond}
For this appendix we will assume $(M,g_{t},\mu)$ is an Anosov system. The definition can be easily modified to accommodate other situations.

Assume $x,x'\in M$ are two points which are \emph{stably related} meaning $x'\in W^{s}(x)$.
We will use the notation $u.x$ to denote points over the unstable leaf of $x, W^{u}(x)$. This notation is convenient and derived from homogeneous dynamics.
Given $u.x\in W^{u}(x)$, we define the \emph{center stable projection} of $u.x$ to $W^{u}(x')$ as the unique point $u'.x'$ such that $u'.x'\in W^{u}(x')\cap W^{cs}(x)$. 
This projection defines a continuous mapping $\proj^{cs}_{x,x'}$ between $W^{u}(x)$ and $W^{u}(x')$.
We remark here that Hasselblat-Wilkinson~\cite{Hasselblatt-Wilkinson} showed that the map is only \emph{Hölder continuous} in general.

\begin{defn}
A dynamical quadrilateral $\mathfrak{Q}$ is an ensemble of four points, $x,x',u.x, z \in M$ such that $x$ is a Pesin regular point and the following holds:
\begin{itemize}
    \item $x'\in W^{s}(x)$,
    \item $u.x \in W^{uu}_{\text{loc}}(x)$,
    \item $y \in W^{u}(x')\cap W^{cs}(u.x)$.
\end{itemize}
So the quadrilateral $\mathfrak{Q}$ is defined by a base point $x$, translation along $W^{uu}_{\text{loc}}(x)$ and translation along $W^{s}(x)$. As the point $y$ is defined by $x,x',u.x$ via means of $y=\proj^{cs}_{x,x'}(u.x)$, we will denote $\mathfrak{Q}=\mathfrak{Q}(x,x',u.x)$.
\end{defn}
Consider a dynamical quadrilateral $\mathfrak{Q}(x,x',u.x)$ for which $x'$ is a Pesin regular point as well.
Inside $W^{u}(x')$, we may consider the embedded disk $W^{>1}_{\text{loc}}(x')$.
The following notion of non-integrability means that this projection does not close on itself, in a quantitative manner.
We refer the reader to \S\ref{subsec:conditionals-construction} to the definition of the conditional measures used in the next definition.
\begin{defn}\label{defn:QNI}
We say that a dynamical quadrilateral $\mathfrak{Q}(x,x',u.x)$ formed of Pesin regular points $x\in M,x'\in W^{s}(x)$ and $u.x\in W^{>1}_{\text{loc}}(x)$  satisfies \emph{quantitative non-integrability} (abbreviated QNI from now on) of order $\alpha$ and constant $C$, for some fixed $\alpha>0$ and fixed $C>0$, if the following estimate holds:
\begin{equation}\label{eq:QNI-def1}
    dist_{W^{u}}(y,W^{uu}_{\text{loc}}(x'))  \geq C\cdot\min\left\{1,\dist(x,x')^{\alpha}, \dist(x,u.x)^\alpha\right\}.
\end{equation}
Estimate~\eqref{eq:QNI-def1} allows one to get an \emph{a-priori} bound over the growth of $dist_{W^{u}}(g_{t}.y,W^{uu}_{\text{loc}}(g_{t}.x'))$ using the minimal growth ensure by the Lyapunov exponent $\lambda_{1}$.
In particular, one sees that the \emph{stopping time} when this distance grows to size $O(1)$, is \emph{polynomial} in $dist(x,x')$. This polynomial depends will play a crucial role in the construction of the factorization operator in~\S\ref{sec:8-pts}.

We say that an Anosov system $(M,g_{t},\mu)$ satisfies \emph{quantitative non-integrability} if there exists some $\alpha>0$ and $C>0$ such that for a set of positive $\mu$-measure of points $\mathfrak{P}\subset M$, for every $x\in \mathfrak{P}$, there exists a subset $\mathfrak{S}(x)\subset W^{s}_{\text{loc}}(x)$ of positive $\mu^{s}_{x}$-density, \emph{not including $x$}, and for $x'\in~\mathfrak{S}(x)$, there exist subsets $Q_{QNI}(x,x')\subset W^{uu}_{\text{loc}}(x)$, $Q_{QNI}(x',x)\subset W^{uu}_{\text{loc}}(x')$ satisfying
\begin{equation*}
    \liminf_{r\to 0}\frac{\mathsf{m}_{x}^{uu}(Q_{QNI}(x,x') \cap B_{r}^{uu}(x))}{\mathsf{m}_{x}^{uu}(B^{uu}_{r}(x))}>0 ,
\end{equation*}
and
\begin{equation*}
    \liminf_{r\to 0}\frac{\mathsf{m}_{x'}^{uu}(Q_{QNI}(x',x) \cap B^{uu}_{r}(x'))}{\mathsf{m}_{x'}^{uu}(B^{uu}_{r}(x'))}>0,  
\end{equation*}
where $B^{uu}_{r}(\star)$ denotes a ball of radius $r$ in the embedded disk $W^{uu}_{\text{loc}}(\star)$, such that for all $u.x\in Q_{QNI}(x,x')$ and all $u'.x'\in Q_{QNI}(x',x)$ the dynamical quadrilaterals $\mathfrak{Q}(x,x',u.x)$, $\mathfrak{Q}(x',x,u'.x')$ satisfy QNI of order greater or equal to $\alpha$ and constants greater or equal to $C$.
\end{defn}
\begin{figure}[H]\label{fig:QNI}
    \centering
    \scalebox{0.8}{
\begin{tikzpicture}
\coordinate (q) at (4,0); 
\coordinate (q') at (4,4);
\coordinate (x) at (0,0); 
\coordinate (z) at (0.6,2.8);
\coordinate (uq) at (0,0);
\coordinate (uq') at (-0.2,4.2);

\draw [fill=gray, opacity=0.2] (5,-0.5) to [bend left] (-1.3,-0.5) to (-1.3,1) to [bend right] (5,1) to (5,-0.5) ;
\node[right] at (5,-0.5) {\smaller $W^{u}(x')$};

\draw [fill=gray, opacity=0.2] (5,3) to [bend left] (-1.3,3) to (-1.3,5) to [bend right] (5,5) to (5,3) ;
\node[right] at (5,3) {\smaller $W^{u}(x)$};

\filldraw (q) circle (2pt) node[below right]{$x$};
\filldraw (q') circle (2pt) node[above right]{$x'$};
\filldraw (x) circle (2pt) node[ left]{$u.x$};
\filldraw (z) circle (2pt) node[left,rotate=-15]{\smaller $\proj^{cs}_{x,x'}(u.x)$};

\draw[decoration={brace,mirror,raise=5pt},decorate]
  (q) -- node[right=6pt] {\smaller $L$} (q');
 
\draw[decoration={brace,mirror,raise=5pt},decorate]
  (z) -- node[right=5pt] {\smaller $L^{\alpha}$} (0.8,3.6);

\draw[very thick] (q) to [bend left] (x);
\draw[very thick] (q') to [bend left](uq');
\draw[dashed] (x) to (z);
\end{tikzpicture}
}
    \caption{Illustration of $\mathfrak{Q}$ and QNI}
\end{figure}
As $x,x'$ are stably-related, flowing forward shrinks the distance between them in an exponential manner, hence we may assume that $\mathfrak{S}(x)$ accumulates at $x$.

As the set satisfying QNI is of positive $\mu$ measure, and as the set of $\mu$-generic points is of full measure, we may assume that the set of points which satisfy QNI is formed of generic points.

The main difference between this definition and Definition~\ref{def:QNI2} is the requirement regarding the the trajectories of $x'$ under $f$ to return to the set $\mathcal{X}$, but the previous definition had to ensure the non-integrability ``in all scales''.
In the proof itself, we actually going to use the second definition. We will indicate where we are using that definition in due course.

\bibliographystyle{plain}
\bibliography{em-anosov}

\begin{thebibliography}{10}

\bibitem{Alvarez_Leguil_Obata_Santiago_2022}
Sébastien Alvarez, Martin Leguil, Davi Obata, and Bruno Santiago.
\newblock Rigidity of $\mathbf{\textit{u}}$-gibbs measures near conservative
  anosov diffeomorphisms on $\mathbb{T}^3$.
\newblock (arXiv:2208.00126), Jul 2022.
\newblock arXiv:2208.00126 [math].

\bibitem{ABV}
Jos\'{e}~F. Alves, Christian Bonatti, and Marcelo Viana.
\newblock S{RB} measures for partially hyperbolic systems whose central
  direction is mostly expanding.
\newblock {\em Invent. Math.}, 140(2):351--398, 2000.

\bibitem{filip}
Vitor Araujo, Alexander~I. Bufetov, and Simion Filip.
\newblock On holder-continuity of oseledets subspaces.
\newblock {\em Journal of the London Mathematical Society}, 93(1):194--218,
  2016.

\bibitem{ACEPWZ}
Artur Avila, Sylvain Crovisier, Alex Eskin, Rafael Potrie, Amie Wilkinson, and
  Zhiyuan Zhang.
\newblock {Private Communication}.

\bibitem{avila}
Artur {Avila}, Jimmy {Santamaria}, and Marcelo {Viana}.
\newblock {Holonomy invariance: rough regularity and applications to Lyapunov
  exponents.}
\newblock In {\em {Cocycles over partially hyperbolic maps}}, pages 13--74.
  Paris: Soci\'et\'e Math\'ematique de France (SMF), 2013.

\bibitem{pesin}
Luis Barreira and Yakov Pesin.
\newblock {\em Nonuniform hyperbolicity}, volume 115 of {\em Encyclopedia of
  Mathematics and its Applications}.
\newblock Cambridge University Press, Cambridge, 2007.
\newblock Dynamics of systems with nonzero Lyapunov exponents.

\bibitem{B-Q}
Yves Benoist and Jean-Fran\c{c}ois Quint.
\newblock Mesures stationnaires et ferm\'{e}s invariants des espaces
  homog\`enes.
\newblock {\em Ann. of Math. (2)}, 174(2):1111--1162, 2011.

\bibitem{bonatti2006dynamics}
Christian Bonatti, Lorenzo~J D{\'\i}az, and Marcelo Viana.
\newblock {\em Dynamics beyond uniform hyperbolicity: A global geometric and
  probabilistic perspective}, volume 102.
\newblock Springer Science \& Business Media, 2006.

\bibitem{bowen-markov}
Rufus Bowen.
\newblock Symbolic dynamics for hyperbolic flows.
\newblock {\em Amer. J. Math.}, 95:429--460, 1973.

\bibitem{brin1995ergodicity}
M~Brin.
\newblock Ergodicity of the geodesic flow.
\newblock {\em Appendix in: Lectures on Spaces of Nonpositive Curvature,
  Ballmann Birkhauser}, 1995.

\bibitem{brown_eskin_filip}
Aaron Brown, Alex Eskin, Simion Filip, and Federico Rodriguez~Hertz.
\newblock The general factorization method.
\newblock Preprint.

\bibitem{brown-hertz}
Aaron Brown and Federico Rodriguez~Hertz.
\newblock Measure rigidity for random dynamics on surfaces and related skew
  products.
\newblock {\em J. Amer. Math. Soc.}, 30(4):1055--1132, 2017.

\bibitem{Brown_Triestino_2019}
Aaron Brown and Michele Triestino.
\newblock {\em Entropy, Lyapunov exponents, and rigidity of group actions}.
\newblock 2019.

\bibitem{Brown_Malicet_Obata_Santiago_Triestino_Alvarez_Roldan_2019}
Aaron~W. Brown, Dominique Malicet, Davi Obata, Bruno Santiago, Michele
  Triestino, Sébastien Alvarez, and Mario Roldan.
\newblock {\em Entropy, Lyapunov exponents, and rigidity of group actions},
  volume~33 of {\em Ensaios Matemáticos}.
\newblock Brazilian Mathematical Society, Jul 2019.

\bibitem{wilki}
Keith Burns and Amie Wilkinson.
\newblock Dynamical coherence and center bunching.
\newblock {\em Discrete Contin. Dyn. Syst.}, 22(1-2):89--100, 2008.

\bibitem{Cantat_Dujardin_2021}
Serge Cantat and Romain Dujardin.
\newblock Random dynamics on real and complex projective surfaces.
\newblock {\em arXiv:2006.04394 [math]}, Mar 2021.
\newblock arXiv: 2006.04394.

\bibitem{chaika-eskin}
Jon Chaika and Alex Eskin.
\newblock Every flat surface is {B}irkhoff and {O}seledets generic in almost
  every direction.
\newblock {\em J. Mod. Dyn.}, 9:1--23, 2015.

\bibitem{clim-katok}
Vaughn Climenhaga and Anatole Katok.
\newblock Measure theory through dynamical eyes, 2012.

\bibitem{young-1-central}
William Cowieson and Lai-Sang Young.
\newblock S{RB} measures as zero-noise limits.
\newblock {\em Ergodic Theory Dynam. Systems}, 25(4):1115--1138, 2005.

\bibitem{Crovisier_Obata_Poletti_2022}
Sylvain Crovisier, Davi Obata, and Mauricio Poletti.
\newblock Uniqueness of $u$-gibbs measures for hyperbolic skew products on
  $\mathbb{T}^4$.
\newblock (arXiv:2209.09151), Sep 2022.
\newblock arXiv:2209.09151 [math].

\bibitem{dolgopyat-u-gibbs-notes}
Dmitry Dolgopyat.
\newblock Lectures on u-{G}ibbs states, 2001.

\bibitem{einsiedler-lindenstrauss}
M.~Einsiedler and E.~Lindenstrauss.
\newblock Diagonal actions on locally homogeneous spaces.
\newblock In {\em Homogeneous flows, moduli spaces and arithmetic}, volume~10
  of {\em Clay Math. Proc.}, pages 155--241. Amer. Math. Soc., Providence, RI,
  2010.

\bibitem{el-tori}
Manfred Einsiedler and Elon Lindenstrauss.
\newblock On measures invariant under tori on quotients of semisimple groups.
\newblock {\em Ann. of Math. (2)}, 181(3):993--1031, 2015.

\bibitem{einsiedlerbook}
Manfred Einsiedler and Thomas Ward.
\newblock {\em Ergodic theory with a view towards number theory}, volume 259 of
  {\em Graduate Texts in Mathematics}.
\newblock Springer-Verlag London, Ltd., London, 2011.

\bibitem{eskin_lindenstrauss-long}
Alex Eskin and Elon Lindenstrauss.
\newblock Random walks on locally homogeneous spaces.

\bibitem{eskin_lindenstrauss}
Alex Eskin and Elon Lindenstrauss.
\newblock Zariski dense random walks on homogeneous spaces.

\bibitem{eskin-mirzakhani}
Alex Eskin and Maryam Mirzakhani.
\newblock Invariant and stationary measures for the {${\rm SL}(2,\Bbb R)$}
  action on moduli space.
\newblock {\em Publ. Math. Inst. Hautes \'{E}tudes Sci.}, 127:95--324, 2018.

\bibitem{emm}
Alex Eskin, Maryam Mirzakhani, and Amir Mohammadi.
\newblock Isolation, equidistribution, and orbit closures for the sl(2, r)
  action on moduli space.
\newblock {\em Annals of Mathematics}, 182(2):673--721, 2015.

\bibitem{gogolev}
Andrey Gogolev, Itai Maimon, and Aleksey N.~Kolmogorov.
\newblock A numerical study of gibbs $u$-measures for partially hyperbolic
  diffeomorphisms on ${T}^3$.
\newblock {\em Experimental Mathematics}, 07 2017.

\bibitem{Hasselblatt-Wilkinson}
Boris Hasselblatt and Amie Wilkinson.
\newblock Prevalence of non-{L}ipschitz {A}nosov foliations.
\newblock {\em Ergodic Theory Dynam. Systems}, 19(3):643--656, 1999.

\bibitem{lagrange-good}
Josef {Hofbauer}.
\newblock {A short proof of the Lagrange-Good formula.}
\newblock {\em {Discrete Math.}}, 25:135--139, 1979.

\bibitem{kalinin-sadovskaya-singe-lyapunov}
Boris {Kalinin} and Victoria {Sadovskaya}.
\newblock {Cocycles with one exponent over partially hyperbolic systems.}
\newblock {\em {Geom. Dedicata}}, 167:167--188, 2013.

\bibitem{kalinin}
Boris Kalinin and Victoria Sadovskaya.
\newblock Normal forms for non-uniform contractions.
\newblock {\em J. Mod. Dyn.}, 11:341--368, 2017.

\bibitem{Katok1995}
Anatole Katok and Boris Hasselblatt.
\newblock {\em Introduction to the Modern Theory of Dynamical Systems}.
\newblock Cambridge University Press, April 1995.

\bibitem{katok2001smooth}
Anatole~B Katok, Rafael De~la Llave, Yakov Pesin, and Howard Weiss.
\newblock {\em Smooth Ergodic Theory and Its Applications: Proceedings of the
  AMS Summer Research Institute on Smooth Ergodic Theory and Its Applications,
  July 26-August 13, 1999, University of Washington, Seattle}, volume~69.
\newblock American Mathematical Soc., 2001.

\bibitem{katz}
Asaf Katz.
\newblock Quantitative disjointness of nilflows from horospherical flows.
\newblock {\em Journal d'Analyse Mathématique}.
\newblock to appear.

\bibitem{Kleinbock-Shi-Weiss}
Dmitry Kleinbock, Ronggang Shi, and Barak Weiss.
\newblock Pointwise equidistribution with an error rate and with respect to
  unbounded functions.
\newblock {\em Math. Ann.}, 367(1-2):857--879, 2017.

\bibitem{LAURET2003201}
Jorge Lauret.
\newblock Examples of anosov diffeomorphisms.
\newblock {\em Journal of Algebra}, 262(1):201 -- 209, 2003.

\bibitem{Ledrappier}
F.~{Ledrappier}.
\newblock {Positivity of the exponent for stationary sequences of matrices.}
\newblock {Lyapunov exponents, Proc. Workshop, Bremen/Ger. 1984, Lect. Notes
  Math. 1186, 56-73 (1986).}, 1986.

\bibitem{led-young}
F.~Ledrappier and L.-S. Young.
\newblock The metric entropy of diffeomorphisms. {I}. {C}haracterization of
  measures satisfying {P}esin's entropy formula.
\newblock {\em Ann. of Math. (2)}, 122(3):509--539, 1985.

\bibitem{led-young-2}
F.~Ledrappier and L.-S. Young.
\newblock The metric entropy of diffeomorphisms. {I}. {C}haracterization of
  measures satisfying {P}esin's entropy formula.
\newblock {\em Ann. of Math. (2)}, 122(3):509--539, 1985.

\bibitem{Ledrappier-Sinai}
Fran\c{c}ois Ledrappier.
\newblock Propri\'et\'es ergodiques des mesures de sina\"\i.
\newblock {\em Publications Math\'ematiques de l'IH\'ES}, 59:163--188, 1984.

\bibitem{Ledrappier_Strelcyn_1982}
Fran\c{c}ois Ledrappier and Jean-Marie Strelcyn.
\newblock A proof of the estimation from below in pesin’s entropy formula.
\newblock {\em Ergodic Theory and Dynamical Systems}, 2(2):203–219, Jun 1982.

\bibitem{Pesin_1976}
Ja~B. Pesin.
\newblock Families of invariant manifolds corresponding to nonzero
  characteristic exponents.
\newblock {\em Mathematics of the USSR-Izvestiya}, 10(6):1261, Dec 1976.

\bibitem{PesinSinai1982}
Ya.~B. Pesin and Ya.~G. Sinai.
\newblock Gibbs measures for partially hyperbolic attractors.
\newblock {\em Ergodic Theory and Dynamical Systems}, 2(3-4):417--438, December
  1982.

\bibitem{ratner-partitions}
M.~Ratner.
\newblock Markov partitions for {A}nosov flows on {$n$}-dimensional manifolds.
\newblock {\em Israel J. Math.}, 15:92--114, 1973.

\bibitem{ruelle1979ergodic}
David Ruelle.
\newblock Ergodic theory of differentiable dynamical systems.
\newblock {\em Publications Math{\'e}matiques de l'Institut des Hautes
  {\'E}tudes Scientifiques}, 50(1):27--58, 1979.

\bibitem{Sinai1}
Ja.~G. Sina\u{\i}.
\newblock Construction of {M}arkov partitionings.
\newblock {\em Funkcional. Anal. i Prilo\v{z}en.}, 2(3):70--80 (Loose errata),
  1968.

\bibitem{Sinai3}
Ja.~G. Sina\u{\i}.
\newblock Gibbs measures in ergodic theory.
\newblock {\em Uspehi Mat. Nauk}, 27(4(166)):21--64, 1972.

\bibitem{wright}
Alex Wright.
\newblock Translation surfaces and their orbit closures: An introduction for a
  broad audience, 2014.

\end{thebibliography}

\end{document}